\documentclass[reqno]{amsart}
\usepackage{amssymb,amsmath,amsthm}
\usepackage{color}
\usepackage{enumerate}
\usepackage{fullpage}
\usepackage{url}
\usepackage[Symbol]{upgreek}
\usepackage[unicode,pdfusetitle]{hyperref}
\usepackage[mathscr]{euscript}
\usepackage[pagewise]{lineno}
\usepackage{tikz-cd}
\usepackage{mathtools}

\usepackage{setspace}
\newtheorem{theorem}{Theorem}
\newtheorem*{theorem*}{Theorem}

\newtheorem*{question*}{Question}

\newtheorem*{conjecture*}{Conjecture}

\newtheorem*{convention*}{Convention}
\newtheorem{assumption}[theorem]{Assumption}
\newtheorem*{assumption*}{Assumption}
\newtheorem{corollary}[theorem]{Corollary}
\newtheorem*{corollary*}{Corollary}
\newtheorem{proposition}[theorem]{Proposition}
\newtheorem*{proposition*}{Proposition}
\newtheorem{lemma}[theorem]{Lemma}
\newtheorem*{lemma*}{Lemma}
\newtheorem{fact}[theorem]{Fact}
\newtheorem*{fact*}{Fact}

\newtheorem*{claim*}{Claim}

\newcommand{\thistheoremname}{}
\newtheorem*{genericthm*}{\thistheoremname}
\newenvironment{namedthm*}[1]
{\renewcommand{\thistheoremname}{#1}%
\begin{genericthm*}}
{\end{genericthm*}}

\theoremstyle{definition}

\newtheorem*{example*}{Example}
\newtheorem{remark}[theorem]{Remark}
\newtheorem*{remark*}{Remark}
\newtheorem{definition}[theorem]{Definition}
\newtheorem*{definition*}{Definition}

\numberwithin{theorem}{section}
\numberwithin{claim}{section}
\numberwithin{equation}{section}

\DeclareMathOperator{\an}{an}
\DeclareMathOperator{\anexp}{an,exp}

\DeclareMathOperator{\dcl}{dcl}

\DeclareMathOperator{\e}{e}

\DeclareMathOperator{\Li}{Li}

\DeclareMathOperator{\RCF}{RCF}
\DeclareMathOperator{\re}{re}
\DeclareMathOperator{\res}{res}
\DeclareMathOperator{\rexp}{rexp}

\DeclareMathOperator{\rk}{rk}

\DeclareMathOperator{\Th}{Th}

\newcommand{\N}{\mathbb{N}}

\newcommand{\Q}{\mathbb{Q}}
\newcommand{\R}{\mathbb{R}}
\newcommand{\T}{\mathbb{T}}

\newcommand{\cC}{\mathcal C}

\newcommand{\cH}{\mathcal H}
\newcommand{\cL}{\mathcal L}

\newcommand{\cO}{\mathcal O}

\newcommand{\cR}{\mathcal R}

\newcommand{\upl}{\uplambda}

\newcommand{\upo}{\upomega}

\newcommand{\dclL}{\dcl_\cL}

\newcommand{\inv}{^{-1}}

\newcommand{\logeq}{\mathrel{\vcentcolon\Longleftrightarrow}}
\newcommand{\Ld}{\cL^\der}
\newcommand{\LdO}{\cL^{\cO,\der}}
\newcommand{\LO}{\cL^{\cO}}
\newcommand{\overbar}[1]{\mkern 2mu\overline{\mkern-2mu#1\mkern-2mu}\mkern 2mu}
\newcommand{\rkL}{\rk_{\mathcal{L}}}

\newcommand{\TdO}{T^{\cO,\der}}
\newcommand{\TO}{T^{\cO}}

\renewcommand{\preceq}{\preccurlyeq}
\renewcommand{\succeq}{\succcurlyeq}
\renewcommand{\geq}{\geqslant}
\renewcommand{\leq}{\leqslant}
\renewcommand{\epsilon}{\varepsilon}
\renewcommand{\k}{\boldsymbol{k}}

\renewcommand{\d}{\operatorname{d}}

\DeclareFontFamily{OMS}{smallo}{}
\DeclareFontShape{OMS}{smallo}{m}{n}{<->s*[.65]cmsy10}{}
\DeclareSymbolFont{smallo@m}{OMS}{smallo}{m}{n}
\DeclareMathSymbol{\smallo}{\mathord}{smallo@m}{79}

\DeclareFontFamily{U}{fsy}{}
\DeclareFontShape{U}{fsy}{m}{n}{<->s*[.9]psyr}{}
\DeclareSymbolFont{der@m}{U}{fsy}{m}{n}
\DeclareMathSymbol{\der}{\mathord}{der@m}{182}
\DeclareSymbolFont{der@m}{U}{fsy}{m}{n}
\DeclareMathSymbol{\derdelta}{\mathord}{der@m}{100}

\author{Elliot Kaplan}
\email{kaplae2@mcmaster.ca}
\title{Liouville closed $H_T$-fields.}
\address{Department of Mathematics and Statistics, McMaster University, 1280 Main Street West, Hamilton, Ontario, L8S 4K1, Canada}
\date{\today}
\begin{document}
\maketitle

\begin{abstract}
Let $T$ be an o-minimal theory extending the theory of real closed ordered fields. An \emph{$H_T$-field} is a model $K$ of $T$ equipped with a $T$-derivation $\der$ such that the underlying ordered differential field of $(K,\der)$ is an $H$-field. We study $H_T$-fields and their extensions. Our main result is that if $T$ is power bounded, then every $H_T$-field $K$ has either exactly one or exactly two minimal Liouville closed $H_T$-field extensions up to $K$-isomorphism. The assumption of power boundedness can be relaxed to allow for certain exponential cases, such as $T = \Th(\R_{\anexp})$.
\end{abstract}

\setcounter{tocdepth}{1}
\tableofcontents

\section*{Introduction}
\noindent
In~\cite{AD02}, Aschenbrenner and van den Dries introduced the class of $H$-fields. This class consists of Hardy fields containing $\R$ as well as various ordered differential fields of formal power series, most notably the field $\T$ of \emph{logarithmic-exponential transseries}. In~\cite{ADH17}, Aschenbrenner, van den Dries, and van der Hoeven showed that the complete theory of $\T$ is essentially the model companion of the theory of $H$-fields.

\medskip\noindent
In this article, we consider $H$-fields equipped with additional o-minimal structure. Let $T$ be an o-minimal theory which extends the theory $\RCF$ of real closed ordered fields, and let $K$ be a model of $T$. Let $\der$ be a $T$-derivation on $K$ as defined in~\cite{FK19}, and let $C\coloneqq \ker(\der)$ be the constant field of $K$ (we recall the definition of a \emph{$T$-derivation} in Subsection~\ref{subsec:Tderiv} below). We say that $(K,\der)$ is an \textbf{$H_T$-field} if the following two conditions hold:
\begin{enumerate}
\item[(H1)] $\der f>0$ for all $f \in K$ with $f>C$, and
\item[(H2)] $\cO = C+ \smallo$, where $\cO$ is the convex hull of $C$ in $K$ and $\smallo$ is the unique maximal ideal of $\cO$.
\end{enumerate}
Axioms (H1) and (H2) are the axioms for $H$-fields, given in~\cite{AD02}, so an $H_T$-field is just a model of $T$ equipped with a $T$-derivation making it an $H$-field. In fact, the class of $H_{\RCF}$-fields coincides with the class of real closed $H$-fields.

\medskip\noindent
Our long-term goal is to show that the theory of $H_T$-fields has a model companion (under the assumption that $T$ is itself model complete). To this end, we examine various extensions of $H_T$-fields, often under the assumption that $T$ is \emph{power bounded}. Our investigation is based on the study of $H$-fields and their extensions, conducted in~\cite{ADH17}. The main results in this article are on the existence and uniqueness of \emph{$T$-Liouville closures}. Before giving an overview of the results, let us discuss two motivating examples: $\cR$-Hardy fields and expansions of the field $\T$.

\subsection*{$\cR$-Hardy fields}
Perhaps the most natural examples of $H_T$-fields are $\cR$-Hardy fields, introduced in~\cite{DMM94}. Let $\cR$ be an o-minimal expansion of the real field $\R$ in an appropriate language $\cL_{\cR}$ and let $T_{\cR}$ be the complete $\cL_{\cR}$-theory of $\cR$. By adding a function symbol to $\cL_{\cR}$ for each definable function, we may arrange that $T_{\cR}$ has quantifier elimination and a universal axiomatization. For convenience, we will also assume that $\cL_{\cR}$ contains a constant symbol for each $r \in \R$, so any model of $T_{\cR}$ contains $\cR$ as an elementary substructure.

\medskip\noindent
Recall that a \textbf{Hardy field} is an ordered differential field of germs at $+\infty$ of unary real-valued functions, where the germ of a function $f\colon \R\to\R$ at $+\infty$ is the equivalence class
\[
[f]\ \coloneqq \ \big\{g:f|_{(a,+\infty)}=g|_{(a,+\infty)}\text{ for some }a\in \R\big\}.
\]
An \textbf{$\cR$-Hardy field} is a Hardy field $\cH$ which is closed under all function symbols in $\cL_{\cR}$. That is, $\cH$ is an $\cR$-Hardy field if for every $n$-ary function symbol $F$ in $\cL_{\cR}$ and all germs $[f_1],\ldots,[f_n]\in \cH$, the germ $\big[F(f_1,\ldots,f_n)\big]$ is in $\cH$, where $F(f_1,\ldots,f_n)$ is the composite function $x\mapsto F\big(f_1(x),\ldots,f_n(x)\big)$. We view constant symbols as nullary functions, so we may identify $\R$ with a subfield of $\cH$ by identifying $r \in \R$ with the germ of the constant function $x\mapsto r$.

\medskip\noindent
If $\cH$ is an $\cR$-Hardy field, then we view $\cH$ as an $\cL_{\cR}$-structure as follows:
\begin{itemize}
\item If $F$ is an $n$-ary function symbol in $\cL_{\cR}$ and $[f_1],\ldots,[f_n]\in \cH$, then 
\[
F\big([f_1],\ldots,[f_n]\big)\ \coloneqq \ \big[F(f_1,\ldots,f_n)\big].
\]
\item If $R$ is an $n$-ary predicate in $\cL_{\cR}$ and $[f_1],\ldots,[f_n]\in \cH$, then
\[
\cH\models R\big([f_1],\ldots,[f_n]\big)\ \logeq\ \cR \models R\big(f_1(x),\ldots,f_n(x)\big)\text{ for all sufficiently large }x.
\]
\end{itemize}
By~\cite[Lemma 5.8]{DMM94}, $\cR$ is an elementary $\cL_{\cR}$-substructure of $\cH$. As a consequence of our assumption on $T_{\cR}$, each $\cL_{\cR}(\emptyset)$-definable function $F$ is given piecewise by terms, so the identity
\[
F\big([f_1],\ldots,[f_n]\big)\ \coloneqq \ \big[F(f_1,\ldots,f_n)\big]
\]
holds for arbitrary $\cL_{\cR}(\emptyset)$-definable functions $F$, not just for function symbols in $\cL_{\cR}$.

\medskip\noindent
Let $\cH$ be an $\cR$-Hardy field. As a Hardy field, $\cH$ admits a derivation $\der\colon \cH\to\cH$ given by $\der[f]\coloneqq [f']$. Using the chain rule from elementary calculus, it is easy to see that this is even a $T_{\cR}$-derivation. We claim that with this derivation, $\cH$ is an $H_{T_{\cR}}$-field. Note that the constant field $C = \ker(\der)$ of $\cH$ is equal to $\R$, so (H2) follows from Dedekind completeness of the reals. For $[f] \in \cH$, we have 
\[
[f]>\R\ \Longrightarrow\ \lim\limits_{x\to \infty} f(x) = \infty\ \Longrightarrow\ [f'] = \der[f]>0,
\]
so (H1) holds as well. 

\subsection*{Transseries}
Let $\R_{\an}$ be the expansion of the real field by restricted analytic functions, and let $\R_{\anexp}$ be the further expansion of $\R_{\an}$ by the unrestricted exponential function. Let $\R_{\re}$ be the expansion of the real field by only the restricted sine, cosine, and exponential functions (collectively, \emph{restricted elementary functions}). Let $T_{\re}$, $T_{\an}$, and $T_{\anexp}$ be the elementary theories of $\R_{\re}$, $\R_{\an}$ and $\R_{\anexp}$ respectively. By~\cite[Corollary 2.8]{DMM97}, the field $\T$ admits a canonical expansion to a model of $T_{\anexp}$, which we denote by $\T_{\anexp}$. Let $\T_{\an}$ and $\T_{\re}$ denote the corresponding reducts of $\T_{\anexp}$. 

\medskip\noindent
We claim that $\T_{\anexp}$, with its natural derivation, is an $H_{\anexp}$-field, where we write ``$H_{\anexp}$-field'' instead of ``$H_{T_{\anexp}}$-field'' for easier reading. It is well-known that the ordered differential field $\T$ is an $H$-field; the axiom (H1) is verified in~\cite[Proposition 4.3]{DMM01} and the axiom (H2) follows easily since the constant field of $\T$ is isomorphic to $\R$, which is Dedekind complete. Moreover, the derivation on $\T_{\anexp}$ is a $T_{\anexp}$-derivation, since it is compatible with all restricted analytic functions and the exponential function by~\cite[Corollaries 3.3 and 3.4]{DMM01}; see~\cite[Lemma 2.9]{FK19} for why this is sufficient. It follows that $\T_{\an}$ is an $H_{\an}$-field, and that $\T_{\re}$ is an $H_{\re}$-field.

\medskip\noindent
We conjecture that $\T_{\anexp}$ and $\T_{\an}$ are both model complete. We also conjecture that the theory of $\T_{\anexp}$ is essentially the model companion of the theory of $H_{\anexp}$-fields, and that the theory of $\T_{\an}$ is essentially the model companion of the theory of $H_{\an}$-fields. These conjectures are, of course, inspired by the corresponding results for $\T$ from~\cite{ADH17}. In~\cite[Chapter 8]{Ka21}, we show that $\T_{\re}$ is model complete, and that its theory is essentially the model companion of the theory of $H_{\re}$-fields. The proof of this result relies heavily on machinery from~\cite{ADH17}, as well as basic facts about restricted elementary functions. The methods used to investigate $\T_{\re}$ are almost surely too case-specific to handle the richer expansions $\T_{\an}$ and $\T_{\anexp}$.

\subsection*{Outline and overview of results}
In this article, we fix a complete, model complete o-minimal theory $T$ which extends the theory $\RCF$ of real closed ordered fields in some appropriate language $\cL\supseteq\{0,1,+,-,\cdot,<\}$. The class of $H_T$-fields is axiomatized in the language $\LdO\coloneqq \cL\cup\{\cO,\der\}$, where $\cO$ is a unary predicate interpreted as in (H2). By~\cite[Lemma 2.3]{FK19}, the constant field $C$ of an $H_T$-field $K$ is an underlying elementary $\cL$-substructure of $K$. It follows that the valuation ring $\cO$ in (H2) is \emph{$T$-convex}, in the sense of van den Dries and Lewenberg~\cite{DL95}. In studying $H_T$-fields, valuation theory plays a key role. Analogs of many classical results about valued fields have been proven in the o-minimal setting under the assumption of power boundedness (a generalization of polynomial boundedness introduced in~\cite{Mi96}). Accordingly, we will assume at various times that $T$ is \emph{power bounded}. Valuation-theoretic background is given in Section~\ref{sec:Tconvex}, along with some new technical results on $T$-convex valuation rings in the power bounded setting.

\medskip\noindent
Expansions of models of $T$ by a $T$-convex valuation ring and a continuous $T$-derivation were first studied in~\cite{Ka21B}. There, the focus was on immediate extensions. Our first result on $H_T$-fields is a corollary of the main result in~\cite{Ka21B}.

\begin{namedthm*}{Corollary~\ref{cor:asympHT}}
Suppose that $T$ is power bounded. Then every $H_T$-field has a spherically complete immediate $H_T$-field extension.
\end{namedthm*}

\noindent
In addition to studying $H_T$-fields, we consider the broader class of \emph{pre-$H_T$-fields}, which arise as substructures of $H_T$-fields. Most of the study of extensions of $H_T$-fields and pre-$H_T$-fields is carried out in Section~\ref{sec:HT}. Here is one such extension result from that section, which shows that every pre-$H_T$-field has an ``$H_T$-field hull.''

\begin{namedthm*}{Theorem~\ref{thm:HThull}}
Suppose that $T$ is power bounded and let $K$ be a pre-$H_T$-field. Then $K$ has an $H_T$-field extension $H_T(K)$ such that for any $H_T$-field extension $M$ of $K$, there is a unique $\LdO(K)$-embedding $H_T(K)\to M$.
\end{namedthm*}

\noindent
The extension theory developed Section~\ref{sec:HT} is applied in Section~\ref{sec:Liouville} to study \emph{$T$-Liouville closures}. An $H_T$-field $K$ is said to be \emph{Liouville closed} if every element in $K$ has an integral and an exponential integral in $K$, that is, if for all $y \in K$, there is $f \in K$ and $g \in K^\times$ with $\der f = \der g/g = y$. A \emph{$T$-Liouville closure} of an $H_T$-field $K$ is a Liouville closed $H_T$-field extension of $K$ which is built from $K$ by adjoining integrals and exponential integrals (a precise definition is given in Section~\ref{sec:Liouville}).

\medskip\noindent
In~\cite{AD02}, Aschenbrenner and van den Dries proved that every $H$-field has at least one and at most two Liouville closures up to isomorphism. They proved that \emph{grounded} $H$-fields have exactly one Liouville closure and that certain types of ungrounded $H$-fields have exactly two. Gehret later showed that the precise dividing line for ungrounded $H$-fields is the property of \emph{$\upl$-freeness}~\cite{Ge17C}. The terms ``grounded'' and ``$\upl$-free'' are defined in Section~\ref{sec:HTasymp}. In Section~\ref{sec:Liouville}, we show that when $T$ is power bounded, the number of $T$-Liouville closures of an $H_T$-field can likewise be characterized in terms of being grounded or $\upl$-free:

\begin{namedthm*}{Theorem~\ref{thm:oneortwotlclosures}}
Suppose that $T$ is power bounded and let $K$ be an $H_T$-field. If $K$ is grounded or if $K$ is ungrounded and $\upl$-free, then $K$ has exactly one $T$-Liouville closure up to $\LdO(K)$-isomorphism. If $K$ is ungrounded and not $\upl$-free, then $K$ has exactly two $T$-Liouville closures up to $\LdO(K)$-isomorphism. For any Liouville closed $H_T$-field extension $M$ of $K$, there is an $\LdO(K)$-embedding of some $T$-Liouville closure of $K$ into $M$.
\end{namedthm*}

\noindent
The assumption of power boundedness excludes some important o-minimal theories, such as $T_{\anexp}$. Fortunately, many of our results can be extended to o-minimal theories which are ``controlled'' by a power bounded base theory in the same way that $T_{\anexp}$ is ``controlled'' by its reduct $T_{\an}$, as shown in~\cite{DMM94}. In formalizing precisely what we mean, we use the framework from Foster's thesis~\cite{Fo10}. Suppose that $T$ is power bounded, that $T$ defines a restricted exponential function, and that the field of exponents of $T$ is cofinal in the prime model of $T$. Foster axiomatizes an extension $T^{\e}$ of $T$ in the language $\cL_{\log}\coloneqq \cL\cup\{\log\}$ whose models are expansions of models of $T$ by an unrestricted exponential function (which is compatible with the restricted exponential function and the power functions of $T$). This extension is also complete, model complete, and o-minimal. We give an overview of Foster's results, as well as some historical context, in Subsection~\ref{subsec:expsandlogs}. 

\medskip\noindent
In Section~\ref{sec:logHT} we show that many of the extension results proven in Section~\ref{sec:HT} under the assumption of power boundedness also go through for $H_{T^{\e}}$-fields and pre-$H_{T^{\e}}$-fields. Instead of working with (pre)-$H_{T^{\e}}$-fields directly, we work with a broader class of $\LdO_{\log}$-structures, called \emph{logarithmic (pre)-$H_T$-fields}, where the logarithm is not assumed to be surjective. In Section~\ref{sec:logLiouville}, we show that Theorem~\ref{thm:oneortwotlclosures} generalizes to the class of logarithmic $H_T$-fields:

\begin{namedthm*}{Theorem~\ref{thm:logoneortwotlclosures}}
Let $K$ be a logarithmic $H_T$-field. If $K$ is $\upl$-free, then $K$ has exactly one logarithmic $T$-Liouville closure up to $\LdO_{\log}(K)$-isomorphism. Otherwise, $K$ has exactly two logarithmic $T$-Liouville closures up to $\LdO_{\log}(K)$-isomorphism. For any Liouville closed logarithmic $H_T$-field extension $M$ of $K$, there is an $\LdO_{\log}(K)$-embedding of some logarithmic $T$-Liouville closure of $K$ into $M$.
\end{namedthm*}

\noindent
In the statement of Theorem~\ref{thm:logoneortwotlclosures}, a \emph{logarithmic $T$-Liouville closure} of $K$ is just a logarithmic $H_T$-field extension of $K$ which is also a $T$-Liouville closure of $K$. The ``grounded'' case in Theorem~\ref{thm:oneortwotlclosures} can not occur for logarithmic $H_T$-fields; see Corollary~\ref{cor:ungrounded}. 

\medskip\noindent
In~\cite{AD02}, Aschenbrenner and van den Dries showed that any $H$-field embedding of a Hardy field $\cH$ into $\T$ extends to the smallest Liouville closed Hardy field extension of $\cH$. As an application of Theorem~\ref{thm:logoneortwotlclosures}, we prove an analog of this embedding theorem for $\R_{\anexp}$-Hardy fields.

\begin{namedthm*}{Theorem~\ref{thm:Ranexpembedding}}
Let $\cH$ be an $\R_{\anexp}$-Hardy field. Then any $H_{\anexp}$-field embedding $\cH \to \T_{\anexp}$ extends to an $H_{\anexp}$-field embedding $\Li_{\anexp}(\cH)\to \T_{\anexp}$, where $\Li_{\anexp}(\cH)$ is the minimal Liouville closed $\R_{\anexp}$-Hardy field extension of $\cH$.
\end{namedthm*}

\noindent
Our final result is that every pre-$H_T$-field has a pre-$H_T$-field extension which satisfies the ``order 1 intermediate value property.'' An analog of this theorem was shown for pre-$H$-fields in~\cite{AD05} and for $\cR$-Hardy fields in~\cite{vdD00}. Unlike many of the other results, this requires no power boundedness assumptions on $T$.

\begin{namedthm*}{Theorem~\ref{thm:IVP}}
Let $K$ be a pre-$H_T$-field. Then $K$ has a pre-$H_T$-field extension $M$ with the following property: for any $\cL(M)$-definable continuous function $F\colon M\to M$ and any $b_1,b_2 \in M$ with 
\[
b_1'\ <\ F(b_1),\qquad b_2'\ >\ F(b_2),
\]
there is $a \in M$ between $b_1$ and $b_2$ with $a' = F(a)$. 
\end{namedthm*}

\subsection*{Acknowledgements}
Research for Sections~\ref{sec:Tconvex}--\ref{sec:Liouville} and Section~\ref{sec:IVP} was conducted at the University of Illinois Urbana-Champaign, and the results in these sections first appeared in my PhD thesis~\cite{Ka21}. The material in Sections~\ref{sec:logHT} and~\ref{sec:logLiouville} is based upon work supported by the National Science Foundation under Award No. 2103240. I would like to thank Allen Gehret and Lou van den Dries for helpful discussions and feedback.

\section{Preliminaries}\label{sec:prelims}

\subsection{Notation and conventions}\label{subsec:notation}
In this article, $k$, $m$, and $n$ always denote elements of $\N =\{0,1,2,\ldots\}$. By ``ordered set'' we mean ``totally ordered set.'' Let $S$ be an ordered set, let $a \in S$, and let $A \subseteq S$. We put 
\[
S^{>a}\ \coloneqq \ \{s\in S:s>a\};
\]
similarly for $S^{\geq a}$, $S^{<a}$, $S^{\leq a}$, and $S^{\neq a}$. We write ``$a>A$'' (respectively ``$a<A$'') if $a$ is greater (less) than each $s \in A$, and we let $A^{\downarrow}$ denote the \textbf{downward closure of $A$}. A \textbf{cut} in $S$ is a downward closed subset of $S$, and if $A$ is a cut in $S$, then an element $y$ in an ordered set extending $S$ is said to \textbf{realize the cut $A$} if $A< y< S\setminus A$. If $\Gamma$ is an ordered abelian group, then we set $\Gamma^{>}\coloneqq \Gamma^{>0}$, and we define $\Gamma^{\geq}$, $\Gamma^<$, $\Gamma^{\leq}$, and $\Gamma^{\neq}$ analogously. If $R$ is a ring, then $R^\times$ denotes the multiplicative group of units in $R$. A \textbf{well-indexed sequence} is a sequence $(a_\rho)$ whose terms are indexed by ordinals $\rho$ less than some infinite limit ordinal $\nu$.

\medskip\noindent
We always use $K$, $L$, and $M$ for models of $T$ (or expansions thereof). Let $A \subseteq K$ and let $D\subseteq K^n$. We say that $D$ is \textbf{$\cL(A)$-definable} if
\[
D\ =\ \varphi(K) \ \coloneqq \ \big\{y \in K^n:K \models \varphi(y)\big\}
\]
for some $\cL(A)$-formula $\varphi$. A function $F\colon D \to K$ is said to be $\cL(A)$-definable if its graph is a definable subset of $K^{n+1}$. Note that the domain of an $\cL(A)$-definable function is $\cL(A)$-definable.

\medskip\noindent
For $A \subseteq K$, let $\dclL(A)$ denote the $\cL$-definable closure of $A$ (in $K$, implicitly, but this doesn't change if we pass to elementary extensions of $K$). If $b \in \dclL(A)$, then $b = F(a)$ for some $\cL(\emptyset)$-definable function $F$ and some tuple $a$ from $A$. It is well-known that $(K,\dclL)$ is a pregeometry. 
A set $B \subseteq K$ is said to be \textbf{$\cL(A)$-independent} if $b\not\in \dclL\big(A\cup(B \setminus \{b\})\big)$ for all $b \in B$. A tuple $a=(a_i)_{i \in I}$ is said to be $\cL(A)$-independent if its set of components $\{a_i:i \in I\}$ is $\cL(A)$-independent and no components are repeated. Since $T$ has definable Skolem functions, any definably closed subset $A \subseteq K$ is an elementary $\cL$-substructure of $K$. Together with our assumption that $T$ is complete, this guarantees that $T$ has a prime model, which can be uniquely identified with $\dclL(\emptyset)$ in any model of $T$. 

\medskip\noindent
Let $M$ be a $T$-extension of $K$, that is, a model of $T$ which contains $K$ as an $\cL$-substructure. Given an $\cL(K)$-definable set $D \subseteq K^n$, we let $D^M$ denote the subset of $M^n$ defined by the same $\cL(K)$-formula as $D$. We sometimes refer to $D^M$ as the \textbf{natural extension of $D$ to $M$}. Since $T$ is assumed to be model complete, this natural extension does not depend on the choice of defining formula. If $F\colon D \to K$ is an $\cL(K)$-definable function, then let $F^M$ be the $\cL(K)$-definable function whose graph is the natural extension of the graph of $F$. Then the domain of $F^M$ is $D^M$, and we often drop the superscript and just write $F\colon D^M\to M^m$. 

\medskip\noindent
For $B \subseteq M$, let $K\langle B\rangle$ denote the $\cL$-substructure of $M$ with underlying set $\dclL(K\cup B)$. If $B = \{b_1,\ldots,b_n\}$, we write $K\langle b_1,\ldots,b_n\rangle$ instead of $K\langle B \rangle$. Note that $K\langle B \rangle$ is an elementary $\cL$-substructure of $M$. If $B$ is $\cL(K)$-independent and $M = K\langle B \rangle$, then $B$ is called a \textbf{basis for $M$ over $K$}. The \textbf{rank of $M$ over $K$}, denoted $\rkL(M|K)$, is the cardinality of a basis for $M$ over $K$ (this doesn't depend on the choice of basis). We say that $M$ is a \textbf{simple extension} of $K$ if $\rkL(M|K) = 1$. If $M$ is a simple extension of $K$, then $M = K\langle b\rangle$ for some $b \in M\setminus K$.

\medskip\noindent
Let $\cL^*\supseteq \cL$, let $T^*$ be an $\cL^*$-theory extending $T$, and let $K \models T^*$. We use the same conventions for $\cL^*$-definability as we do for $\cL$-definability. A \textbf{$T^*$-extension of $K$} is a model $M \models T^*$ which contains $K$ as an $\cL^*$-substructure. If $M$ is an \emph{elementary} $T^*$-extension of $K$ and $D \subseteq K^n$ is $\cL^*(K)$-definable, then let $D^M$ denote the subset of $M^n$ defined by the same formula as $D$.

\subsection{Power functions and power boundedness}\label{subsec:powers}
Here we mention some basic facts about power functions and exponentials. Proofs can be found in~\cite{Mi96} or~\cite[Section 2.2.4]{Fo10}. A \textbf{power function on $K$} is an $\cL(K)$-definable endomorphism of the multiplicative group $K^>$. Each power function $F$ is $\cC^1$ on $K^>$ and uniquely determined by $F'(1)$. Set
\[
\Lambda\ \coloneqq \ \big\{F'(1): F \text{ is a power function on }K\big\}.
\]
Then $\Lambda$ is a subfield of $K$, and it is called the \textbf{field of exponents of $K$}. For $a \in K^>$ and a power function $F$, we suggestively write $F(a)$ as $a^\lambda$, where $\lambda=F'(1)$. A straightforward computation tells us that the derivative of the power function $x \mapsto x^\lambda$ at $a$ is $\lambda a^{\lambda - 1}$.

\medskip\noindent
We say that $K$ is \textbf{power bounded} if for each $\cL(K)$-definable function $F\colon K^>\to K^>$, there is $\lambda$ in the field of exponents of $K$ with $|F(x)|<x^\lambda$ for all sufficiently large positive $x$. Note that $K$ is \textbf{polynomially bounded} (any unary $\cL(K)$-definable function is eventually bounded by $x^n$ for some $n$) if and only if $K$ is power bounded with archimedean field of exponents.

\medskip\noindent
An \textbf{exponential function on $K$} is an ordered group isomorphism from the additive group $K$ to the multiplicative group $K^>$. Any exponential function on $K$ grows more quickly than every power function on $K$. By~\cite{Mi96}, either $K$ is power bounded or $K$ defines an exponential function. Any definable exponential function on $K$ is everywhere differentiable, and if $K$ defines an exponential function, then it is fairly easy to see that there is a unique $\cL(\emptyset)$-definable exponential function which is equal to its own derivative. Thus, defining an exponential function is a property of the theory $T$ (we say that \textbf{$T$ defines an exponential}), and so power boundedness is a property of the theory $T$ as well (we say that \textbf{$T$ is power bounded}). If $T$ is power bounded, then each power function on $K$ is $\cL(\emptyset)$-definable, so we refer to the field of exponents $\Lambda$ as the \textbf{field of exponents of $T$}, as $\Lambda$ does not depend on $K$.

\subsection{Exponentials and logarithms} \label{subsec:expsandlogs}
Most of the results in Sections~\ref{sec:HTasymp},~\ref{sec:HT}, and~\ref{sec:Liouville} are proved under the assumption that $T$ is power bounded. In Sections~\ref{sec:logHT} and~\ref{sec:logLiouville}, we show that many of these results can be extended to theories which define an exponential, so long as these theories are essentially controlled by a power bounded reduct. The prototypical example of this is the theory $T_{\exp} = \Th(\R_{\exp})$. Let $T_{\rexp}$ be the (polynomially bounded) theory of the real field expanded by the restriction of the exponential function to the interval $[-1,1]$. Then $T_{\rexp}$ is model complete by Wilkie~\cite{Wi96}. In~\cite{Re93}, Ressayre demonstrated that $T_{\exp}$ can be axiomatized by extending $T_{\rexp}$ by the following axioms:
\begin{enumerate}
\item[(E1)] $\exp$ is an exponential function, as defined above,
\item[(E2)] $\exp$ agrees with its restricted counterpart on the interval $[-1,1]$, and
\item[(E3)] $\exp$ grows sufficiently fast, that is, if $x \geq n^2$ then $\exp x\geq x^n$ for each $n>0$.
\end{enumerate}
Ressayre also showed that if $T_{\rexp}$ eliminates quantifiers in some language $\cL^*$, then $T_{\exp}$ eliminates quantifiers in the language $\cL^*\cup\{\exp,\log\}$.

\medskip\noindent
Ressayre's method was expounded on in~\cite{DMM94}, where the authors took $T_{\an}$ as the polynomially bounded base theory and showed that $T_{\anexp}$ can be axiomatized by extending $T_{\an}$ by the axioms (E1), (E2), and (E3) above. They used this to show that $T_{\anexp}$ has quantifier elimination and a universal axiomatization in the language $\cL_{\an} \cup \{\exp,\log\}$, where $\cL_{\an}$ extends the language of ordered rings by function symbols for all restricted analytic functions, inversion away from zero, and $n^{\text{\tiny th}}$ roots for all $n>1$.

\medskip\noindent
In~\cite{DS00}, van den Dries and Speissegger further generalized this approach. Let $\cR$ be a polynomially bounded expansion of the real field which defines the restriction of the exponential function to $[-1,1]$ and let $\cL_{\cR}$ be a language in which $T_{\cR}$, the theory of $\cR$, has quantifier elimination and a universal axiomatization. Then the expansion of $\cR$ by the unrestricted exponential function is o-minimal and completely axiomatized by the theory $T_\cR^{\e}$, which expands $T_{\cR}$ by axioms (E1), (E2), and (E3) above along with an axiom scheme stating that $\exp(rx) = (\exp x)^r$ whenever the power function $x \mapsto x^r$ is definable in $\cR$~\cite[Theorem B]{DS00}. As with $T_{\anexp}$, the theory $T_\cR^{\e}$ has quantifier elimination and a universal axiomatization in the language $\cL_{\cR} \cup \{\exp,\log\}$. Note that the axiom expressing compatibility with the power function $x \mapsto x^r$ is really only necessary when $r$ is irrational; if $r \in \Q$, then compatibility follows from (E1).

\medskip\noindent
The most general application of this method is in Foster's thesis~\cite{Fo10}, where $T_{\cR}$ above is replaced by any complete power bounded o-minimal theory, subject to some natural constraints. We will use Foster's framework in Sections~\ref{sec:logHT} and~\ref{sec:logLiouville}, and we spend the remainder of this subsection describing his hypotheses and results.

\begin{definition}
A \textbf{restricted exponential on $K$} is a strictly increasing map $\e\colon K\to K$ which is zero outside of $[-1,1]$, which satisfies the identity
\[
\e(x + y)\ =\ \e(x)\e(y)
\]
for $x, y \in [-1, 1]$ with $|x+y|\leq 1$, and which is differentiable at $0$ with derivative 1. 
\end{definition}

\noindent
As is the case with total exponential functions, any $\cL(K)$-definable restricted exponential on $K$ is actually $\cL(\emptyset)$-definable, continuous on the interval $[-1,1]$, differentiable on the interval $(-1,1)$, and equal to its own derivative on that interval. Thus, we say that \textbf{$T$ defines a restricted exponential} to mean that any model of $T$ admits an $\cL(\emptyset)$-definable restricted exponential. For the remainder of this subsection, we assume that $T$ defines a restricted exponential $\e$. We also assume $T$ is power bounded and that $\Lambda$, the field of exponents of $T$, is cofinal in the prime model of $T$.

\begin{lemma}\label{lem:leqyminus1}
Let $y\in [-1,1]$. Then $y+1 \leq \e(y)$, with equality if and only if $y = 0$.
\end{lemma}
\begin{proof}
Since $\e(0) = 1$, so may assume that $y$ is nonzero. The o-minimal mean value theorem gives some $u$ between $y$ and $0$ with
\[
\e(y)-1\ = \ \e(y)-\e(0) \ = \ \e(u)y.
\]
Treating the cases $y>0$ and $y<0$ separately, we see that $\e(u)y>y$, so $\e(y)>y+1$.
\end{proof}

\noindent
For $y\in K$ with $\e(-1)\leq y \leq \e(1)$, let $\ln(y)$ be the unique element of $[-1,1]$ with $\e(\ln y) = y$. Then $\ln$ is $\cL(\emptyset)$-definable and continuous where defined. A straightforward computation gives that $\ln$ is differentiable at $y$ strictly between $\e(-1)$ and $\e(1)$ with derivative $y\inv$.

\begin{definition}
A \textbf{logarithm on $K$} is a function $\log\colon K^>\to K$ which satisfies the following axioms for all $a\in K^>$ and all $\lambda\in \Lambda$:
\begin{enumerate}
\item[(L1)] $\log$ is an ordered group embedding of the multiplicative group $K^>$ into the additive group $K$;
\item[(L2)] if $\e(-1)\leq a \leq \e(1)$, then $\log a = \ln a$;
\item[(L3)] if $\lambda >1$ and $a\geq \lambda^2$, then $a\geq \lambda\log a$;
\item[(L4)] $\log(a^\lambda)= \lambda \log a$.
\end{enumerate}
\end{definition}

\noindent
Let $\log$ be a logarithm on $K$. We set $\cL_{\log}\coloneqq \cL\cup \{\log\}$, and we view $K$ as an $\cL_{\log}$-structure, where $\log$ is interpreted to be identically zero on $K^{\leq}$. Let $\exp$ denote the compositional inverse of $\log$, where it is defined. Combining (L1) and (L2), we see that $\log$ is differentiable at any $y >0$ with derivative $y\inv$. Thus, $\exp$ is also differentiable where defined, and equal to its own derivative. 

\medskip\noindent
The axioms (L1)--(L4) are analogs of Foster's axioms (A1)--(A5)~\cite[Section 6.5]{Fo10}, though they are presented in terms of logarithms rather than exponentials. Axiom (L1) is an analog of Ressayre's axiom (E1), axiom (L2) corresponds to Ressayre's axiom (E2), and axiom (L3) is a version of (E3) which works for arbitrary (possibly non-archimedean) fields of exponents. Axiom (L4) is an analog of van den Dries and Speissegger's additional axiom. 

\medskip\noindent
We prove here a couple of lemmas for later use.

\begin{lemma}\label{lem:closeislog}
Let $\log$ be a logarithm on $K$, let $f,g \in K$ with $|f-g|\leq 1$, and suppose that $g \in \log(K^>)$. Then $f \in \log(K^>)$.
\end{lemma}
\begin{proof}
Let $h\coloneqq \exp(g)\e(f-g)\in K^>$. Then $f = \log h$.
\end{proof}

The inequality in axiom (L3) is not strict (a useful formulation for verifying that it holds in certain situations). However, the strict version of (L3) follows:

\begin{lemma}\label{lem:powerstrictlog}
Let $\log$ be a logarithm on $K$, let $\lambda \in \Lambda$ with $\lambda > 1$, and let $a \in K$ with $a >\lambda^2$. Then $a > \lambda \log a$. 
\end{lemma}
\begin{proof}
Let $\rho\in \Lambda$ with $1<\rho^2\leq \e(1)$. For example, we may take $\rho = 4/3$, since $\e(1)>2$ by Lemma~\ref{lem:leqyminus1}. If $a\geq \rho^2\lambda^2$, then $a\geq \rho\lambda \log a> \lambda \log a$ by (L3), so we may assume that $\lambda^2<a<\rho^2\lambda^2$. Take $u \in K$ with $\lambda^2u = a$, so $1< u< \e(1)$. We have
\[
\lambda \log a\ =\ \lambda \log \lambda^2+ \lambda \ln u\ \leq \ \lambda^2+ \lambda^2\ln u,
\]
where the equality uses (L1) and (L2) and the inequality uses (L3) and our assumption that $\lambda > 1$. Lemma~\ref{lem:leqyminus1} gives $u -1= \e(\ln u)-1 > \ln u$, so $\lambda \log a < \lambda^2+ \lambda^2(u-1) = \lambda^2 u = a$.
\end{proof}

\noindent
Let $T^{\e}$ be the $\cL_{\log}$-theory which extends the axioms (L1)--(L4) by an additional axiom which states that $\log$ is surjective. Foster's main results on the theory $T^{\e}$ are as follows:

\begin{fact}[\cite{Fo10}, Theorems 6.5.2 and 6.6.9]\label{fact:expelim}\
\begin{enumerate}
\item $T^{\e}$ is complete, model complete, and o-minimal.
\item The prime model of $T$ admits a unique expansion to a model of $T^{\e}$.
\item If $T$ has quantifier elimination and a universal axiomatization, then so does $T^{\e}$ in the language $\cL_{\log}\cup \{\exp\}$, where $\exp$ is interpreted as the compositional inverse of $\log$.
\end{enumerate}
\end{fact}

\begin{corollary}\label{cor:expterms}
Any $\cL_{\log}(\emptyset)$-definable function in any model of $T^{\e}$ is given piecewise by a composition of $\cL(\emptyset)$-definable functions, $\log$, and $\exp$. 
\end{corollary}
\begin{proof}
Let $\cL^*$ be the extension of $\cL$ by function symbols for each $\cL(\emptyset)$-definable function. Then $T$ has quantifier elimination and a universal axiomatization in the language $\cL^*$, since $T$ has definable Skolem functions. Thus, $T^{\e}$ has quantifier elimination and a universal axiomatization in the language $\cL^*_{\log} \cup \{\exp\}$ by Fact~\ref{fact:expelim} (3) above. It follows that any $\cL^*_{\log}(\emptyset)$-definable function in any model of $T^{\e}$ is given piecewise by terms in the language $\cL^*_{\log} \cup \{\exp\}$. These terms are exactly compositions of $\cL(\emptyset)$-definable functions, $\log$, and $\exp$. 
\end{proof}

\subsection{$T$-derivations}\label{subsec:Tderiv}
Let $\der\colon K \to K$ be a map. For $a \in K$, we use $a'$ or $\der a$ in place of $\der(a)$, and we use $a^{(r)}$ in place of $\der^r(a)$. If $a\neq 0$, then we set $a^\dagger\coloneqq a'/a$. Given a set $A \subseteq K$, we set $\der A \coloneqq \{a':a \in A\}$. Given a tuple $b = (b_1, \ldots, b_n) \in K^n$, we use $\der b$ or $b'$ to denote the tuple $(b_1', \ldots, b_n')$.

\medskip\noindent
Let $F\colon U \to K$ be an $\cL(\emptyset)$-definable $\cC^1$-function with $U \subseteq K^n$ open. Let $\nabla F$ denote the gradient
\[
\nabla F\ \coloneqq \ \left(\frac{\partial F}{\partial Y_1},\ldots,\frac{\partial F}{\partial Y_n}\right),
\]
viewed as an $\cL(K)$-definable map from $U$ to $K^n$. If $n = 1$, then we write $F'$ instead of $\nabla F$. We say that $\der$ is \textbf{compatible with $F$} if
\[
F(u)'\ =\ \nabla F(u)\cdot u'
\]
for each $u\in U$, where $a\cdot b$ denotes the dot product $\sum_{i=1}^n a_i b_i$ for $a,b \in K^n$. We say that $\der$ is a \textbf{$T$-derivation on $K$} if $\der$ is compatible with every $\cL(\emptyset)$-definable $\cC^1$-function with open domain. 

\medskip\noindent
$T$-derivations were introduced in~\cite{FK19}. Compatibility with the functions $(x,y) \mapsto x+y$ and $(x,y) \mapsto xy$ gives that any $T$-derivation on $K$ is a \emph{derivation} on $K$, that is, a map satisfying the identities $(a+b)' = a' +b'$ and $(ab)' = a'b+ab'$ for $a,b \in K$. For the remainder of this subsection, we assume that $\der$ is a $T$-derivation on $K$, and we let $C\coloneqq \ker(\der)$ denote the constant field of $K$. By~\cite[Lemma 2.3]{FK19}, the constant field $C$ is the underlying set of an elementary $\cL$-substructure of $K$. We recall two facts from~\cite{FK19} for later use:

\begin{fact}[\cite{FK19}, Lemma 2.12]
\label{fact:basicTderivation2}
Let $U\subseteq K^n$ be open and let $F\colon U\to K$ be an $\cL(K)$-definable $\cC^1$-function. Then there is a (necessarily unique) $\cL(K)$-definable function $F^{[\der]}\colon U \to K$ such that
\[
F(u)'\ =\ F^{[\der]}(u)+\nabla F(u)\cdot u'
\]
for all $u \in U$.
\end{fact}

\begin{fact}[\cite{FK19}, Lemma 2.13]
\label{fact:transext}
Let $M$ be a $T$-extension of $K$, let $A$ be a $\dclL$-basis for $M$ over $K$, and let $s\colon A\to M$ be a map. Then there is a unique extension of $\der$ to a $T$-derivation on $M$ such that $a'= s(a)$ for all $a \in A$.
\end{fact}

\noindent
Let us also note for future use that if $x \mapsto x^\lambda\colon K^>\to K$ is a $\cL(\emptyset)$-definable power function on $K$, then $(y^\lambda)' = \lambda y ^{\lambda-1}y'$ for all $y \in K^>$.

\noindent\medskip
Let $s \in K$. An element $a$ with $a' = s$ is called an \textbf{integral of $s$} and a nonzero element $b$ with $b^\dagger = s$ is called an \textbf{exponential integral of $s$}. If $s \neq 0$, then an element $f$ with $f'= s^\dagger$ is called a \textbf{$\d$-logarithm of $s$}. In differential algebra, ``$\d$-logarithms'' are often just called ``logarithms'' (this is the case in~\cite{ADH17}). We include the ``$\d$'' since there may be actual logarithms present. Suppose that $T$ defines an exponential function with compositional inverse $\log$. Then for $s>0$, the $\d$-logarithms of $s$ are exactly the elements $\log(s)+c$ for some $c \in C$.

\medskip\noindent
We say that $K$ is \textbf{closed under integration} if every element of $K$ has an integral in $K$, and we say that $K$ is \textbf{closed under exponential integration} if every element of $K$ has an exponential integral in $K^\times$. We say that $K$ is \textbf{Liouville closed} if $K$ is closed under both integration and exponential integration. Finally, we say that $K$ is \textbf{closed under taking $\d$-logarithms} if every element of $K^\times$ has a $\d$-logarithm in $K$. Note that if $K$ is closed under integration, then $K$ is closed under taking $\d$-logarithms. Of course, if $T$ defines an exponential function, then $K$ is closed under taking $\d$-logarithms. 

\section{$T$-convex valuation rings}\label{sec:Tconvex}
\noindent
Following~\cite{DL95}, a subset $\cO \subseteq K$ is said to be a \textbf{$T$-convex valuation ring of $K$} if $\cO$ is nonempty and convex and if $F(\cO) \subseteq \cO$ for all $\cL(\emptyset)$-definable continuous functions $F\colon K \to K$. Let $\LO\coloneqq \cL\cup \{\cO\}$ be the extension of $\cL$ by a unary predicate $\cO$ and let $\TO$ be the $\LO$-theory which extends $T$ by axioms asserting that $\cO$ is a $T$-convex valuation ring. For the rest of this section, let $K = (K,\cO) \models \TO$. Unlike in~\cite{DL95}, we allow $\cO = K$, in which case $K$ is said to be \textbf{trivially valued}. The theory $\TO$ is \emph{weakly o-minimal}---every $\LO(K)$-definable subset of $K$ is a finite union of convex subsets of $K$~\cite[Corollary 3.14]{DL95}.

\medskip\noindent
Any $T$-convex valuation ring is a valuation ring, so $\cO$ has a unique maximal ideal, which we denote by $\smallo$. We let $\Gamma$ denote the value group of $K$, and we denote the (surjective) Krull valuation by $v\colon K^\times\to \Gamma$. If $T$ is power bounded with field of exponents $\Lambda$, then the value group $\Gamma$ naturally admits the structure of an ordered $\Lambda$-vector space by
\[
\lambda v(a)\ \coloneqq \ v(a^\lambda)
\]
for $a \in K^>$ (this does not depend on the choice of $a$). We set $\Gamma_\infty\coloneqq \Gamma\cup\{\infty\}$ where $\infty > \Gamma$, and we extend $v$ to all of $K$ by setting $v(0)\coloneqq \infty$. For $a,b \in K$, set
\[
a \preceq b\ \logeq\ va\ \geq\ vb,\qquad a \prec b\ \logeq\ va> vb,\qquad a\asymp b\ \logeq\ va =vb,\qquad a\sim b\ \logeq\ v(a - b)>va.
\]
Note that $\sim$ is an equivalence relation on $K^\times$, and that if $a\sim b$, then $a \asymp b$ and $a$ is positive if and only if $b$ is.

\medskip\noindent
Let $\res(K)\coloneqq \cO/\smallo$ denote the \textbf{residue field of $K$}, and for $a\in \cO$, let $\bar{a}$ denote the image of $a$ under the residue map $\cO \to \res K$. Under this map, $\res K$ admits a natural expansion to a $T$-model; see~\cite[Remark 2.16]{DL95} for details. A \textbf{lift of $\res K$} is an elementary $\cL$-substructure $\k$ of $K$ contained in $\cO$ such that the map $a\mapsto\bar{a}\colon\k\to \res K$ is an $\cL$-isomorphism. By~\cite[Theorem 2.12]{DL95}, one can always find a lift of $\res K$. For $D \subseteq K$, set $\overbar{D}\coloneqq \{\bar{a}:a \in D\cap \cO\} \subseteq \res(K)$.

\medskip\noindent
Let $M$ be a $\TO$-extension of $K$ with $T$-convex valuation ring $\cO_M$ and maximal ideal $\smallo_M$. We view $\Gamma$ as a subgroup of $\Gamma_M$ and $\res K$ as an $\cL$-substructure of $\res M$ in the obvious way. Let $v$ and $x \mapsto \bar{x}$ denote their extensions to functions $M^\times \to \Gamma_M$ and $\cO_M\to \res M$. If $\cO \neq K$, then $\cO_M\neq M$ and $M$ is an elementary $\TO$-extension of $K$ by~\cite[Corollary 3.13]{DL95}. If $\cO_M = M$, then $\cO = K$ so $M$ is again an elementary $\TO$-extension of $K$.

\begin{fact}[\cite{DL95}, Section 3]\label{fact:tworings}
Let $K\langle a \rangle$ be a simple $T$-extension of $K$. There are at most two $T$-convex valuation rings $\cO_1$ and $\cO_2$ of $K\langle a\rangle$ which make $K\langle a \rangle$ a $\TO$-extension of $K$:
\[
\cO_1\ \coloneqq \ \big\{y \in K\langle a \rangle: |y|<u\text{ for some }u \in \cO\big\},\qquad \cO_2\ \coloneqq \ \big\{y \in K\langle a \rangle: |y|<d\text{ for all }d \in K\text{ with }d>\cO\big\}.
\]
If there is $b \in K\langle a \rangle$ which realizes the cut $\cO^\downarrow$, then $b$ is contained in $\cO_2$ but not $\cO_1$, so $\cO_1\subsetneq \cO_2$. If there is no such $b$, then $\cO_1 = \cO_2$. 
\end{fact}

\subsection{The Wilkie inequality}\label{subsec:Wilkie}
In this subsection, we assume that $T$ is power bounded with field of exponents $\Lambda$. The following fact is an analog of the Abhyankar-Zariski inequality, and it is referred to in the literature as the \textbf{Wilkie inequality}.

\begin{fact}[\cite{vdD97}, Section 5]
\label{fact:wilkieineq}
Let $M$ be a $\TO$-extension of $K$ and suppose $\rkL(M|K)$ is finite. Then
\[
\rkL(M|K)\geq \rkL(\res M|\res K)+\dim_\Lambda(\Gamma_M/\Gamma).
\]
\end{fact}

\noindent
We most frequently use the Wilkie inequality when $M$ is a simple extension of $K$. Here is a consequence of the Wilkie inequality:

\begin{lemma}\label{lem:uniqueS}
Let $S$ be a cut in $\Gamma$. Then there is a simple $\TO$-extension $K\langle f \rangle$ of $K$ where $f > 0$ and where $vf$ realizes the cut $S$. This extension is unique up to $\LO(K)$-isomorphism and is completely described as follows: $f$ realizes the cut 
\[
\{y \in K: y\leq 0 \text{ or } vy>S\}
\]
and $\cO_{K\langle f \rangle}$ is the convex hull of $\cO$ in $K\langle f \rangle$.
\end{lemma}
\begin{proof}
Let $K\langle f \rangle$ be a simple extension of $K$ where $f$ realizes the cut 
\[
\{y \in K: y\leq 0 \text{ or } vy>S\},
\]
and let $\cO_{K\langle f \rangle}$ be the convex hull of $\cO$ in $K\langle f\rangle$. Then $K\langle f\rangle$ with this $T$-convex valuation ring is indeed a $\TO$-extension of $K$ by Fact~\ref{fact:tworings}, and $vf$ clearly realizes the cut $S$. It remains to show uniqueness. Let $\cO^*$ be another $T$-convex valuation ring of $K\langle f \rangle$ with $\cO^*\cap K = \cO$. If $\cO^* \neq \cO_{K\langle f \rangle}$, then by Fact~\ref{fact:tworings}, there is $g \in \cO^*$ with $g>\cO$. Then the residue field of $K\langle f \rangle$ with respect to $\cO^*$ is strictly bigger than $\res K$, as it contains the image of $g$. By the Wilkie inequality, the value group of $K\langle f \rangle$ with respect to $\cO^*$ is equal to $\Gamma$, so the valuation of $f$ with respect to $\cO^*$ can not realize the cut $S$.
\end{proof}

\noindent
In Proposition~\ref{prop:smallderivsim} below, we use the Wilkie inequality to bound the derivative of a unary $\cL(K)$-definable function. This proposition will be used a number of times in Section~\ref{sec:HT}. First, we need two lemmas.

\begin{lemma}\label{lem:smallderivval}
Let $M=K\langle a \rangle$ be a simple $\TO$-extension of $K$ with $a \succ 1$ and $va\not\in \Gamma$. Let $F\colon K\to K$ be an $\cL(K)$-definable function with $F(a) \preceq 1$. Then $F'(a) \prec a\inv$. 
\end{lemma}
\begin{proof}
By replacing $a$ with $-a$ if need be, we may assume that $a> 0$. The Wilkie inequality gives $\res M = \res K$, so we may take $u \in \cO^\times$ with $F(a) - u \prec 1$. Replacing $F$ with $F-u$, we may assume that $F(a) \prec 1$. Note that this does not change $F'$. 

We first handle the case that $\cO = K$, so $a>K$ and $|F(a)|<K^>$. Phrased in terms of limits, we have
\[
\lim\limits_{x\to \infty}|F(x)| \ = \ 0,
\]
and we want to show that 
\[
\lim\limits_{x\to \infty}x|F'(x)| \ = \ 0.
\]
Let $\epsilon,g \in K^>$ be given. We need to find $d>g$ with $d|F'(d)|<\epsilon$. By increasing $g$, we may assume that $|F(g)|<\epsilon/4$ and that $|F|$ is decreasing and $\cC^1$ on a neighborhood of $[g,+\infty)$. The o-minimal mean value theorem gives $d\in (g,2g)$ with
\[
|F'(d)|\ =\ \Big|\frac{F(2g)-F(g)}{g}\Big|\ \leq\ \frac{2|F(g)|}{g}.
\]
Since $d<2g$ and $|F(g)|<\epsilon/4$ we have
\[
d|F'(d)|\ \leq\ d\frac{2|F(g)|}{g}\ <\ 4|F(g)|\ <\ \epsilon.
\]

Now suppose $\cO \neq K$. Then $M$ is an elementary $\TO$-extension of $K$, so it suffices to show that for any $\LO(K)$-definable set $A \subseteq K^>$ with $a \in A^M$, there is $y \in A$ with $F'(y) \prec y\inv$. Let $A$ be such a set. Since $\TO$ is weakly o-minimal, we may assume that $A$ is open and convex. By shrinking $A$, we arrange that $F$ is $\cC^1$ on $A$ and that $F(y) \prec 1$ for all $y \in A$. Since $va \not\in \Gamma$ and $\Gamma$ is densely ordered, the set $A$ contains elements $y_1, y_2$ with $y_1 \prec y_2$. The o-minimal mean value theorem gives
\[
F'(y)\ =\ \frac{F(y_2)-F(y_1)}{y_2-y_1}.
\]
for some $y \in A$ between $y_1$ and $y_2$. Since $F(y_2)-F(y_1)\prec 1$ and $y_2-y_1 \asymp y_2 \succeq y$, we have $F'(y)\prec y\inv$, as desired.
\end{proof}

\begin{lemma}\label{lem:smallderivres}
Let $M=K\langle a \rangle$ be a simple $\TO$-extension of $K$ with $a\asymp 1$ and $\bar{a} \not\in \res K$. Let $F\colon K\to K$ be an $\cL(K)$-definable function with $F(a) \preceq 1$. Then $F'(a) \preceq 1$.
\end{lemma}
\begin{proof}
This is trivial if $\cO = K$, so we may assume that $K$ is nontrivially valued. Let $\k\subseteq \cO^\times$ be a lift of $\res K$, so $\k\langle a \rangle$ is a lift of $\res M$ by~\cite[Lemma 5.1]{DL95}. Take an $\cL(\k)$-definable function $G\colon K \to K$ with $|F(a)|<G(a)$. The Wilkie inequality gives that $\Gamma_M = \Gamma$, so since $\Gamma^<$ has no largest element, it suffices to show that $|F'(a)| < d$ for each $d \in K^>$ with $d \succ 1$. Let $d$ be given and let $I$ be an arbitrary subinterval of $K^>$ with $a \in I^M$. It suffices to find some $y \in I$ with $|F'(y)|< d$. By shrinking $I$, we arrange that $F$ is $\cC^1$ on $I$ and that $|F(y)|<G(y)$ for all $y \in I$. As $\bar{a} \in \overbar{I}^{\res M}$, we see that $\overbar{I}$ must be infinite, so $I \cap \k$ is infinite. Take $y_1,y_2 \in I\cap \k$ with $y_1<y_2$, so $y_2-y_1 \asymp 1$. Note that $G(y_i)\in \k$, so $|F(y_i)| < G(y_i) \prec d$ for $i = 1,2$. The o-minimal mean value theorem gives
\[
F'(y)\ =\ \frac{F( y_2)-F( y_1)}{y_2-y_1}\ \prec\ d
\]
for some $y \in I$ between $ y_1$ and $ y_2$. In particular, $|F'(y)|<d$. 
\end{proof}

\begin{proposition}\label{prop:smallderivsim}
Let $M=K\langle a \rangle$ be a simple $\TO$-extension of $K$ with $a\not\sim f$ for all $f \in K$ and let $F\colon K\to K$ be an $\cL(K)$-definable function. Then $F'(a) \preceq a\inv F(a)$.
\end{proposition}
\begin{proof}
First, suppose $a \succ 1$ and $va \not\in \Gamma$. The Wilkie inequality gives $\Gamma_M = \Gamma\oplus\Lambda va$, so take $d \in K^>$ and $\lambda \in \Lambda$ with $F(a) \asymp da^\lambda$. Then $d\inv a^{-\lambda} F(a) \preceq 1$ and, applying Lemma~\ref{lem:smallderivval} to the function $y\mapsto d\inv y^{-\lambda} F(y)$, we get 
\[
d\inv a^{-\lambda} F'(a)-\lambda d\inv a^{-\lambda-1} F(a)\ \prec\ a\inv.
\]
Since $-\lambda d\inv a^{-\lambda-1} F(a)\preceq a\inv$, we see that $d\inv a^{-\lambda} F'(a)\preceq a\inv$, so 
\[
F'(a)\ \preceq\ a\inv d a^{\lambda}\ \asymp\ a\inv F(a).
\]
Now, suppose $a \prec 1$ and $va\not\in \Gamma$. Let $G\colon K \to K$ be the function given by
\[
G(y)\ =\ \left\{
\begin{array}{ll}
F(y\inv) & \mbox{ if }y \neq0 \\
0 & \mbox{ if }y=0.
\end{array}
\right.
\]
Then $F(a) = G(a\inv)$. By applying the previous case to $G$ and $a\inv\succ 1$, we get
\[
F'(a) \ =\ -G'(a\inv)a^{-2}\ \preceq\ aG(a\inv)a^{-2}\ =\ a\inv F(a).
\]
Finally, suppose $va \in \Gamma$ and take $b \in K$ with $b \asymp a$, so $b\inv a \asymp 1$. Note that $\overbar{b\inv a} \not\in \res K$, for otherwise we would have $a \sim bu$ for some $u \in \cO^\times$, contradicting our assumption on $a$. The Wilkie inequality gives $\Gamma_M = \Gamma$, so take $d \in K^>$ with $F(a) \asymp d$. Applying Lemma~\ref{lem:smallderivres} with $b\inv a$ in place of $a$ and with the function $y\mapsto d\inv F(by)$ in place of $F$, we see that
\[
d\inv b F'(a)\ \preceq\ 1,
\]
so $F'(a) \preceq b\inv d \asymp a\inv F(a)$.
\end{proof}

\noindent
Note that our standing assumption of power boundedness is necessary for Proposition~\ref{prop:smallderivsim}, as the proposition clearly fails when $a$ is infinite and $F$ is an exponential function with $F' = F$. Our assumption that $a \not\sim f$ for all $f \in K$ is also necessary. To see this, suppose $a \sim f\in K$ and let $F(Y) = Y-f$. Then $F(a) \prec a$ so $a\inv F(a) \prec 1$, but $F'(a) = 1$. Here is an application of Proposition~\ref{prop:smallderivsim} for use in the proof of Lemma~\ref{lem:gapgoesup}.

\begin{corollary}\label{cor:wholethingsmall}
Suppose $\cO = K$, let $b \in K^n$ be an $\cL(\emptyset)$-independent tuple, and let $K \langle a \rangle$ be a simple $\TO$-extension of $K$ with $a \prec 1$. Let $G\colon K^{1+n}\to K$ be an $\cL(\emptyset)$-definable function with $G(a,b) \prec 1$ and let $d= (d_0,\ldots,d_n) \in K^{1+n}$. Then $\nabla G(a,b)\cdot d \prec a\inv$.
\end{corollary}
\begin{proof}
Viewing $G$ as a function of the variables $Y_0,\ldots,Y_n$, we have
\[
\nabla G(a,b)\cdot d \ =\ \frac{\partial G}{\partial Y_0}(a,b)d_0+ \frac{\partial G}{\partial Y_1}(a,b)d_1+\cdots +\frac{\partial G}{\partial Y_n}(a,b)d_n.
\]
Since $d_i\preceq 1$ for $i = 0,\ldots n$, it suffices to show that $\frac{\partial G}{\partial Y_i}\prec a\inv$ for each $i$. For $i = 0$, we apply Proposition~\ref{prop:smallderivsim} to the function $y \mapsto G(y,b)$ to get $\frac{\partial G}{\partial Y_0}(a,b)\preceq a\inv G(a,b) \prec a\inv$. For $i >0$, we will again use Proposition~\ref{prop:smallderivsim}, but doing so requires a bit of an argument. By symmetry, it suffices to show that $\frac{\partial G}{\partial Y_1}(a,b) \prec a\inv$. Let $E \coloneqq \dcl(b_2,\ldots,b_n)$ and view $E$ as an elementary $\LO$-substructure of $K$ with $\cO_E = \cO \cap E = E$. Then $b_1 \not\in E$, so $b_1 \not\sim f$ for any $f \in E$, since $E$ is trivially valued. Viewing $E\langle a \rangle$ as a $\TO$-extension of $E$ with $va \not\in \Gamma_E = \{0\}$, the Wilkie inequality gives $\res E\langle a \rangle = \res E$, so $b_1 \not\sim f$ for any $f \in E\langle a \rangle$. Thus, we may apply Proposition~\ref{prop:smallderivsim} with $E\langle a \rangle$ in place of $K$, with $b_1$ in place of $a$, and with the function $y \mapsto G(a,y,b_2,\ldots,b_n)$ in place of $F$ to get $\frac{\partial G}{\partial Y_1}(a,b)\preceq b_1\inv G(a,b)$. Since $b_1 \inv \in K$ and $G(a,b) \prec 1$, this gives $\frac{\partial G}{\partial Y_1}(a,b)\prec 1 \prec a\inv$.
\end{proof}

\subsection{Immediate extensions and pseudocauchy sequences}
In this subsection, let $M$ be a $\TO$-extension of $K$. If $\Gamma_M = \Gamma$ and $\res M = \res K$, then $M$ is said to be an \textbf{immediate extension of $K$}. If $M$ is an immediate extension of $K$, then $M$ is an elementary $\TO$-extension of $K$. Note that $M$ is an immediate extension of $K$ if and only if for all $a \in M^\times$ there is $b \in K^\times$ with $a \sim b$. The following fact from~\cite{Ka21B}, is useful for studying how $\LO(K)$-definable functions behave in immediate extensions. 

\begin{fact}[\cite{Ka21B}, Corollary 1.5]
\label{fact:Lflat}
Suppose $M$ is an immediate extension of $K$, let $F\colon A\to K$ be an $\LO(K)$-definable function, and let $a \in A^M$. Then there is an $\cL(K)$-definable cell $D \subseteq A$ with $a \in D^M$ such that either $F(y) = 0$ for all $y\in D^M$ or $F(y) \sim F(a)$ for all $y \in D^M$.
\end{fact}

\noindent
If $K\langle a \rangle$ is a simple immediate $\TO$-extension of $K$, then $v(a-K)\coloneqq \big\{v(a-y):y \in K\big\}$ is a downward closed subset of $\Gamma$ without a greatest element. For each $b \in K\langle a \rangle$, the set $v(b-K)$ can be expressed as a translate of $v(a-K)$:

\begin{lemma}\label{lem:shift}
Let $K\langle a \rangle$ be a simple immediate $\TO$-extension of $K$ and let $b \in K\langle a \rangle\setminus K$. Then 
\[
v(b-K) \ =\ \gamma+v(a-K)
\]
for some $\gamma \in \Gamma$.
\end{lemma}
\begin{proof}
Let $F\colon K\to K$ be an $\cL(K)$-definable function with $F(a) = b$. Take an open interval $I \subseteq K$ with $a \in I^{K\langle a \rangle}$ such that $F$ is $\cC^1$ on $I$. Since $b\not\in K$, we have $F'(a)\neq 0$, so we may use Fact~\ref{fact:Lflat} to shrink $I$ and arrange that $F'(y) \sim F'(a)$ for all $y \in I^{K\langle a \rangle}$. Set $\gamma\coloneqq vF'(a) \in \Gamma$ and let $u\in I$. By the o-minimal mean value theorem, we have
\[
F(a)- F(u)\ =\ F'(c)(a-u)
\]
for some $c\in K\langle a \rangle$ between $a$ and $u$. Then $vF'(c)=\gamma$ since $c\in I^{K\langle a \rangle}$, so 
\[
v\big(b-F(u)\big) \ =\ v\big(F(a)-F(u)\big)\ =\ \gamma+v(a-u).
\]
The set $\big\{v(a-u):u\in I\big\}$ is cofinal in $v(a-K)$ and, since $\big\{F(u):u \in I\big\}$ contains an interval around $b$, the set $\big\{v\big(b-F(u)\big):u\in I\big\}$ is cofinal in $v(b-K)$. This gives $v(b-K) = \gamma+v(a-K)$, as desired.
\end{proof}

\noindent
Recall that a \textbf{pseudocauchy sequence (pc-sequence) in $K$} is a well-indexed sequence $(a_\rho)$ in $K$ such that
\[
a_\tau-a_\sigma\ \prec\ a_\sigma-a_\rho
\]
for all $\tau>\sigma>\rho$ greater than some index $\rho_0$. Let $(a_\rho)$ be a pc-sequence. An element $a$ in a $\TO$-extension of $K$ is said to be a \textbf{pseudolimit of $(a_\rho)$} if for some index $\rho_0$, we have
\[
a-a_\sigma\ \prec\ a-a_\rho
\]
for all $\sigma>\rho>\rho_0$. In this case, we say that \textbf{$(a_\rho)$ pseudoconverges to $a$}, and we write $a_\rho\leadsto a$. The pc-sequence $(a_\rho)$ is said to be \textbf{divergent} if it has no pseudolimit in $K$. Suppose $(a_\rho)$ is divergent with pseudolimit $a$ in some $\TO$-extension of $K$. Given $y \in K$, we have $a-a_\rho\prec a-y$ for all sufficiently large $\rho$; otherwise, we would have $a_\rho \leadsto y$.

\medskip\noindent
Under the assumption of power boundedness, pc-sequences are related to immediate extensions as follows:

\begin{lemma}
\label{lem:immediateext2}
Suppose that $T$ is power bounded and let $K\langle a \rangle$ be a simple $\TO$-extension of $K$. The following are equivalent:
\begin{enumerate}
\item $K\langle a \rangle$ is an immediate extension of $K$;
\item $v(a-K)$ has no largest element;
\item there is a divergent pc-sequence in $K$ which pseudoconverges to $a$.
\end{enumerate}
\end{lemma}
\begin{proof}
The equivalence of (1) and (2) follows from results in Tyne's thesis~\cite{Ty03}; see~\cite[Lemma 1.9]{Ka21B} for a full proof. Assume (2) holds and let $(a_\rho)$ be a well-indexed sequence in $K$ such that $v(a-a_\rho)$ is strictly increasing and cofinal in $v(a-K)$. One easily verifies that $(a_\rho)$ is a divergent pc-sequence in $K$ which pseudoconverges to $a$. Now, assume (3) holds and let $(a_\rho)$ be a pc-sequence witnessing this. Then for $y\in K$, we may take $\rho$ with $a-a_\rho \prec a-y$, proving (2). 
\end{proof}

\medskip\noindent
As a corollary, we get that any divergent pc-sequence in $K$ has a pseudolimit in an immediate extension of $K$.

\begin{corollary}
\label{cor:immediateext4}
Let $T$ be power bounded and let $(a_\rho)$ be a divergent pc-sequence in $K$. Then there is an immediate $\TO$-extension $K\langle a \rangle$ of $K$ with $a_\rho\leadsto a$. If $b$ is an element of a $\TO$-extension $M$ of $K$ with $a_\rho \leadsto b$, then there is a unique $\LO(K)$-embedding $K\langle a \rangle\to M$ sending $a$ to $b$.
\end{corollary}
\begin{proof}
Let $K\langle a \rangle$ be a simple $\TO$-extension of $K$ with $a_\rho\leadsto a$ (such an extension exists by compactness). Lemma~\ref{lem:immediateext2} gives that $K\langle a \rangle$ is an immediate extension of $K$. Let $b$ be an element of a $\TO$-extension $M$ of $K$ with $a_\rho \leadsto b$. We claim that $a$ and $b$ realize the same cut in $K$. Let $y \in K$ and take $a_\rho$ with $a-a_\rho \prec a-y$ and $b-a_\rho \prec b-y$. Then 
\[
a-y\ \sim\ (a-y)-(a-a_\rho)\ =\ a_\rho-y\ =\ (b-y)-(b-a_\rho)\ \sim\ b-y,
\]
so $a < y$ if and only if $b<y$. This gives us a unique $\cL(K)$-embedding $\iota\colon K\langle a \rangle\to M$ sending $a$ to $b$. To get that $\iota$ is an $\LO(K)$-embedding, let $F\colon K\to K$ be $\cL(K)$-definable. We need to show that $F(a) \in \cO_{K\langle a \rangle}$ if and only if $F(b) \in \cO_M$. We assume that $F(a) \neq 0$, and we will show that $F(a) \sim F(b)$. Using Fact~\ref{fact:Lflat}, take an interval $I \subseteq K$ with $a \in I^{K\langle a \rangle}$ such that $F(y) \sim F(a)$ for all $y \in I^{K\langle a \rangle}$. Since $K$ has a proper immediate extension, it is not trivially valued, so $M$ is an \emph{elementary} $\TO$-extension of $K$. Thus, $F(b) \sim F(y) \sim F(a)$, since $b \in I^M$.
\end{proof}

\section{$H_T$-asymptotic fields}\label{sec:HTasymp}
In this section, we introduce the class of $H_T$-asymptotic fields. These fields are o-minimal analogs of the $H$-asymptotic fields from~\cite[Chapter 9]{ADH17}. We will collect a few facts from~\cite{ADH17} for later use, and then we will discuss the immediate extensions of $H_T$-asymptotic fields. As we will see in Section~\ref{sec:HT}, all (pre)-$H_T$-fields are $H_T$-asymptotic. Before proceeding, let us fix some notation. Set $\LdO\coloneqq \cL\cup\{\cO,\der\}$, and let $\TdO$ be the $\LdO$-theory which extends $\TO$ by axioms stating that $\der$ is a $T$-derivation. We also make the following standing assumption:

\begin{assumption}\label{ass:TdO}
For the remainder of this article, let $K = (K,\cO,\der) \models \TdO$. We will continue to use the notation introduced in Subsection~\ref{subsec:Tderiv} and Section~\ref{sec:Tconvex}. In particular, we write $\smallo$ for the maximal ideal of $\cO$, $\Gamma$ for the value group of $(K,\cO)$, and $C$ for the constant field of $(K,\der)$.
\end{assumption}

\begin{definition}
$K$ is an \textbf{$H_T$-asymptotic field} if for all $g \in K$ with $g\succ 1$, we have
\begin{enumerate}
\item[(HA1)] $g^\dagger > 0$,
\item[(HA2)] $g^\dagger \succ f'$ for all $f \in \smallo$, and 
\item[(HA3)] $g^\dagger \succeq f'$ for all $f \in \cO^\times$.
\end{enumerate}
\end{definition}

\noindent
The definition above differs slightly from the definition of an $H$-asymptotic field given in~\cite{ADH17}, though we claim that every $H_T$-asymptotic field is $H$-asymptotic. Indeed, (HA2) and (HA3) along with~\cite[Proposition 9.1.3]{ADH17} imply that every $H_T$-asymptotic field $K$ is \emph{asymptotic}, that is, $f\prec g\Longleftrightarrow f' \prec g'$ for all $f,g \in K^\times$ with $f,g\prec 1$. To see that each $H_T$-asymptotic field $K$ is \emph{$H$-asymptotic} in the sense of~\cite{ADH17}, let $f,g \in K^\times$ with $f\prec g \prec 1$. We need to show that $f^\dagger \succeq g^\dagger$. Applying condition (HA1) to $g\inv$ and $g/f$, we have
\[
g^\dagger\ =\ -(g\inv)^\dagger\ < \ 0,\qquad g^\dagger - f^\dagger\ =\ (g/f)^\dagger \ >\ 0,
\]
so $f^\dagger <g^\dagger < 0$. In particular, $f^\dagger \succeq g^\dagger$, as desired. Conversely, if $K$ is asymptotic, then $K$ satisfies (HA2) and (HA3) by~\cite[Proposition 9.1.3]{ADH17}. However, the $H$-asymptotic fields in~\cite{ADH17} are not necessarily ordered, and even convexly ordered $H$-asymptotic fields need not satisfy (HA1). 

\medskip\noindent
For the remainder of this section, we assume that $K$ is an $H_T$-asymptotic field. Note that if $\cO \neq K$, then the derivation on $K$ is nontrivial by (HA1). Indeed, (HA1) ensures that the constant field of any $H_T$-asymptotic field is contained in the valuation ring. 

\begin{fact}[\cite{ADH17}, Corollary 9.1.4]\label{fact:asymsim}
Let $f,g \in K$ with $g \not\asymp 1$. If $f \prec g$, then $f'\prec g'$. If $f \preceq g$, then $f \sim g \Longleftrightarrow f' \sim g'$.
\end{fact}

\noindent 
Let $f \in K^\times$ with $f \not\asymp 1$. As $K$ is asymptotic, the values $v(f')$ and $v(f^\dagger)$ only depend on $vf$, so for $\gamma =vf$, we set
\[
\gamma^\dagger\ \coloneqq \ v(f^\dagger),\qquad \gamma' \ \coloneqq \ v(f')\ =\ \gamma+ \gamma^\dagger.
\]
This gives us a map
\[
\psi\colon \Gamma^{\neq}\to \Gamma,\qquad \psi(\gamma) \ \coloneqq \ \gamma^\dagger.
\]
Following Rosenlicht~\cite{Ro81}, we call the pair $(\Gamma,\psi)$ the \textbf{asymptotic couple of $K$}. We have the following important subsets of $\Gamma$:
\begin{align*}
(\Gamma^<)'\ &\coloneqq \ \{\gamma': \gamma \in \Gamma^<\},& (\Gamma^>)'\ &\coloneqq \ \{\gamma': \gamma \in \Gamma^>\},\\
(\Gamma^{\neq})'\ &\coloneqq \ (\Gamma^<)' \cup (\Gamma^>)',& \Psi\ &\coloneqq \ \psi(\Gamma^{\neq}) \ =\ \{\gamma^\dagger: \gamma \in \Gamma^{\neq}\}.
\end{align*}
It is always the case that $(\Gamma^<)' < (\Gamma^>)'$ and that $\Psi< (\Gamma^>)'$. If there is $\beta \in \Gamma$ with $\Psi <\beta< (\Gamma^>)'$, then we call $\beta$ a \textbf{gap in $K$}. There is at most one such $\beta$, and if $\Psi$ has a largest element, then there is no such $\beta$. If $K$ has trivial valuation, then the four subsets above are empty and $0$ is a gap in $K$. We say that \textbf{$K$ is grounded} if $\Psi$ has a largest element, and we say that \textbf{$K$ is ungrounded} otherwise. Finally, we say that \textbf{$K$ has asymptotic integration} if $\Gamma = (\Gamma^{\neq})'$. If $\beta$ is a gap in $K$ or if $\beta = \max \Psi$, then $\Gamma\setminus (\Gamma^{\neq})' = \{\beta\}$. We have the following trichotomy for the structure of $H_T$-asymptotic fields:

\begin{fact}[\cite{ADH17}, Corollary 9.2.16]\label{fact:trich}
Exactly one of the following holds:
\begin{enumerate}
\item $K$ has asymptotic integration;
\item $K$ has a gap;
\item $K$ is grounded.
\end{enumerate}
\end{fact}

\noindent
It follows that $K$ is ungrounded if and only if $\Psi \subseteq (\Gamma^<)'$. This allows us to define the \textbf{contraction map} $\chi\colon \Gamma^{\neq} \to \Gamma^<$ for ungrounded $K$ as follows: let $\gamma \in \Gamma^{\neq}$, so $\gamma^\dagger \in \Psi \subseteq (\Gamma^<)'$. Let $\chi(\gamma)$ be the unique element of $\Gamma^<$ with $\chi(\gamma)' = \gamma^\dagger$. We extend $\chi$ to a map $\Gamma \to \Gamma^{\leq}$ by setting $\chi(0)\coloneqq 0$. 

\subsection{Spherically complete immediate extensions of $H_T$-asymptotic fields}
In this subsection, we will use the results in~\cite{Ka21B} to prove the following theorem:

\begin{theorem}
\label{thm:asymp}
Suppose that $T$ is power bounded. Then $K$ has an immediate $H_T$-asymptotic field extension which is spherically complete.
\end{theorem}

\noindent
The analog of Theorem~\ref{thm:asymp} for $H$-asymptotic fields was established in~\cite{ADH18B}. Theorem~\ref{thm:asymp} is a fairly immediate consequence of~\cite[Theorem 6.3]{Ka21B}, but we will provide some additional detail. Let us begin with a test for whether an immediate extension of $K$ is $H_T$-asymptotic:

\begin{lemma}\label{lem:testforsmall}
Let $M$ be an immediate $\TdO$-extension of $K$. If $f' \prec g^\dagger$ for all $f \in \smallo_M$ and all $g \in K$ with $g \succ 1$, then $M$ is an $H_T$-asymptotic field.
\end{lemma}
\begin{proof}
Let $h \in M$ with $h \succ 1$ and take $g \in K$ with $h \sim g$. For $\epsilon \in \smallo_M$ with $h = g(1+\epsilon)$, we have
\[
h^\dagger - g^\dagger \ =\ (h/g)^\dagger \ =\ (1+\epsilon)^\dagger\ =\ \frac{\epsilon'}{1+\epsilon} \ \sim \ \epsilon'.
\] 
By assumption, $\epsilon' \prec g^\dagger$, so $h^\dagger \sim g^\dagger$. As $K$ is $H_T$-asymptotic, we have $g^\dagger >0$, so $h^\dagger > 0$ as well. Additionally, we have $h^\dagger \sim g^\dagger \succ f'$ for all $f \in \smallo_M$, by assumption. Now let $f \in \cO_M^\times$. Take $u \in K$ and $\delta \in \smallo_M$ with $f = u+\delta$, so $f' = u'+\delta'$. We have $\delta' \prec g^\dagger$ by assumption and $u' \preceq g^\dagger$, since $K$ is $H_T$-asymptotic and $u,g \in K$, so $f' \preceq g^\dagger \sim h^\dagger$.
\end{proof}

\noindent
The main objects of study in~\cite{Ka21B} are \emph{$T$-convex $T$-differential fields}: models of $\TdO$ in which the $T$-derivation which is continuous with respect to the valuation topology. Thus, the next step for us is to verify that this continuity assumption holds.

\begin{lemma}\label{lem:ctsderiv}
The derivation on $K$ is continuous with respect to the valuation topology.
\end{lemma}
\begin{proof}
If $\cO = K$, then the valuation topology on $K$ is the discrete topology, and the derivation on $K$ is trivially continuous, so we may assume that $\cO \neq K$. Take $g \in K$ with $g\succ 1$. Then $g^\dagger \succ f'$ for all $f \in \smallo$ by (HA2), so $\der \smallo \subseteq g^\dagger \smallo$. By~\cite[Lemma 4.4.7]{ADH17}, continuity of the derivation is equivalent to the existence of $\phi \in K^\times$ with $\der\smallo \subseteq \phi \smallo$, so $\der$ is indeed continuous. 
\end{proof}

\noindent
Central in~\cite{Ka21B} are the following subsets of $\Gamma$, which were first introduced in~\cite{ADH18B}:
\[
\Gamma(\der)\ \coloneqq \ \big\{v \phi :\phi \in K^\times \text{ and } \der\smallo\subseteq \phi\smallo\big\},\qquad S(\der)\ \coloneqq \ \big\{\gamma\in \Gamma:\Gamma(\der)+\gamma = \Gamma(\der)\}.
\]
Using that $K$ is $H_T$-asymptotic, these subsets can be described explicitly:

\begin{lemma}\label{lem:GammaandPsi}
$\Gamma(\der) = \Gamma \setminus (\Gamma^>)'$ and $S(\der) = \{0\}$.
\end{lemma}
\begin{proof}
For $\phi \in K^\times$, we have 
\[
v\phi \not\in (\Gamma^>)'\ \Longleftrightarrow\ v\phi < (\Gamma^>)'\ \Longleftrightarrow\ f' \prec \phi\text{ for all }f \in \smallo\ \Longleftrightarrow\ v\phi \in \Gamma(\der),
\]
proving the first equality. For the second, note that $S(\der)$ is a convex subgroup of $\Gamma$, so it suffices to show that $S(\der) \cap \Gamma^> = \emptyset$. Let $\gamma \in \Gamma^>$, so 
\[
\gamma^\dagger \ \in\ \Psi\ \subseteq\ \Gamma \setminus (\Gamma^>)',\qquad \gamma^\dagger + \gamma\ =\ \gamma' \ \in\ (\Gamma^>)'.
\]
Since $\Gamma(\der) = \Gamma \setminus (\Gamma^>)'$, this tells us that $\gamma^\dagger \in \Gamma(\der)$, but $\gamma^\dagger+\gamma\not\in \Gamma(\der)$, so $\gamma \not\in S(\der)$.
\end{proof}

\noindent
An immediate $\TdO$-extension $M$ of $K$ is said to be a \textbf{strict extension of $K$} if $\der_M \smallo_M \subseteq \phi \smallo_M$ for each $\phi \in K^\times$ with $v\phi \in \Gamma(\der)$. We note that this is \emph{not} the definition of a strict extension given in~\cite{Ka21B} or in~\cite{ADH18B}; however, it is equivalent under the assumption that $M$ is an immediate $\TO$-extension of $K$ by~\cite[Lemma 1.5]{ADH18B}. If $M$ is an immediate strict extension of $K$, then $\der_M$ is automatically continuous.

\begin{corollary}\label{cor:asympiffstrict}
Let $M$ be an immediate $\TdO$-extension of $K$. Then $M$ is a strict extension of $K$ if and only if $M$ is an $H_T$-asymptotic field.
\end{corollary}
\begin{proof}
Suppose that $M$ is a strict extension of $K$ and let $f \in \smallo_M$. As $\Psi \subseteq \Gamma(\der)$ by Lemma~\ref{lem:GammaandPsi}, we have $f'\prec \phi$ for all $\phi \in K^\times$ with $v\phi \in \Psi$. Thus, $M$ is an $H_T$-asymptotic field by Lemma~\ref{lem:testforsmall}. Conversely, suppose that $M$ is $H_T$-asymptotic, let $\phi \in K^\times$ with $v\phi \in \Gamma(\der)$, and let $f \in \smallo_M$. We need to show that $f' \prec \phi$. Take $g \in \smallo$ with $f \sim g$. Then $f' \sim g'\prec \phi$ by Fact~\ref{fact:asymsim}.
\end{proof}

\begin{proof}[Proof of Theorem~\ref{thm:asymp}]
Since $S(\der) = \{0\}$, there is an immediate strict $T$-convex $T$-differential field extension $M$ of $K$ which is spherically complete by~\cite[Theorem 6.3]{Ka21B} (this uses that $T$ is power bounded). By Corollary~\ref{cor:asympiffstrict}, $M$ is an $H_T$-asymptotic field. 
\end{proof}

\subsection{Simple immediate extensions of $H_T$-asymptotic fields}
In this subsection, we turn our focus to certain types of simple immediate extensions. We assume throughout this subsection that $T$ is power bounded.

\begin{proposition}\label{prop:HT-simimm}
Let $G\colon K\to K$ be an $\cL(K)$-definable function, let $(a_\rho)$ be a divergent pc-sequence in $K$, and suppose $a_\rho'-G(a_\rho) \leadsto 0$. Then $K$ has an immediate $H_T$-asymptotic field extension $K\langle a \rangle$ with $a_\rho\leadsto a$ and $a' = G(a)$. If $b$ is a pseudolimit of $(a_\rho)$ in an $H_T$-asymptotic field extension $M$ of $K$ with $b' = G(b)$, then there is a unique $\LdO(K)$-embedding $K\langle a \rangle \to M$ sending $a$ to $b$.
\end{proposition}
\begin{proof} Let $a$ be a pseudolimit of $(a_\rho)$ in an immediate $\TO$-extension of $K$; such an extension exists by Corollary~\ref{cor:immediateext4}. Using Fact~\ref{fact:transext}, extend the derivation on $K$ to a $T$-derivation on $K\langle a\rangle$ with $a' = G(a)$. We claim that $K\langle a \rangle$ is $H_T$-asymptotic. By Lemma~\ref{lem:testforsmall}, it suffices to show that $f' \prec g^\dagger$ for all $f \in M$ and all $g \in K$ with $f \prec 1 \prec g$. Let $F\colon K \to K$ be an $\cL(K)$-definable function with $F(a) \prec 1$, let $g \in K$ with $g \succ 1$, and suppose towards contradiction that $F(a)' \succeq g^\dagger$. Then $F(a) \not\in K$, so $F'(a)\neq 0$. Take an open interval $I \subseteq K$ with $a \in I^{K\langle a \rangle}$ such that $F$ is $\cC^1$ on $I$, and let $H\colon I\to K$ be the $\cL(K)$-definable function $H(Y) = F^{[\der]}(Y)+F'(Y)G(Y)$ where $F^{[\der]}$ is as defined in Fact~\ref{fact:basicTderivation2}. Then
\[
H(a)\ =\ F^{[\der]}(a)+F'(a)G(a)\ =\ F(a)'\ \succeq\ g^\dagger. 
\]
Using Fact~\ref{fact:Lflat}, we may shrink $I$ to arrange that 
\[
F(y) \sim F(a) \prec 1,\qquad H(y) \sim H(a) \succeq g^\dagger,\qquad F'(y) \sim F'(a)
\]
for all $y \in I^{K\langle a \rangle}$. Let $\rho$ be a sufficiently large index, so $a_\rho \in I$. Since $F(a_\rho) \prec 1$ and $K$ is $H_T$-asymptotic, we have $F(a_\rho)' \prec g^\dagger \preceq H(a_\rho)$. Thus,
\[
H(a_\rho)\ \sim \ H(a_\rho) - F(a_\rho)' \ = \ \big(F^{[\der]}(a_\rho)+F'(a_\rho)G(a_\rho)\big)- \big(F^{[\der]}(a_\rho)+F'(a_\rho)a_\rho'\big) \ =\ F'(a_\rho)\big(G(a_\rho) - a_\rho'\big).
\]
Since $H(a_\rho) \sim H(a)$ and $F'(a_\rho) \sim F'(a) \neq 0$, we have
\[
G(a_\rho) - a_\rho'\ \sim\ \frac{H(a)}{F'(a)}. 
\]
Since this holds for all sufficiently large $\rho$, we have $G(a_\rho) - a_\rho' \sim G(a_\sigma) - a_\sigma'$ for $\sigma,\rho$ sufficiently large, contradicting our assumption that $a_\rho'-G(a_\rho) \leadsto 0$. Thus, $K \langle a \rangle$ is an $H_T$-asymptotic field, as claimed. The embedding property follows from Fact~\ref{fact:transext} and Corollary~\ref{cor:immediateext4}.
\end{proof}

\begin{corollary}
\label{cor:smallint0}
Let $s \in K$ with $vs \in (\Gamma^>)'$ and $s \not\in \der\smallo$, and suppose $v(s-\der\smallo)$ has no largest element. Then $K$ has an immediate $H_T$-asymptotic field extension $K\langle a \rangle$ with $a \prec 1$ and $a' = s$ such that for any $H_T$-asymptotic field extension $M$ of $K$ with $s \in \der\smallo_M$, there is a unique $\LdO(K)$-embedding $K\langle a \rangle \to M$.
\end{corollary}
\begin{proof}
Let $(a_\rho)$ be a well-indexed sequence in $\smallo$ such that $v(s-a_\rho')$ is strictly increasing as a function of $\rho$ and cofinal in $v(s-\der\smallo)$. The proof of~\cite[Lemma 10.2.4]{ADH17} gives that $(a_\rho)$ is a divergent pc-sequence in $K$. We apply Proposition~\ref{prop:HT-simimm} where $G$ is the constant function $s$ to get an immediate $H_T$-asymptotic field extension $K\langle a \rangle$ of $K$ with $a_\rho\leadsto a$ and $a' = s$. Let $M$ be an $H_T$-asymptotic field extension of $K$ and let $b\in \smallo_M$ with $b' = s$. Then for $\rho<\sigma$, we have
\[
(b-a_\rho)'\ =\ s-a_\rho'\ \sim\ (a_\sigma-a_\rho)'.
\]
Since $b-a_\rho,a_\sigma - a_\rho \prec 1$, Fact~\ref{fact:asymsim} gives us that $b-a_\rho \sim a_\sigma-a_\rho$, so $a_\rho\leadsto b$. Proposition~\ref{prop:HT-simimm} gives an $\LdO(K)$-embedding $\imath\colon K\langle a \rangle\to M$ sending $a$ to $b$. For uniqueness, let $\jmath\colon K\langle a\rangle \to M$ be an arbitrary $\LdO(K)$-embedding. Then $\jmath(a) - b \in C_M$ since $\jmath(a)' = s = b'$. Since $\jmath(a),\, b \prec 1$ and $C_M^\times \subseteq \cO_M^\times$, we must have $\jmath(a) = b$. This shows that $\jmath = \imath$.
\end{proof}

\begin{corollary}
\label{cor:bigint0}
Let $s \in K$ with $v(s-\der K)< (\Gamma^>)'$ and suppose $v(s-\der K)$ has no largest element. Then $K$ has an immediate $H_T$-asymptotic field extension $K\langle a \rangle$ with $a' = s$ such that for any $H_T$-asymptotic field extension $M$ of $K$ and $b \in M$ with $b' = s$, there is a unique $\LdO(K)$-embedding $K\langle a \rangle \to M$ sending $a$ to $b$.
\end{corollary}
\begin{proof}
Let $(a_\rho)$ be a well-indexed sequence in $K$ such that $v(s-a_\rho')$ is strictly increasing as a function of $\rho$ and cofinal in $v(s-\der K)$ and such that $s- a_\rho'\prec s$ for each $\rho$. The proof of~\cite[Lemma 10.2.6]{ADH17} gives that $(a_\rho)$ is a divergent pc-sequence in $K$. We apply Proposition~\ref{prop:HT-simimm} where $G$ is the constant function $s$ to get an immediate $H_T$-asymptotic field extension $K\langle a \rangle$ of $K$ with $a_\rho\leadsto a$ and $a' = s$. Let $M$ be an $H_T$-asymptotic field extension of $K$ and let $b\in M$ with $b' = s$. For $\rho<\sigma$, we have
\[
(b-a_\rho)'\ =\ s-a_\rho'\ \sim\ (a_\sigma-a_\rho)',
\]
so $v(b-a_\rho)' \in (\Gamma^<)'$ and $b-a_\rho\succ 1$. Fact~\ref{fact:asymsim} gives $b-a_\rho \sim a_\sigma-a_\rho$, so $a_\rho\leadsto b$ and Proposition~\ref{prop:HT-simimm} gives an $\LdO(K)$-embedding $\imath \colon K\langle a \rangle\to M$ sending $a$ to $b$.
\end{proof}

\subsection{$\upl$-freeness and $\upo$-freeness}
In this subsection, let $K$ be an ungrounded $H_T$-asymptotic field with $\cO \neq K$. A \textbf{logarithmic sequence in $K$} is a well-indexed sequence $(\ell_\rho)$ in $K$ such that:
\begin{enumerate}
\item $\ell_\rho \succ \ell_\sigma\succ 1$ for all $\sigma>\rho$ and $(v\ell_\rho)$ is cofinal in $\Gamma^<$;
\item $v\ell_{\rho+1} = \chi(v\ell_\rho)$ for all $\rho$.
\end{enumerate}
Logarithmic sequences can be constructed by transfinite recursion. Note that if $M$ is an $H_T$-asymptotic field extension of $K$ with $\Gamma^<$ cofinal in $\Gamma^<_M$, then any logarithmic sequence in $K$ is a logarithmic sequence in $M$.

\medskip\noindent
A \textbf{$\upl$-sequence in $K$} is a sequence $(\upl_\rho)$ where $\upl_\rho = -\ell_\rho^{\dagger\dagger}$ for some logarithmic sequence $(\ell_\rho)$ in $K$. By~\cite[Proposition 11.5.3]{ADH17}, any two $\upl$-sequences in $K$ are equivalent as pc-sequences (they have the same pseudolimits in every extension of $K$). We say that $K$ is \textbf{$\upl$-free} if no $\upl$-sequence in $K$ has a pseudolimit in $K$. 

\medskip\noindent
An \textbf{$\upo$-sequence in $K$} is a sequence $(\upo_\rho)$ where $\upo_\rho =-(2\upl_\rho'+\upl_\rho^2)$ for some $\upl$-sequence $(\upl_\rho)$ in $K$. We say that $K$ is \textbf{$\upo$-free} if no $\upo$-sequence in $K$ has a pseudolimit in $K$. If $\upl_\rho \leadsto \upl \in K$, then the corresponding $\upo$-sequence $(\upo_\rho)$ has pseudolimit $-(2\upl'+\upl^2) \in K$, so $\upo$-freeness implies $\upl$-freeness. The property of $\upo$-freeness plays a much larger role than $\upl$-freeness in~\cite{ADH17}, but in this article, $\upl$-freeness is the more central concept. Even so, $\upo$-freeness makes an appearance in Corollary~\ref{cor:omegaconstruction} and Proposition~\ref{prop:upofreeext} below, with an eye towards future work. 

\begin{lemma}\label{lem:unionupl}
If $K$ is an increasing union of $\upl$-free $H_T$-asymptotic fields, then $K$ is $\upl$-free. If $K$ is an increasing union of $\upo$-free $H_T$-asymptotic fields, then $K$ is $\upo$-free.
\end{lemma}
\begin{proof}
By~\cite[Corollary 11.6.1]{ADH17}, $K$ is $\upl$-free if and only if for all $s \in K$, there is $g \in K$ with $g \succ 1$ and $s-g^{\dagger\dagger} \succeq g^\dagger$. By~\cite[Corollary 11.7.8]{ADH17}, $K$ is $\upo$-free if and only if for all $f \in K$, there is $g \in K$ with $g \succ 1$ and 
\[
f-2(g^{\dagger\dagger})'+(g^{\dagger\dagger})^2 \succeq (g^\dagger)^2.
\]
Both of these equivalent conditions are preserved by increasing unions.
\end{proof}

\noindent
For the remainder of this section, let $(\ell_\rho)$ be a logarithmic sequence in $K$ with corresponding $\upl$-sequence $(\upl_\rho)$. Nothing here will depend on the specific choice of $(\ell_\rho)$. The following facts about $\upl$-sequences and $\upl$-freeness are from~\cite{ADH17}:

\begin{fact}[\cite{ADH17}, Lemma 11.5.2 and Corollary 11.6.1]\label{fact:upl->asymp}
If $K$ is $\upl$-free, then $K$ has asymptotic integration. If $K$ is $\upl$-free and $\upl$ is a pseudolimit of $(\upl_\rho)$ in an $H_T$-asymptotic field extension of $K$, then $v(\upl-K) = \Psi^{\downarrow}$. 
\end{fact}

\begin{fact}[\cite{ADH17}, Lemma 11.5.13]\label{fact:uplminusdaggers}
Suppose $K$ has asymptotic integration and let $\upl \in K$ be a pseudolimit of $(\upl_\rho)$. Then $v\big(\upl + (K^\times)^\dagger\big)$ is a cofinal subset of $\Psi^\downarrow$.
\end{fact}

\noindent
For us, the importance of $\upl$-freeness comes from its relation to gaps:

\begin{lemma}\label{lem:uplgapcreator}
Suppose that $T$ is power bounded with field of exponents $\Lambda$ and that $K$ has asymptotic integration. Let $s \in K$ and let $M=K\langle f \rangle$ be an $H_T$-asymptotic field extension of $K$ with $f \neq 0$ and $f^\dagger = s$. Then $vf$ is a gap in $M$ if and only $\upl_\rho\leadsto -s$.
\end{lemma}
\begin{proof}
One direction is by~\cite[Lemma 11.5.12]{ADH17}: if $vf$ is a gap in $M$, then $\upl_\rho\leadsto -s$. For the other direction, suppose $\upl_\rho\leadsto -s$. We first note that $vf \not \in \Gamma$. Indeed, suppose towards contradiction that $f \asymp y$ for some $y \in K^\times$ and take $u \in M$ with $f = uy$. Then 
\[
v(s-y^\dagger)\ =\ v(f^\dagger - y^\dagger)\ = \ v(u^\dagger)\ =\ v(u')\ >\ \Psi,
\]
contradicting that $v\big(s-(K^\times)^\dagger\big)$ is a cofinal subset of $\Psi^\downarrow$ by Fact~\ref{fact:uplminusdaggers}. 

Now we claim that $\Psi_M\subseteq \Psi^\downarrow$. We have established that $vf \not\in \Gamma$, so the Wilkie inequality gives $\Gamma_M = \Gamma\oplus \Lambda vf$. Thus, we fix $\gamma \in \Gamma$ and $\lambda \in \Lambda$ with $\lambda \neq 0$, and we need to show that $\psi(\gamma+\lambda vf) \in \Psi^\downarrow$. Take $y \in K^>$ with $vy = \gamma$ and set $z\coloneqq y^{-1/\lambda}$, so $y^\dagger = -\lambda z^\dagger$. We have
\[
\psi(\gamma+\lambda vf)\ =\ v(yf^\lambda)^\dagger\ =\ v(y^\dagger +\lambda s)\ =\ v(\lambda s - \lambda z^\dagger)\ =\ v(s-z^\dagger),
\]
so $\psi(\gamma+\lambda vf) \in \Psi^\downarrow$ by Fact~\ref{fact:uplminusdaggers}, proving the claim.

Finally, suppose toward contradiction that $vf$ is not a gap in $M$. Then $\Psi_M$, being a cofinal subset of $\Psi^\downarrow$, has no maximum and so $vf \in (\Gamma^{\neq}_M)'$. Take $\beta \in \Gamma^{\neq}_M$ with $\beta' = vf$ and take $y \in K$ with $\beta^\dagger < vy \in \Psi$. Our assumption that $\upl_\rho\leadsto -s$ along with~\cite[Lemma 11.5.6 (iii)]{ADH17} gives $s-y^\dagger \prec y$, so 
\[
\psi(vf- vy) \ =\ v(f/y)^\dagger \ =\ v(s-y^\dagger)\ >\ vy,
\]
contradicting~\cite[Lemma 9.2.2]{ADH17} with $\alpha= vf$ and $\gamma = vy$.
\end{proof}

\noindent
In~\cite{Ge17C}, Gehret defines a property---the yardstick property---which allows us to check whether $\upl$-freeness is preserved in various extensions. Let $S$ be a nonempty convex subset of $\Gamma$ without a largest element. 
\begin{enumerate}
\item We say that $S$ has the \textbf{yardstick property} if there is $\beta \in S$ such that $\gamma - \chi(\gamma) \in S$ for all $\gamma \in S^{>\beta}$.
\item We say that $S$ is \textbf{jammed} if for every nontrivial convex subgroup $\{0\}\neq \Delta \subseteq \Gamma$, there is $\beta \in S$ such that $\gamma - \beta \in \Delta$ for all $\gamma \in S^{>\beta}$.
\end{enumerate}
Note that if $S$ is jammed, then so is $\gamma+S$ for any $\gamma \in \Gamma$. Being jammed and having the yardstick property are incompatible, except in the following case:

\begin{fact}[\cite{Ge17C}, Lemma 3.17]\label{fact:jamyardstick}
Let $S$ be a nonempty convex subset of $\Gamma$ without a largest element which has the yardstick property. Then $S$ is jammed if and only if $S^\downarrow = \Gamma^<$.
\end{fact}

\noindent
The lemma below is an analog of~\cite[Proposition 6.19]{Ge17C} with virtually the same proof; only minor modifications and substitutions are required.

\begin{lemma}\label{lem:yardstickpreserve}
Let $K\langle a\rangle$ be a simple immediate $H_T$-asymptotic field extension of $K$. Suppose that $K$ is $\upl$-free and that $v(a-K) \subseteq \Gamma$ has the yardstick property. Then $K\langle a \rangle$ is $\upl$-free. 
\end{lemma}
\begin{proof}
Suppose toward contradiction that $K\langle a \rangle$ is not $\upl$-free and take $\upl \in K\langle a \rangle$ with $\upl_\rho \leadsto \upl$. Since $\upl\not\in K$, Lemma~\ref{lem:shift} gives $\gamma \in \Gamma$ with $v(\upl-K) = \gamma+v(a-K)$. Fact~\ref{fact:upl->asymp} gives $v(\upl-K) = \Psi^\downarrow$, so $v(\upl-K)$ is jammed by~\cite[Lemma 3.11]{Ge17C}. Thus, $v(a-K)$ is jammed as well, so $v(a-K) = \Gamma^<$ by Fact~\ref{fact:jamyardstick}. In particular, $v(a-K)$ has a supremum in $\Gamma$, so $v(\upl-K) = \Psi^\downarrow$ also has a supremum in $\Gamma$. This implies that $K$ does not have asymptotic integration, contradicting Fact~\ref{fact:upl->asymp}.
\end{proof}

\section{$H_T$-fields and pre-$H_T$-fields}\label{sec:HT}
\noindent
In this section, we introduce the main objects of study. Recall our standing assumption that $K = (K,\cO,\der)$ is a model of $\TdO$.

\begin{definition}
$K$ is a \textbf{pre-$H_T$-field} if for all $g \in K$ with $g\succ 1$, we have
\begin{enumerate}
\item[(PH1)] $g^\dagger > 0$, and
\item[(PH2)] $g^\dagger \succ f'$ for all $f \in \cO$.
\end{enumerate}
\end{definition}

\noindent
Every pre-$H_T$-field is $H_T$-asymptotic, and if $K$ is $H_T$-asymptotic and $g^\dagger \succ f'$ for all $f,g \in K$ with $g\succ f\asymp1$, then $K$ is a pre-$H_T$-field. As with $H_T$-asymptotic fields, every pre-$H_T$-field is a pre-$H$-field, as defined in~\cite{ADH17}. In the case of pre-$H_T$-fields, the converse also holds: If $K$ is a pre-$H$-field, then $K$ is a pre-$H_T$-field. To see this, use~\cite[Lemma 10.1.1]{ADH17} and note that (PH1) is equivalent to the condition that $g' > 0$ for all $g \in K$ with $g >\cO$. If $K$ is an $\LdO$-substructure of a pre-$H_T$-field, then $K$ is itself a pre-$H_T$-field. 

\begin{lemma}\label{lem:logvals}
Suppose that $K$ is a pre-$H_T$-field, let $a \in K^\times$, and let $b\in K$ be a $\d$-logarithm of $a$.
\begin{enumerate}
\item[(i)] If $a \not\asymp 1$, then $b \succ 1$ and $vb = \chi(va)$.
\item[(ii)] If $a \asymp 1$, then $b \preceq 1$.
\end{enumerate}
\end{lemma}
\begin{proof}
If $a \succ 1$, then since $b' = a^\dagger$, we must have $b \succ 1$ by (PH2). It follows immediately that $vb= \chi(va)$. This holds more generally for $a \not\asymp 1$, since $-b' = (a\inv)^\dagger$. On the other hand, if $a \asymp 1$, then
\[
v(b')\ =\ v(a^\dagger) \ =\ v(a')\ >\ (\Gamma^<)',
\]
so $b \preceq 1$.
\end{proof}

\begin{corollary}\label{cor:expungrounded}
Suppose that $K$ is a pre-$H_T$-field. If every element in $K^>$ has a $\d$-logarithm in $K$, then $K$ is ungrounded. In particular, if $T$ defines an exponential function, then $K$ is ungrounded.
\end{corollary}
\begin{proof}
Suppose that every element in $K^>$ has a $\d$-logarithm in $K$, let $\gamma \in \Psi$, and take $a \in K^>$ with $a\not\asymp 1$ and $va^\dagger = \gamma$. Let $b$ be a $\d$-logarithm of $a$, so $b \succ 1$ by Lemma~\ref{lem:logvals} and $\gamma = vb' \in (\Gamma^<)'$. Thus, $\Psi \subseteq (\Gamma^<)'$.
\end{proof}

\noindent
Recall from the introduction that $K$ is an \textbf{$H_T$-field} if
\begin{enumerate}
\item[(H1)] $f'>0$ for all $f \in K$ with $f>\cO$, and
\item[(H2)] $\cO = C+ \smallo$.
\end{enumerate}
Note that if $K$ is an $H_T$-field, then $C$ is a lift of $\res K$.

\begin{lemma}\label{lem:equivHpreH}
The following are equivalent:
\begin{enumerate}
\item $K$ is a pre-$H_T$-field and $\cO = C+\smallo$;
\item $K$ is an $H_T$-asymptotic field and $\cO = C+\smallo$;
\item $K$ is an $H_T$-field.
\end{enumerate}
\end{lemma}
\begin{proof}
It is immediate that (1) implies (2). Suppose (2) holds and let $f \in K$ with $f>\cO$. Then $f \succ 1$ so $f^\dagger >0$ by (HA1). As $f>0$, this gives $f'>0$, so (H1) is satisfied. Of course (H2) is satisfied by assumption, so (3) holds. To see that (3) implies (1), we assume that $K$ is an $H_T$-field, and we will verify that (PH1) and (PH2) hold. For (PH1), let $f\in K$ with $f \succ 1$. Then $|f|>\cO$, so $|f|'>0$. Since $f^\dagger = |f|^\dagger$, this gives $f^\dagger > 0$. Now for (PH2), let $f,g \in K$ with $g\succ 1$ and $f \preceq 1$. We need to show that $g^\dagger \succ f'$. This is shown in~\cite[Lemma 10.5.1]{ADH17}, but we repeat the proof here. First, by replacing $g$ with $-g$ if need be, we may assume that $g>0$. As $\cO = C + \smallo$, we may subtract a constant from $f$ to arrange that $f \prec 1$. Let $c \in C^>$, so $0<c+f,c-f\asymp 1$. This gives $g(c+f),g(c-f)>\cO$, so $g'(c+f)+gf',g'(c-f)-gf'> 0$ by (H1), yielding
\[
g'(c-f)\ >\ gf'\ >\ -g'(c+f).
\]
Dividing by $g$ gives
\[
g^\dagger(c-f)\ >\ f'\ >\ -g^\dagger(c+f).
\]
As $f \prec 1$ and $c\in C^>$ can be taken to be arbitrarily small, we see that $f' \prec g^\dagger$ as desired.
\end{proof}

\begin{corollary}\label{cor:HTres}
Let $K$ be an $H_T$-field and let $M$ be an $H_T$-asymptotic field extension of $K$ with $\res M = \res K$. Then $M$ is an $H_T$-field with $C_M = C$.
\end{corollary}
\begin{proof}
We have $C\subseteq C_M$ and by (HA1), we have $C_M \subseteq \cO_M$. As $C$ is a lift of $\res K = \res M$, it is maximal among the elementary $\cL$-substructures of $M$ contained in $\cO_M$, so $C = C_M$ and $\cO_M = C+\smallo_M$; see~\cite[Remark 2.11 and Theorem 2.12]{DL95}. We conclude that $M$ is an $H_T$-field by Lemma~\ref{lem:equivHpreH}.
\end{proof}

\begin{lemma}\label{lem:asympisHT}
Let $K$ be a pre-$H_T$-field and let $M$ be an immediate $H_T$-asymptotic field extension of $K$. Then $M$ is a pre-$H_T$-field. If $K$ is an $H_T$-field, then $M$ is as well.
\end{lemma}
\begin{proof}
Let $f,g \in M$ with $g\succ f\asymp1$. We need to show that $g^\dagger \succ f'$. Using that $\Gamma_M = \Gamma$, take $a \in K$ with $g\asymp a$, so $g^\dagger \asymp a^\dagger$. Using that $\res M = \res K$, take $b\in K$ with $f- b\prec 1$, so $(f - b)' \prec a^\dagger$, as $M$ is $H_T$-asymptotic. As $K$ is a pre-$H_T$-field, we also have $b' \prec a^\dagger$, so 
\[
f'\ =\ (f - b)'+b'\ \prec\ a^\dagger\ \asymp\ g^\dagger.
\]
If $K$ is an $H_T$-field, then Corollary~\ref{cor:HTres} gives that $M$ is an $H_T$-field as well.
\end{proof}

\noindent
Using Lemma~\ref{lem:asympisHT}, we have the following consequence of Theorem~\ref{thm:asymp}:

\begin{corollary}
\label{cor:asympHT}
Suppose that $T$ is power bounded. Then every pre-$H_T$-field has a spherically complete immediate pre-$H_T$-field extension and every $H_T$-field has a spherically complete immediate $H_T$-field extension.
\end{corollary}

\begin{assumption}
For the remainder of this section, we assume that $T$ is power bounded with field of exponents $\Lambda$ and that $K$ is a pre-$H_T$-field.
\end{assumption}

\subsection{Adjoining integrals} \label{subsec:ints}
To begin this subsection, let us use Corollaries~\ref{cor:smallint0} and~\ref{cor:bigint0} to say something about immediate extensions of $K$ by integrals.

\begin{corollary}\label{cor:smallint}
Let $s \in K$ with $vs \in (\Gamma^>)'$ and $s \not\in \der\smallo$. Then $K$ has an immediate pre-$H_T$-field extension $K\langle a \rangle$ with $a \prec 1$ and $a' = s$ such that for any $H_T$-asymptotic field extension $M$ of $K$ with $s \in \der\smallo_M$, there is a unique $\LdO(K)$-embedding $K\langle a \rangle \to M$. If $K$ is ungrounded and $\upl$-free, then so is $K\langle a \rangle$.
\end{corollary}
\begin{proof}
Let $S \coloneqq v(s-\der\smallo)\subseteq (\Gamma^>)'$. We claim that $S$ has no largest element. Let $y \in \smallo$ and take $b \in \smallo$ with $s-y' \asymp b'$. Take $u \in \cO^\times$ with $s-y' = ub'$. Then (PH2) gives $u' \prec b^\dagger$, so $u'b\prec b'$, and 
\[
s- (y+ub)' \ =\ s-y' - ub'-u'b\ =\ -u'b\ \prec\ b'\ \asymp\ s-y'.
\]
Thus, $S$ has no largest element as claimed, and Corollary~\ref{cor:smallint0} gives an immediate $H_T$-asymptotic field extension $K\langle a \rangle$ of $K$ with $a \prec 1$, $a' = s$, and the desired embedding property. By Lemma~\ref{lem:asympisHT}, $K\langle a \rangle$ is itself a pre-$H_T$-field. By~\cite[Lemma 8.5]{Ge17C}, the set $v(a-K)$ has the yardstick property, so if $K$ is ungrounded and $\upl$-free, then $K\langle a \rangle$ is as well by Lemma~\ref{lem:yardstickpreserve}.
\end{proof}

\begin{corollary}\label{cor:bigint}
Let $s \in K$ with $v(s-\der K)\subseteq (\Gamma^<)'$. Then $K$ has an immediate pre-$H_T$-field extension $K\langle a \rangle$ with $a' = s$ such that for any $H_T$-asymptotic field extension $M$ of $K$ and $b \in M$ with $b' = s$, there is a unique $\LdO(K)$-embedding $K\langle a \rangle \to M$ sending $a$ to $b$. If $K$ is ungrounded and $\upl$-free, then so is $K\langle a \rangle$.
\end{corollary}
\begin{proof}
Let $S \coloneqq v(s-\der K)\subseteq (\Gamma^<)'$. Again, we claim that $S$ has no largest element. Let $y \in K$ and take $b \succ 1$ and $u \asymp 1$ with $s-y' =ub'$. As in the proof of Corollary~\ref{cor:smallint}, we see that $s- (y+ub)' \prec s-y'$, as desired. Corollary~\ref{cor:bigint0} gives an immediate $H_T$-asymptotic field extension $K\langle a \rangle$ of $K$ with $a' = s$ and the desired embedding property. By Lemma~\ref{lem:asympisHT}, $K\langle a \rangle$ is itself a pre-$H_T$-field. By~\cite[Lemma 9.6]{Ge17C}, the set $v(a-K)$ has the yardstick property, so if $K$ is ungrounded and $\upl$-free, then $K\langle a \rangle$ is as well by Lemma~\ref{lem:yardstickpreserve}.
\end{proof}

\noindent
Now we turn to the case that $K$ has a gap. First, we give a useful test for whether a simple extension of $K$ is a pre-$H_T$-field.

\begin{lemma}
\label{lem:simprehext}
Let $K$ be a pre-$H_T$-field and let $M = K\langle a \rangle$ be a $\TdO$-extension of $K$ with $va\not\in\Gamma$. Suppose that for all $g \in K^\times$ and $\lambda \in \Lambda$ with $ga^\lambda \succ 1$, we have 
\begin{enumerate}
\item[(i)] $(ga^\lambda)^\dagger >0$,
\item[(ii)] $(ga^\lambda)^\dagger \succ f'$ for all $f \in K$ with $f \preceq 1$,
\item[(iii)] $(ga^\lambda)^\dagger \succ F(a)'$ for all $\cL(K)$-definable functions $F\colon K\to K$ with $F(a) \prec 1$ and $F(a)\not\in K$.
\end{enumerate}
Then $M$ is a pre-$H_T$-field. If $K$ is an $H_T$-field, then so is $M$. 
\end{lemma}
\begin{proof}
Let $h \in M$ with $h \succ 1$ and take $g \in K^\times$ and $\lambda \in \Lambda$ with $h \asymp ga^\lambda$. By the Wilkie inequality, we have $\res M= \res K$, so by multiplying $g$ with an element in $\cO^\times$, we may even assume that $h \sim ga^\lambda$. Take $\epsilon \in \smallo_M$ with $h = ga^\lambda(1+\epsilon)$. Then
\[
h^\dagger - (ga^\lambda)^\dagger \ =\ (1+\epsilon)^\dagger\ =\ \frac{\epsilon'}{1+\epsilon} \ \sim \ \epsilon'.
\] 
We have $\epsilon' \prec (ga^\lambda)^\dagger$ by (ii) and (iii), so $h^\dagger \sim (ga^\lambda)^\dagger$. In particular, $h^\dagger > 0$ by (i), so (PH1) holds. For (PH2), let $f \in \cO_M$. We need to show that $h^\dagger \succ f'$. Since $h^\dagger \sim (ga^\lambda)^\dagger$, it suffices to show that $(ga^\lambda)^\dagger \succ f'$. This follows from (ii) if $f \in K$, so we may assume $f \not\in K$. As $\res M = \res K$, we may take $u \in \cO$ with $f-u\prec 1$. Take an $\cL(K)$-definable function $F\colon K\to K$ with $F(a) = f-u$. Then $f' = u' +F(a)' \prec (ga^\lambda)^\dagger$ by (ii) and (iii). Finally, if $K$ is an $H_T$-field, then $M$ is as well by Corollary~\ref{cor:HTres}, since $\res M = \res K$.
\end{proof}

\noindent
The next lemma shows that if $K$ has a gap, then this gap has an integral in some pre-$H_T$-field extension of $K$.

\begin{lemma}\label{lem:gapgoesup}
Let $s \in K$ and suppose $vs$ is a gap in $K$. Then $K$ has a pre-$H_T$-field extension $K\langle a \rangle$ with $a \prec 1$ and $a' = s$ such that for any $H_T$-asymptotic field extension $M$ of $K$ with $s \in \der\smallo_M$, there is a unique $\LdO(K)$-embedding $K\langle a \rangle\to M$. The pre-$H_T$-field $K\langle a \rangle$ is grounded with 
\[
\res K\langle a \rangle\ =\ \res K, \qquad \Gamma_{K\langle a \rangle}\ =\ \Gamma \oplus \Lambda va,\qquad \Psi_{K\langle a \rangle}\ =\ \Psi \cup \{va^\dagger\},\qquad va^\dagger\ >\ \Psi.
\]
\end{lemma}
\begin{proof}
By replacing $s$ with $-s$ if need be, we arrange that $s<0$. Let $K\langle a \rangle$ be a simple $\TO$-extension of $K$ where $a> 0$ and $0<va<\Gamma^>$. The Wilkie inequality gives $\Gamma_{K\langle a \rangle}=\Gamma \oplus \Lambda va$ and $\res K\langle a \rangle =\res K$. Using Fact~\ref{fact:transext}, we equip $K\langle a \rangle$ with the unique $T$-derivation that extends the derivation on $K$ and satisfies $a' = s$. We need to show that $K\langle a \rangle$ is a pre-$H_T$-field extension of $K$. Let $g \in K^\times$ and $\lambda \in \Lambda$ with $ga^\lambda \succ 1$. By Lemma~\ref{lem:simprehext}, it suffices to verify the following:
\begin{enumerate}
\item[(i)] $(ga^\lambda)^\dagger >0$;
\item[(ii)] $(ga^\lambda)^\dagger \succ f'$ for all $f \in K$ with $f \preceq 1$;
\item[(iii)] $(ga^\lambda)^\dagger \succ F(a)'$ for all $\cL(K)$-definable functions $F\colon K\to K$ with $F(a) \prec 1$ and $F(a) \not\in K$.
\end{enumerate}
Since $v(ga^\lambda) = vg+\lambda va$ is assumed to be negative and since $0<va<\Gamma^>$, it must be that either $g \succ 1$ or $g \asymp 1$ and $\lambda< 0$. If $g \succ 1$, then $vg^\dagger \in (\Gamma^<)'$ and $va^\dagger = vs-va > (\Gamma^<)'$, so 
\[
(ga^\lambda)^\dagger\ =\ g^\dagger+\lambda a^\dagger\ \sim\ g^\dagger\ >\ 0.
\]
On the other hand, if $g\asymp 1$ and $\lambda > 0$, then $g^\dagger\asymp g' \preceq s\prec s/a = a^\dagger$. This gives
\[
(ga^\lambda)^\dagger\ =\ g^\dagger+\lambda a^\dagger\ \sim\ \lambda a^\dagger,
\]
since $\lambda \asymp 1$. Since $a^\dagger <0$, we have $\lambda a^\dagger>0$. This takes care of (i) and also tells us that $(ga^\lambda)^\dagger\succeq a^\dagger$. This can be used to quickly take care of (ii): if $f \in \cO$, then $f'\preceq s\prec a^\dagger\preceq (ga^\lambda)^\dagger$.

Now we turn to (iii). Let $F\colon K\to K$ with $F(a) \prec 1$ and $F(a) \not\in K$. We need to show that $F(a)' \prec a^\dagger = s/a$. We consider two cases. First, suppose $\cO = K$. Take an $\cL(\emptyset)$-definable function $G\colon K^{1+n}\to K$ and an $\cL(\emptyset)$-independent tuple $b = (b_1,\ldots,b_n)\in K^n$ with $F(a) = G(a,b)$. Then 
\[
F(a)'\ =\ G(a,b)' \ =\ \nabla G(a,b)\cdot (s,b_1',\ldots,b_n'),
\]
so by applying Corollary~\ref{cor:wholethingsmall} with $(s,b_1',\ldots,b_n')\in K^{1+n}$ in place of $d$, we get $F(a)' \prec a\inv$. Since $s\asymp 1$, this gives $F(a)' \prec s/a$, as desired. Now suppose $\cO \neq K$. We need to show that $F^{[\der]}(a)+F'(a)s \prec s/a$. Proposition~\ref{prop:smallderivsim} gives $F'(a)\preceq a\inv F(a) \prec a\inv$, so $F'(a)s \prec s/a$ and it remains to show that $F^{[\der]}(a) \prec s/a$. Since $K\langle a \rangle$ is an elementary $\TO$-extension of $K$, it suffices to show that for each $\LO(K)$-definable set $A \subseteq K$ with $a \in A^{K\langle a \rangle}$, there is $y \in A$ with $F^{[\der]}(y) \prec s/y$. Let $A$ be such a set and, by shrinking $A$ if need be, assume that $F$ is $\cC^1$ on $A$ and that $y,F(y) \prec 1$ for all $y \in A$. Since $F'(a)\prec a\inv$, we can use $\LO$-elementarity to take $y \in A$ with $F'(y) \prec y\inv$. Multiplying by $y'$ gives $F'(y)y' \prec y^\dagger$ for this $y$. Since $F(y) \prec 1$ and $K$ is a pre-$H_T$-field, we have $F(y)' \prec y^\dagger$. Thus,
\[
F^{[\der]}(y) \ =\ F(y)' - F'(y)y' \ \prec \ y^\dagger.
\]
Since $y \prec 1$ and $vs$ is a gap in $K$, we have $y' \prec s$, so $F^{[\der]}(y)\prec y^\dagger \prec s/y$, as desired.

Finally, let $M$ be an $H_T$-asymptotic field extension of $K$ and let $b \in \smallo_M$ with $b' = s$. Then $b^\dagger = s/b$ must be negative by (HA1), so $b$ is positive since $s$ is negative. Moreover, $vb$ must realize the cut $\Gamma^{\leq}$ since $vs \in (\Gamma_M^>)'$ and $vs<(\Gamma^>)'$. Lemma~\ref{lem:uniqueS} gives a unique $\LO(K)$-embedding $\imath \colon K\langle a \rangle \to M$ sending $a$ to $b$ and Fact~\ref{fact:transext} tells us that $\imath$ is an $\LdO(K)$-embedding. Let $\jmath\colon K\langle a\rangle \to M$ be an arbitrary $\LdO(K)$-embedding. Then $\jmath(a) - b \in C_M$ since $\jmath(a)' = s = b'$. Since $\jmath(a),\, b \prec 1$, we see that $\jmath(a) = b$. This shows that $\jmath = \imath$, so $\imath$ is unique.
\end{proof}

\noindent
If we further assume that $K$ is an $H_T$-field, then one can find an $H_T$-field extension of $K$ with an infinite integral for a gap in $K$. 

\begin{lemma}\label{lem:gapgoesdown}
Let $K$ be an $H_T$-field, let $s \in K$, and suppose $vs$ is a gap in $K$. Then $K$ has an $H_T$-field extension $K\langle a \rangle$ with $a\succ 1$ and $a' = s$ such that for any $H_T$-asymptotic field extension $M$ of $K$ and $b \in M$ with $b \succ 1$ and $b' = s$, there is a unique $\LdO(K)$-embedding $K\langle a \rangle \to M$ sending $a$ to $b$. The $H_T$-field $K\langle a \rangle$ is grounded with 
\[
\res K\langle a \rangle\ =\ \res K, \qquad \Gamma_{K\langle a \rangle}\ =\ \Gamma \oplus \Lambda va,\qquad \Psi_{K\langle a \rangle}\ =\ \Psi \cup \{va^\dagger\},\qquad va^\dagger\ >\ \Psi.
\]
\end{lemma}
\begin{proof}
We may assume that $s>0$. Let $K\langle a \rangle$ be a simple $\TO$-extension of $K$ where $a>0$ and $\Gamma^<<va<0$, so $\Gamma_{K\langle a \rangle}=\Gamma \oplus \Lambda va$ and $\res K\langle a \rangle =\res K$ by the Wilkie inequality. Using Fact~\ref{fact:transext}, we equip $K\langle a \rangle$ with the unique $T$-derivation that extends the derivation on $K$ and satisfies $a' = s$. To see that $K\langle a \rangle$ is an $H_T$-field extension of $K$, let $g \in K^\times$ and $\lambda \in \Lambda$ with $ga^\lambda \succ 1$. By Lemma~\ref{lem:simprehext}, it suffices to verify the following:
\begin{enumerate}
\item[(i)] $(ga^\lambda)^\dagger >0$;
\item[(ii)] $(ga^\lambda)^\dagger \succ f'$ for all $f \in K$ with $f \preceq 1$;
\item[(iii)] $(ga^\lambda)^\dagger \succ F(a)'$ for all $\cL(K)$-definable functions $F\colon K\to K$ with $F(a) \prec 1$ and $F(a)\not\in K$.
\end{enumerate}
Proving (i) is similar to the proof of Lemma~\ref{lem:gapgoesup}. This time, either $g \succ 1$ or $g \asymp 1$ and $\lambda> 0$. If $g \succ 1$, then $g^\dagger\succ a^\dagger$, so $(ga^\lambda)^\dagger\sim g^\dagger > 0$. Suppose $g \asymp 1$. We need to show that $g^\dagger \prec a^\dagger$. Using that $K$ is an $H_T$-field, take $c \in C$ with $g- c \in \smallo$. Then $g^\dagger \asymp g' = (g-c)' \in \der\smallo$, so $v(g^\dagger) \in (\Gamma^>)'$. Since $va^\dagger <(\Gamma^>)'$, we have $(ga^\lambda)^\dagger \sim \lambda a^\dagger > 0$. This takes care of (i) and tells us that $(ga^\lambda)^\dagger\succeq a^\dagger$. For (ii), let $f \in \cO$ and take $c \in C$ with $f-c \in \smallo$. Then $f' =(f-c)'$, so $v(f') \in (\Gamma^>)'>va^\dagger\geq v(ga^\lambda)^\dagger$.

Now we turn to (iii). Let $F\colon K\to K$ with $F(a) \prec 1$ and $F(a) \not\in K$. As in the proof of Lemma~\ref{lem:gapgoesup}, we need to show that $F(a)' =F^{[\der]}(a)+F'(a)s \prec s/a$. Proposition~\ref{prop:smallderivsim} gives $F'(a)\preceq a\inv F(a) \prec a\inv$, so $F'(a)s \prec s/a$ and it remains to show that $F^{[\der]}(a) \prec s/a$. We claim that $|F^{[\der]}(a)| < s/a^2\prec s/a$. Since $F^{[\der]}$ is $\cL(K)$-definable, it suffices to show that for each interval $I \subseteq K$ with $a \in I^{K\langle a \rangle}$, there is $y \in I$ with $|F^{[\der]}(y)| < s/y^2$. Let $I$ be such an interval and, by shrinking $I$ if need be, assume that $F$ is $\cC^1$ on $I$ and that $|F(y)|<1$ for all $y \in I$. Since $K$ is an $H_T$-field and $a$ realizes the cut $\cO^\downarrow$, the interval $I$ contains a constant $c\in C^>$. Since $|F(c)|<1$ and $s$ is a gap in $K$, we have
\[
F(c)' \ =\ F^{[\der]}(c)+ F'(c)c' \ =\ F^{[\der]}(c)\ \prec\ s.
\]
Since $c^2 \asymp 1$, we have $c^2F^{[\der]}(c) \prec s$, which yields $|F^{[\der]}(c)|< s/c^2$, as desired.

Finally, let $M$ be an $H_T$-asymptotic field extension of $K$ and let $b \in M$ with $b\succ 1$ and $b' = s$. Then $b^\dagger = s/b$ must be positive, so $b$ is positive since $s$ is positive. Since $vs \in (\Gamma_M^<)'$ and $vs>(\Gamma^<)'$, we see that $vb$ must realize the cut $\Gamma^<$. Lemma~\ref{lem:uniqueS} gives a unique $\LO(K)$-embedding $\imath \colon K\langle a \rangle \to M$ sending $a$ to $b$, and this is even an $\LdO(K)$-embedding by Fact~\ref{fact:transext}.
\end{proof}

\subsection{The $H_T$-field hull}
We now show that the pre-$H_T$-field $K$ has a minimal $H_T$-field extension. We say that $\beta \in \Gamma$ is a \textbf{fake gap in $K$} if $\beta$ is a gap in $K$ and $\beta = v(b')$ for some $b \in K$. Then necessarily $b \asymp 1$, for otherwise $\beta \in (\Gamma^{\neq})'$. Likewise, $b \not\sim c$ for any $c \in C$, for otherwise $b' = (b-c)'$ and $\beta \in (\Gamma^>)'$. Thus, no $H_T$-field has a fake gap. Of course, if $K$ is grounded or has asymptotic integration, then $K$ does not have a fake gap. Suppose $K$ does not have a fake gap and let $M$ be an immediate pre-$H_T$-field extension of $K$. We claim that $M$ does not have a fake gap. Let $b \in M$ with $b \asymp 1$ and take $a \in K$ with $b-a\prec 1$. Then $v(b-a)' \in (\Gamma^>)'$. As $K$ has no fake gap, we also have $v(a') \in (\Gamma^>)'$, so
\[
v(b') \ =\ v\big((b-a)' + a'\big)\ \geq\ \min\big\{v(b-a)',v(a')\big\} \in (\Gamma^>)'.
\]

\begin{theorem}\label{thm:HThull}
$K$ has an $H_T$-field extension $H_T(K)$ such that for any $H_T$-field extension $M$ of $K$, there is a unique $\LdO(K)$-embedding $H_T(K)\to M$. For $L\coloneqq H_T(K)$, we have 
\[
L\ =\ K\langle C_L\rangle,\qquad \res L = \res K.
\]
\end{theorem}
\begin{proof}
We first construct a pre-$H_T$-field extension $K_0$ of $K$ which does not have a fake gap as follows: if $K$ does not have a fake gap, then let $K_0\coloneqq K$. Suppose that $K$ has a fake gap $\beta = v(b')$, and apply Lemma~\ref{lem:gapgoesup} with $s= b'$ to get a pre-$H_T$-field extension $K\langle a \rangle$ of $K$ with $a\prec 1$ and $a' = b'$. Then $K\langle a \rangle$ does not have a fake gap as it is grounded, and we set $K_0\coloneqq K\langle a \rangle$. We claim that $K_0 = K\langle C_{K_0}\rangle$, that $\res K_0 = \res K$, and that for any $H_T$-field extension $M$ of $K$, there is a unique $\LdO(K)$-embedding $K_0\to M$. This is all trivial if $K_0 = K$, so we assume that $K_0 \neq K$ and we let $a,b$ be as above. Then $K_0 = K\langle b-a\rangle$ and $b-a\in C_{K_0}$, so $K_0 = K\langle C_{K_0}\rangle$, and Lemma~\ref{lem:gapgoesup} gives $\res K_0 = \res K$. For the embedding property, take $c \in C_M$ with $b\sim c$. Then $b' = (b-c)' \in \der\smallo_M$, so Lemma~\ref{lem:gapgoesup} gives a unique $\LdO(K)$-embedding $K_0\to M$.

Suppose $K_0$ is not an $H_T$-field, so there is $b \in \cO_{K_0}$ with $b\not\in C_{K_0}+ \smallo_{K_0}$. Then $b' \not\in \der\smallo_{K_0}$, for otherwise we would have $b -\epsilon \in C_{K_0}$ for some $\epsilon \in \smallo_{K_0}$. Since $v(b')$ is not a fake gap, we have $v(b') \in (\Gamma_{K_0}^>)'$. Corollary~\ref{cor:smallint} gives an immediate pre-$H_T$-field extension $K^*\coloneqq K_0\langle a \rangle$ of $K_0$ with $a\prec 1$ and $a' = b'$. Given an $H_T$-field extension $M$ of $K_0$, take $c \in C_M$ with $b\sim c$. Then $b' = (b-c)' \in \der\smallo_M$, so Corollary~\ref{cor:smallint} gives a unique $\LdO(K_0)$-embedding $K^* \to M$. Note that $K^* = K_0\langle b-a\rangle$ and $b-a \in C_{K^*}$, so $K^* = K_0 \langle C_{K^*}\rangle = K\langle C_{K^*}\rangle$. As $K^*$ is an immediate extension of $K_0$, there is no fake gap in $K^*$. By iterating this process, we build an immediate $H_T$-field extension $L$ of $K_0$ such that for any $H_T$-field extension $M$ of $K_0$, there is a unique $\LdO(K_0)$-embedding $L\to M$. Using also the embedding property for $K_0$ over $K$, we see that for any $H_T$-field extension $M$ of $K$, there is a unique $\LdO(K)$-embedding $L\to M$. Moreover, $\res L = \res K_0 = \res K$ and $L = K\langle C_L\rangle$ by construction. Let $H_T(K)$ be this extension $L$.
\end{proof}

\noindent
The universal property in Theorem~\ref{thm:HThull} determines $H_T(K)$ uniquely up to unique $\LdO(K)$-isomorphism. We call $H_T(K)$ the \textbf{$H_T$-field hull of $K$}. If $K$ does not have a fake gap, then $H_T(K)$ is an immediate extension of $K$; in particular, $\Gamma_{H_T(K)} = \Gamma$. If $\beta$ is a fake gap in $K$, then $\Gamma_{H_T(K)} = \Gamma \oplus \Lambda va$ for $a \in H_T(K)$ with $0<va<\Gamma^>$ and $v(a') = \beta$. The following consequence of Theorem~\ref{thm:HThull} is not used anywhere, but it may be worth noting.

\begin{corollary}
The following are equivalent:
\begin{enumerate}
\item every spherically complete immediate $H_T$-asymptotic field extension of $K$ is an $H_T$-field;
\item $K$ has a spherically complete immediate $H_T$-field extension;
\item $K$ does not have a fake gap.
\end{enumerate}
\end{corollary}
\begin{proof}
By Theorem~\ref{thm:asymp}, we know that $K$ has a spherically complete immediate $H_T$-asymptotic field extension, so (1) implies (2). Suppose (2) holds and let $M$ be a spherically complete immediate $H_T$-field extension of $K$. By the universal property of the $H_T$-field hull, there is a unique $\LdO(K)$-embedding $H_T(K) \to M$. Then $H_T(K)$ is an immediate extension of $K$, so $\Gamma_{H_T(K)} = \Gamma$ and $K$ does not have a fake gap by the remarks preceding this corollary. Finally, suppose (3) holds and let $M$ be a spherically complete immediate $H_T$-asymptotic field extension of $K$. Then $M$ is a pre-$H_T$-field by Lemma~\ref{lem:asympisHT}, and the remarks before Theorem~\ref{thm:HThull} tell us that $M$ does not have a fake gap. Thus, $H_T(M)$ is an immediate extension of $M$, so $M = H_T(M)$, as $M$ has no proper immediate extensions.
\end{proof}

\subsection{$\upo$-free extensions of grounded pre-$H_T$-fields}
In this subsection, we show that each grounded pre-$H_T$-field $K$ has a canonical ungrounded $\upo$-free extension, denoted $K_{\upo}$. First, we show how to extend a grounded pre-$H_T$-field by an integral for the maximum of the set $\Psi$. 

\begin{lemma}\label{lem:maxgoesdown}
Let $s \in K$ and suppose $vs = \max \Psi$. Then $K$ has a pre-$H_T$-field extension $K\langle a \rangle$ with $a' = s$ such that for any pre-$H_T$-field extension $M$ of $K$ and $b \in M$ with $b' = s$, there is a unique $\LdO(K)$-embedding $K\langle a \rangle \to M$ sending $a$ to $b$. The pre-$H_T$-field $K\langle a \rangle$ is grounded with
\[
\res K\langle a \rangle\ =\ \res K, \qquad \Gamma_{K\langle a \rangle}\ =\ \Gamma \oplus \Lambda va,\qquad \Psi_{K\langle a \rangle}\ =\ \Psi \cup \{va^\dagger\},\qquad va^\dagger\ >\ \Psi.
\]
\end{lemma}
\begin{proof}
Let $L = H_T(K)$, so $L$ is an immediate extension of $K$ and $vs = \max \Psi_L$, since $K$ is grounded. Then $L$ is an $H_T$-field, and a proof identical to the proof of Lemma~\ref{lem:gapgoesdown} tells us that $L$ has a grounded $H_T$-field extension $L\langle a \rangle$ with $a \succ 1$ and $a' = s$, where
\[
\res L\langle a \rangle\ =\ \res L, \qquad \Gamma_{L\langle a \rangle}\ =\ \Gamma_L \oplus \Lambda va,\qquad \Psi_{L\langle a \rangle}\ =\ \Psi_L \cup \{va^\dagger\},\qquad va^\dagger\ >\ \Psi_L.
\]
Then $K\langle a \rangle$, being an $\LdO$-substructure of the $H_T$-field $L\langle a \rangle$, is a pre-$H_T$-field. Since $L$ is an immediate extension of $K$, we have $\res L = \res K$, $\Gamma_L = \Gamma$, and $\Psi_L = \Psi$, so 
\[
\res K\langle a \rangle\ =\ \res K, \qquad \Gamma_{K\langle a \rangle}\ =\ \Gamma \oplus \Lambda va,\qquad \Psi_{K\langle a \rangle}\ =\ \Psi \cup \{va^\dagger\},\qquad va^\dagger\ >\ \Psi.
\]
Since $a \succ 1$, we get that $va$ realizes the cut $\Gamma^<$ and $a^\dagger = s/a$ is positive. Let $M$ be a pre-$H_T$-field extension of $K$ and let $b \in M$ with $b' = s$. Then $b \succ 1$ by (PH2), so $vb$ must realize the cut $\Gamma^<$ and $b^\dagger = s/b$ must be positive. Lemma~\ref{lem:uniqueS} gives a unique $\LO(K)$-embedding $\imath \colon K\langle a \rangle \to M$ sending $a$ to $b$, and this is even an $\LdO(K)$-embedding by Fact~\ref{fact:transext}.
\end{proof}

\begin{remark}
Our passage though the $H_T$-field hull of $K$ is admittedly circuitous, but we don't have a better argument that $K\langle a \rangle$ is a pre-$H_T$-field than the one given above. 

One may wonder whether the assumption that $K$ is an $H_T$-field can be relaxed in Lemma~\ref{lem:gapgoesdown}. That is, if $vs$ is a gap in the pre-$H_T$-field $K$ with $s \in K$, then does $K$ have a pre-$H_T$-field extension $K\langle a \rangle$ with $a \succ 1$ and $a' = s$? Arguing as in the proof of Lemma~\ref{lem:maxgoesdown}, we see that this is true so long as $vs$ is not a fake gap in $K$. If $vs$ is a fake gap in $K$, then $s$ can not have an infinite integral, so there is no such extension.

The version of Lemma~\ref{lem:maxgoesdown} in the author's thesis~\cite[Lemma 7.34]{Ka21} contains a slight error. There, it was claimed that $K\langle a \rangle$ embeds into any $H_T$-asymptotic field extension of $K$. It is only true that $K\langle a \rangle$ embeds into any pre-$H_T$-field extension of $K$.
\end{remark}

\noindent
We now turn to the construction of $K_{\upo}$. 

\begin{corollary}\label{cor:omegaconstruction}
Let $K$ be a grounded pre-$H_T$-field. Then $K$ has an ungrounded $\upo$-free (hence, $\upl$-free) pre-$H_T$-field extension $K_{\upo}$ with $\res K_{\upo} = \res K$ which embeds over $K$ into any pre-$H_T$-field extension of $K$ which is closed under taking $\d$-logarithms. 
\end{corollary}
\begin{proof}
Let $s \in K$ with $vs^\dagger = \max\Psi$. Using Lemma~\ref{lem:maxgoesdown}, take a pre-$H_T$-field extension $K\langle a \rangle$ where $a' = s^\dagger$. We have
\[
\res K\langle a \rangle\ =\ \res K, \qquad \Gamma_{K\langle a \rangle}\ =\ \Gamma \oplus \Lambda va,\qquad \Psi_{K\langle a \rangle}\ =\ \Psi \cup \{va^\dagger\},\qquad va^\dagger\ >\ \Psi.
\]
Repeating this process, we construct for each $n$ a pre-$H_T$-field extension $K_n$ of $K$ with 
\[
K_0\ =\ K,\qquad K_{n+1}\ =\ K_n\langle a_n\rangle,\qquad a_0\ =\ a,\qquad a_{n+1}'\ =\ a_n^\dagger.
\]
Set $K_{\upo} \coloneqq \bigcup_n K_n$. Then $\res K_{\upo} = \res K$ and
\[
\Gamma_{K_{\upo}}\ =\ \Gamma \oplus \bigoplus_n\Lambda va_n,\qquad \Psi_{K_{\upo}}\ =\ \Psi \cup \{va_0^\dagger,va_1^\dagger,\ldots\},\qquad \Psi\ <\ va_0^\dagger\ <\ va_1^\dagger \ <\ \cdots.
\]
Moreover, $K_{\upo}$ is $\upo$-free by~\cite[Corollary 11.7.15]{ADH17}, since $K_{\upo}$ is ungrounded and each $K_n$ is grounded. Let $M$ be a pre-$H_T$-field extension of $K$ which is closed under taking $\d$-logarithms. Then there are elements $b_0,b_1,\ldots \in M$ with $b_0' = s^\dagger$ and $b_{n+1}' = b_n^\dagger$ for each $n$. Repeated use of the embedding property in Lemma~\ref{lem:maxgoesdown} allows us construct an $\LdO(K)$-embedding $K_{\upo}\to M$ which sends $a_n$ to $b_n$ for each $n$.
\end{proof}

\subsection{Adjoining exponential integrals}
\begin{lemma}\label{lem:smallexpint}
Let $s \in K$ with $vs \in (\Gamma^>)'$ and suppose that $s \neq y^\dagger$ for all $y \in K^\times$. Then $K$ has an immediate pre-$H_T$-field extension $K\langle a \rangle$ with $a\sim 1$ and $a^\dagger = s$ such that for any $H_T$-asymptotic field extension $M$ of $K$ with $s \in (1+\smallo_M)^\dagger$, there is a unique $\LdO(K)$-embedding $K\langle a \rangle\to M$. If $K$ is ungrounded and $\upl$-free, then so is $K\langle a \rangle$.
\end{lemma}
\begin{proof}
Let $S\coloneqq v\big(s-(1+\smallo)^\dagger\big)\subseteq (\Gamma^>)'$. By the proof of~\cite[Lemma 10.4.3]{ADH17}, $S$ has no largest element. Let $(a_\rho)$ be a well-indexed sequence in $1+\smallo$ such that $v(s-a_\rho^\dagger)$ is strictly increasing in $S$ as a function of $\rho$. Then $(a_\rho)$ is a divergent pc-sequence in $K$, again by the proof of~\cite[Lemma 10.4.3]{ADH17}. We apply Proposition~\ref{prop:HT-simimm} with $G(Y) = sY$ to get an immediate pre-$H_T$-field extension $K\langle a \rangle$ of $K$ with $a_\rho \leadsto a$ and $a' = sa$. Note that $a \sim 1$, since each $a_\rho \sim 1$. Let $M$ be an $H_T$-asymptotic field extension of $K$ and let $b \in M$ with $b \sim 1$ and $b^\dagger = s$. Then $a_\rho^\dagger\leadsto b^\dagger$, and so $a_\rho \leadsto b$ by the proof of~\cite[Lemma 10.4.3]{ADH17}. Proposition~\ref{prop:HT-simimm} gives an $\LdO(K)$-embedding $\imath \colon K\langle a \rangle \to M$ that sends $a$ to $b$. For uniqueness, let $\jmath\colon K\langle a\rangle \to M$ be an arbitrary $\LdO(K)$-embedding. Then $\jmath(a)/b \in C_M^\times$ since $\jmath(a)^\dagger = s = b^\dagger$. Since $\jmath(a)\sim 1 \sim b$, we see that $\jmath(a)= b$, so $\jmath = \imath$. By~\cite[Lemma 7.6]{Ge17C}, the set $v(a-K)$ has the yardstick property, so if $K$ is ungrounded and $\upl$-free, then $K\langle a \rangle$ is as well by Lemma~\ref{lem:yardstickpreserve}.
\end{proof}

\begin{lemma}\label{lem:bigexpint}
Let $s\in K$ with $v\big(s-(K^\times)^\dagger\big)\subseteq \Psi^\downarrow$. Then $K$ has a pre-$H_T$-field extension $K\langle a \rangle$ with $a>0$ and $a^\dagger = s$ such that for any pre-$H_T$-field extension $M$ of $K$ and $b \in M^>$ with $b^\dagger= s$, there is a unique $\LdO(K)$-embedding $K\langle a \rangle\to M$ sending $a$ to $b$. Moreover, the extension $K\langle a \rangle$ has the following properties:
\begin{enumerate}
\item $va \not\in \Gamma$ and $\Gamma_{K\langle a \rangle}=\Gamma \oplus \Lambda va$;
\item $\res K\langle a \rangle = \res K$;
\item $\Psi$ is cofinal in $\Psi_{K\langle a \rangle}$;
\item a gap in $K$ remains a gap in $K\langle a \rangle$;
\item if $K$ is ungrounded and $\upl$-free, then so is $K\langle a \rangle$.
\end{enumerate}
\end{lemma}
\begin{proof}
Suppose that $b$ is an element in a pre-$H_T$-field extension $M$ of $K$ with $b> 0$ and $b^\dagger = s$. Then $vb \not\in \Gamma$; otherwise there is $f \in K$ and $u \in \cO_M^\times$ with $b/f = u$, so $s -f^\dagger= u^\dagger \asymp u'$ and $v(s-f^\dagger)> \Psi$, a contradiction. Let $y \in K^\times$ with $y \prec b$. Then $y/b\prec 1$ so $y^\dagger <b^\dagger = s$. Likewise if $y \in K^\times$ with $y \succ b$, then $y^\dagger > s$. Thus, $vb$ realizes the cut
\[
S\ \coloneqq \ \{vy:y^\dagger>s\}\ \subseteq \Gamma.
\]
Let $K\langle a \rangle$ be a simple $\TO$-extension of $K$ where $a> 0$ and $va$ realizes the cut $S$.
The Wilkie inequality gives $\Gamma_{K\langle a \rangle}=\Gamma \oplus \Lambda va$ and $\res K\langle a \rangle =\res K$. Using Fact~\ref{fact:transext}, we equip $K\langle a \rangle$ with the unique $T$-derivation which extends the derivation on $K$ and satisfies $a^\dagger = s$. If we can show that $K\langle a \rangle$ is a pre-$H_T$-field, then the embedding property of $K\langle a \rangle$ follows from Fact~\ref{fact:transext}, Lemma~\ref{lem:uniqueS} and the discussion above.

To see that $K\langle a \rangle$ is a pre-$H_T$-field, let $g \in K^\times$ and $\lambda \in \Lambda$ with $ga^\lambda \succ 1$. By Lemma~\ref{lem:simprehext}, it suffices to verify the following:
\begin{enumerate}
\item[(i)] $(ga^\lambda)^\dagger >0$;
\item[(ii)] $(ga^\lambda)^\dagger \succ f'$ for all $f \in K$ with $f \preceq 1$;
\item[(iii)] $(ga^\lambda)^\dagger \succ F(a)'$ for all $\cL(K)$-definable functions $F\colon K\to K$ with $F(a) \prec 1$ and $F(a) \not\in K$.
\end{enumerate}
First we deal with (i). If $\lambda = 0$, then $g \succ 1$ and $(ga^\lambda)^\dagger = g^\dagger > 0$. If $\lambda > 0$, then since $ga^\lambda \succ 1$ we have $a\succ g^{-\lambda\inv}$, so $s> (g^{-\lambda\inv})^\dagger = - \lambda\inv g^\dagger$. This gives $\lambda\inv g^\dagger+s > 0$, so
\[
(ga^\lambda)^\dagger \ =\ g^\dagger + \lambda s\ =\ \lambda(\lambda\inv g^\dagger+s)\ >\ 0.
\]
On the other hand, if $\lambda < 0$, then $a \prec g^{-\lambda\inv}$, so $\lambda\inv g^\dagger+s <0$ and again, 
\[
(ga^\lambda)^\dagger\ =\ \lambda(\lambda\inv g^\dagger+s)\ >\ 0.
\]
Note that since 
\[
(ga^\lambda)^\dagger\ =\ \lambda(\lambda\inv g^\dagger+s)\ =\ \lambda\big(s-(g^{-\lambda\inv})^\dagger\big)\ \asymp\ s-(g^{-\lambda\inv})^\dagger,
\]
we have $v(ga^\lambda)^\dagger\in \Psi^\downarrow$. Thus, for (ii) it suffices to note that $h^\dagger \succ f'$ for all $f,h \in K$, since $K$ is a pre-$H_T$-field. Likewise, for (iii), it suffices to show that $h^\dagger \succ F(a)'$ for all $h \in K$ with $h \succ 1$ and all $\cL(K)$-definable functions $F\colon K\to K$ with $F(a) \prec 1$ and $F(a) \not\in K$. Suppose toward contradiction that there are $F$, $h$ for which this does not hold, so $F(a) \prec 1$ but 
\[
F(a)'\ =\ F^{[\der]}(a) + F'(a)as\ \succeq\ h^\dagger.
\]
By replacing $F$ with $-F$ if necessary, we may assume that $F(a)' > 0$. Since $\res K\langle a \rangle =\res K$, we have $F^{[\der]}(a) + F'(a) as > uh^\dagger>0$ for some $u \in K$ with $u\asymp 1$. Take an interval $I \subseteq K^>$ with $a \in I^{K\langle a \rangle}$ such that
\[
|F(y)|\ <\ 1,\qquad F^{[\der]}(y) + F'(y) ys \ >\ uh^\dagger
\]
for all $y \in I$. For $y \in I$, we have $F(y) \preceq 1$, so $F(y)' = F^{[\der]}(y)+ F'(y)y' \prec h^\dagger$ since $K$ is a pre-$H_T$-field. This gives
\[
\big(F^{[\der]}(y) + F'(y) ys\big)-\big( F^{[\der]}(y)+ F'(y)y'\big)\ =\ F'(y)y(s-y^\dagger) \ > \ \frac{1}{2}uh^\dagger \ > \ 0.
\]
In particular, the function $F'(y)y(s-y^\dagger)$ has constant sign on $I$. By shrinking $I$, we may assume that $F'(y)$ and $y$ have constant sign on $I$, so $s-y^\dagger$ has constant sign on $I$ as well. This is a contradiction: if $y \in I$ is greater than $a$, then $y \succ a$ and $y^\dagger > s$ and if $y \in I$ is less than $a$, then $y \prec a$ and $y ^\dagger < s$.

Now that we know that $K\langle a \rangle$ is a pre-$H_T$-field extension of $K$ with the required embedding property, all that remains is to check that $K\langle a \rangle$ satisfies properties (1)--(5). We have already verified properties (1) and (2). For (3), let $h \in K\langle a \rangle$ with $h \succ 1$. Then $h \asymp ga^\lambda$ for some $g\in K$ and some $\lambda \in \Lambda$ by (1), so $v(h^\dagger )=v(ga^\lambda)^\dagger$, since $K\langle a \rangle$ is $H_T$-asymptotic. We have already shown that $v(ga^\lambda)^\dagger\in \Psi^\downarrow$, so $v(h^\dagger )\in \Psi^\downarrow$ as well. As for (4), let $\beta$ be a gap in $K$. Then $\beta >\Psi$, so $\beta >\Psi_{K\langle a \rangle}$ since $\Psi$ is cofinal in $\Psi_{K\langle a \rangle}$. Suppose toward contradiction that $\beta$ is not a gap in $K\langle a \rangle$, so $\beta = \alpha'$ for some $\alpha \in \Gamma_{K\langle a \rangle}^>$. The universal property in Lemma~\ref{lem:gapgoesup} gives $\max \Psi_{K\langle a \rangle} = \alpha^\dagger >\Psi$, contradicting that $\Psi$ is cofinal in $\Psi_{K\langle a \rangle}$.

Finally, suppose $K$ is ungrounded and let $(\ell_\rho)$ be a logarithmic sequence in $K$ with corresponding $\upl$-sequence $(\upl_\rho)$. Then $K{\langle a \rangle}$ is ungrounded since $\Psi$ is cofinal in $\Psi_{K\langle a \rangle}$. It follows that $\Gamma^<$ is cofinal in $\Gamma_{K\langle a \rangle}^<$. To see this, let $f \in K\langle a \rangle$ with $f \succ 1$ and suppose toward contradiction that $f \prec g$ for all $g \in K$ with $g \succ 1$. Then $f^\dagger \preceq g^\dagger$ for such $g$, so $v(f^\dagger)> \Psi$ since $K$ is ungrounded, contradicting (3). Therefore, $(\ell_\rho)$ remains a logarithmic sequence in $K\langle a \rangle$ and $(\upl_\rho)$ remains a $\upl$-sequence in $K\langle a \rangle$. Suppose toward contradiction that $K$ is $\upl$-free and that $\upl_\rho\leadsto \upl\in K\langle a \rangle$. Then $\upl\not\in K$, so $a \in K\langle \upl\rangle$. This is a contradiction, as $K\langle \upl\rangle$ is an immediate $\TO$-extension of $K$ by Lemma~\ref{lem:immediateext2} and $va \not\in \Gamma$. This proves (5).
\end{proof}

\noindent
We can use Lemma~\ref{lem:bigexpint} along with Lemma~\ref{lem:uplgapcreator} and Corollary~\ref{cor:asympHT} to prove the following extension result for pre-$H_T$-fields.

\begin{proposition}\label{prop:upofreeext}
$K$ has an ungrounded $\upo$-free $H_T$-field extension.
\end{proposition}
\begin{proof}
By passing to the $H_T$-field hull of $K$, we may assume that $K$ is an $H_T$-field. First, we will show that every $H_T$-field with asymptotic integration has an $H_T$-field extension with a gap. Then, we will show that every $H_T$-field with a gap has a grounded $H_T$-field extension. Finally, we will show that every grounded $H_T$-field has an ungrounded $\upo$-free $H_T$-field extension.

For the first part, suppose $K$ has asymptotic integration. As having asymptotic integration is a property of the asymptotic couple of $K$, every immediate extension of $K$ also has asymptotic integration. By applying Corollary~\ref{cor:asympHT} we can pass to a spherically complete extension of $K$, so we may assume that $K$ is spherically complete. Let $(\upl_\rho)$ be a $\upl$-sequence in $K$, so $(\upl_\rho)$ has a pseudolimit $\upl \in K$ by spherical completeness. The set $v\big(\upl+(K^\times)^\dagger\big)$ is a cofinal subset of $\Psi^\downarrow$ by Fact~\ref{fact:uplminusdaggers}, so Lemma~\ref{lem:bigexpint} gives an $H_T$-field extension $K\langle a \rangle$ of $K$ with $a^\dagger = -\upl$. This extension has a gap, namely $va$, by Lemma~\ref{lem:uplgapcreator}.

Now, assume that $K$ has a gap $\beta \in \Gamma$. Take $s \in K$ with $vs = \beta$ and use Lemma~\ref{lem:gapgoesup} to get a grounded $H_T$-field extension $K\langle a \rangle$ of $K$ with $a' = s$. Finally, if $K$ is grounded, apply Corollary~\ref{cor:omegaconstruction} to get an ungrounded $\upo$-free $H_T$-field extension $K_{\upo}$ of $K$.
\end{proof}

\subsection{Constant field extensions}
Recall that the constant field $C$ is naturally a model of $T$. In the following proposition, we show that if $K$ is an $H_T$-field, then we may take extensions of $K$ by constants corresponding to $T$-extensions of $C$. 

\begin{proposition}
\label{prop:constext}
Let $K$ be an $H_T$-field and let $E$ be a $T$-extension of $C$. Then there is an $H_T$-field extension $L$ of $K$ where $C_L$ is $\cL(C)$-isomorphic to $E$ such that for any $H_T$-field extension $M$ of $K$ and any $\cL(C)$-embedding $\imath\colon C_L\to C_M$, there is a unique $\LdO(K)$-embedding $L\to M$ extending $\imath$.
\end{proposition}
\begin{proof}
It suffices to consider the case $E = C\langle f \rangle$ where $f \not\in C$. Let $L = K\langle a\rangle$ be a simple $T$-extension of $K$ where $a$ realizes the cut
\[
(C^{<f}+\smallo)^\downarrow \ =\ \{y \in K: y<\cO\}\cup \{c+\epsilon :c\in C^{<f}\text{ and }\epsilon \in \smallo\}.
\]
We expand $L$ to an $\LO$-structure by letting
\[
\cO_L\ \coloneqq \ \big\{y\in L:|y|<d\text{ for all }d \in K\text{ with } d >\cO\big\}.
\]
This expansion of $L$ is a $\TO$-extension of $K$ by Fact~\ref{fact:tworings}. Note that $a \in \cO_L$ and $\bar{a} \not\in \res K$, so the Wilkie inequality gives $\Gamma_L = \Gamma$. Using Fact~\ref{fact:transext}, we extend $\der$ uniquely to a $T$-derivation on $L$ with $a' = 0$. 

We claim that $L$ is an $H_T$-field extension of $K$. We may assume that $\cO \neq K$, since otherwise, $K$ and $L$ both have trivial valuation and derivation, so $L$ is trivially an $H_T$-field extension of $K$. To see that $L$ satisfies (H1), let $F\colon K\to K$ be an $\cL(K)$-definable function with $F(a) >\cO_L$. We need to show 
\[
F(a)' \ =\ F^{[\der]}(a)+ F'(a)a' \ =\ F^{[\der]}(a)\ >\ 0.
\]
As $F^{[\der]}$ is $\cL(K)$-definable, it suffices to show that for any subinterval $I \subseteq K$ with $a \in I^L$, there is $y \in I$ with $F^{[\der]}(y)>0$. Let $I$ be such a subinterval. Using that $\Gamma_L = \Gamma$ and that $\Gamma^>$ has no least element, take $d \in K$ with $F(a)> d>\cO_L$. By shrinking $I$, we arrange that $F(y)> d$ for all $y \in I$. Since $\bar{a} \in \overbar{I}^{\res L}$, we see that $\overbar{I}$ must be infinite. Thus, $I \cap C$ is infinite, so take $c \in I\cap C$. As $c \in C$, we have $F(c)' = F^{[\der]}(c)$. As $c \in I$, we have $F(c)>d>\cO$, so $F(c)' = F^{[\der]}(c)>0$, as desired. Now, let us show that (H2) holds. By (H1), we have $C_L \subseteq \cO_L$. Clearly, $C_L$ contains $C\langle a \rangle$. Since $C\langle a \rangle$ is a lift of $\res L$, it is maximal among the elementary $\cL$-substructures of $L$ contained in $\cO_L$, so $C\langle a \rangle = C_L$ and $\cO_L = C_L + \smallo_L$; see~\cite[Remark 2.11 and Theorem 2.12]{DL95}. This completes the proof that $L$ is an $H_T$-field. It also tells us that $C_L$ is $\cL(C)$-isomorphic to $E$.

Given an $H_T$-field extension $M$ of $K$ and an $\cL(C)$-embedding $\imath\colon C_L\to M$, there is at most one possible $\LdO(K)$-embedding $\jmath\colon L \to M$ which extends $\imath$, namely the one which sends $a$ to $\imath(a)$. Let us show that this is actually an $\LdO(K)$-embedding. By assumption, $a$ and $\imath(a)$ realize the same cut over $C$. Since $a-c\not\in \smallo_L$ and $\imath(a)-c\not\in \smallo_M$ for all $c \in C$, this assumption gives that $a$ and $\imath(a)$ realize the same cut over $\cO$. As $a \in \cO_L$ and $\imath(a) \in \cO_M$, we see that $a$ and $\imath(a)$ realize the same cut over $K$, so $\jmath$ is an $\cL(K)$-embedding. Fact~\ref{fact:transext} ensures that $\jmath$ is an $\Ld(K)$-embedding. To see that $\jmath$ is an $\LdO(K)$-embedding, let $f \in K\langle a\rangle$. If $f \in \cO_L$, then $|f|<c$ for some $c\in C_L$, so $|\jmath(f)|<\imath(c) \in C_M$, which gives $\jmath(f) \in \cO_M$. Conversely, if $f \not\in \cO_L$, then $|f|>d$ for some $d \in K$ with $d>\cO$, so $|\jmath(f)|> d$, which gives $\jmath(f)\not\in \cO_M$.
\end{proof}

\section{Liouville closed $H_T$-fields}\label{sec:Liouville}
\noindent
In this section, we assume that $K$ is an $H_T$-field. For now, we drop the assumption that $T$ is power bounded (though we will re-introduce this assumption at the beginning of Subsection~\ref{subsec:towers}). Recall that $K$ is \emph{Liouville closed} if for each $y \in K$, there is $f\in K$ and $g \in K^\times$ with $f'=g^\dagger = y$.

\begin{definition}
A \textbf{$T$-Liouville extension} of $K$ is an $H_T$-field extension $L$ of $K$ where
\begin{enumerate}
\item $C_L = C$, and 
\item each $a \in L$ is contained in an $H_T$-subfield $K\langle t_1,\ldots,t_n\rangle\subseteq L$ where for $i=1,\ldots,n$, either $t_i' \in K\langle t_1,\ldots,t_{i-1}\rangle$ or $t_i\neq 0$ and $t_i^\dagger \in K\langle t_1,\ldots,t_{i-1}\rangle$.
\end{enumerate}
\end{definition}

\noindent
Below we list some easily verified facts about $T$-Liouville extensions of $K$.
\begin{fact}\label{fact:Lexts}\
\begin{enumerate}
\item If $L$ is a $T$-Liouville extension of $K$ and $M$ is a $T$-Liouville extension of $L$, then $M$ is a $T$-Liouville extension of $K$.
\item If $M$ is a $T$-Liouville extension of $K$ and $L$ is an $H_T$-field extension of $K$ contained in $M$, then $M$ is a $T$-Liouville extension of $L$.
\item If $(L_i)_{i\in I}$ is an increasing chain of $T$-Liouville extensions of $K$, then the union $\bigcup_{i\in I}L_i$ is a $T$-Liouville extension of $K$.
\item Every $T$-Liouville extension of $K$ has the same cardinality as $K$.
\end{enumerate}
\end{fact}

\begin{lemma}
\label{lem:closedimpliesmaximal}
Suppose $K$ is Liouville closed. Then $K$ has no proper $T$-Liouville extensions.
\end{lemma}
\begin{proof}
Let $L$ be a $T$-Liouville extension of $K$, and let $a \in L$. We will show that $a \in K$. By definition, $a$ is contained in an $H_T$-subfield $K\langle t_1,\ldots,t_n\rangle\subseteq L$ where for $i =1,\ldots,n$, either $t_i' \in K\langle t_1,\ldots,t_{i-1}\rangle$ or $t_i^\dagger \in K\langle t_1,\ldots,t_{i-1}\rangle$. We show by induction that $K\langle t_1,\ldots,t_i\rangle=K$ for each $i \leq n$. Fix $i\leq n$ and suppose that $K\langle t_1,\ldots,t_{i-1}\rangle = K$. If $t_i' \in K$, then since $K$ is Liouville closed, there is $f \in K$ with $f' = t_i'$. Then $t_i =f + c$ for some $c \in C_L = C$, so $t_i \in K$ as well. Likewise, if $t_i^\dagger \in K$, then there is $g \in K^\times$ with $g^\dagger = t_i^\dagger$, so $t_i =cg$ for some $c \in C_L^\times = C^\times$, again giving $t_i \in K$.
\end{proof}

\begin{definition}
A \textbf{$T$-Liouville closure} of $K$ is a $T$-Liouville extension of $K$ which is Liouville closed.
\end{definition}

\begin{corollary}
\label{cor:charofTLclosure1}
Let $L$ be a Liouville closed $H_T$-field extension of $K$. If $L$ is a $T$-Liouville closure of $K$, then $L$ has no proper Liouville closed $H_T$-subfields which contain $K$. If $L$ has no proper Liouville closed $H_T$-subfields which contain $K$ and $C_L = C$, then $L$ is a $T$-Liouville closure of $K$. 
\end{corollary}
\begin{proof}
Suppose that $L$ is a $T$-Liouville closure of $K$ and let $M$ be a Liouville closed $H_T$-subfield of $L$ containing $K$. Then $L$ is an $T$-Liouville extension of $M$ by Fact~\ref{fact:Lexts}, so $M= L$ by Lemma~\ref{lem:closedimpliesmaximal}. Now suppose that $C_L = C$ and that $L$ is not a $T$-Liouville closure of $K$. We will find a proper Liouville closed subfield of $L$ containing $K$. Let $M$ be a maximal $T$-Liouville extension of $K$ contained in $L$ (we know that $M$ exists by Fact~\ref{fact:Lexts} and Zorn's lemma). We claim that $M$ is Liouville closed. Let $s \in M$ and take $a \in L$ and $b \in L^\times$ with $a' =b^\dagger= s$. Since $C_{M \langle a \rangle},C_{M \langle b \rangle} \subseteq C_L = C$, we see that $M \langle a \rangle$ and $M\langle b\rangle$ are both $T$-Liouville extensions of $M$, and therefore equal to $M$ by maximality and Fact~\ref{fact:Lexts}. Thus, every element of $M$ has an integral in $M$ and an exponential integral in $M^\times$.
\end{proof}

\subsection{$T$-Liouville towers}\label{subsec:towers}
\begin{assumption}
For the remainder of this section, we assume that $T$ is power bounded.
\end{assumption}

\begin{definition}
A \textbf{$T$-Liouville tower on $K$} is a strictly increasing chain $(K_\mu)_{\mu \leq \nu}$ of $H_T$-fields such that:
\begin{enumerate}
\item $K_0 = K$;
\item if $\mu\leq \nu$ is an infinite limit ordinal, then $K_\mu = \bigcup_{\eta<\mu}K_\eta$;
\item if $\mu<\nu$, then $K_{\mu+1} = K_\mu\langle a_\mu \rangle$ with $a_\mu \not\in K_\mu$ and one of the following holds:
\begin{enumerate}
\item $a_\mu' = s_\mu \in K_\mu$ with $a_\mu \prec 1$ and $vs_\mu$ is a gap in $K_\mu$;
\item $a_\mu' = s_\mu \in K_\mu$ with $a_\mu \succ 1$ and $vs_\mu$ is a gap in $K_\mu$;
\item $a_\mu' = s_\mu \in K_\mu$ with $vs_\mu = \max\Psi_{K_\mu}$;
\item $a_\mu' = s_\mu \in K_\mu$ with $a_\mu \prec 1$, $vs_\mu \in (\Gamma^>_{K_\mu})'$, and $s_\mu \not\in \der \smallo_{K_\mu}$;
\item $a_\mu' = s_\mu \in K_\mu$ with $v(s_\mu-\der K_\mu)\subseteq (\Gamma^<_{K_\mu})'$;
\item $a_\mu^\dagger = s_\mu \in K_\mu$ with $a_\mu \sim 1$, $v s_\mu\in (\Gamma^>_{K_\mu})'$, and $s_\mu \neq y^\dagger$ for all $y \in K_\mu^\times$;
\item $a_\mu^\dagger = s_\mu \in K_\mu$ with $a_\mu >0$ and $v\big(s_\mu-(K_\mu^\times)^\dagger\big)\subseteq \Psi_{K_\mu}^\downarrow$.
\end{enumerate}
\end{enumerate}
The $H_T$-field $K_\nu$ is called the \textbf{top of the tower $(K_\mu)_{\mu \leq \nu}$}.
\end{definition}

\noindent
Let $(K_\mu)_{\mu \leq \nu}$ be a $T$-Liouville tower on $K$. Note that (a), (b), (c), (f), and (g) correspond to Lemmas~\ref{lem:gapgoesup},~\ref{lem:gapgoesdown},~\ref{lem:maxgoesdown},~\ref{lem:smallexpint}, and~\ref{lem:bigexpint} and that (d) and (e) correspond to Corollaries~\ref{cor:smallint} and~\ref{cor:bigint}, respectively. In each of these extensions, we have $\res K_{\mu+1} = \res K_\mu$, so $C_{K_{\mu+1}} = C_{K_\mu}$ by Corollary~\ref{cor:HTres}. Thus, $K_{\mu+1}$ is a $T$-Liouville extension of $K_\mu$ for each $\mu\leq \nu$. Using also Fact~\ref{fact:Lexts}, we see that each $K_\mu$ is a $T$-Liouville extension of $K$. Since each $T$-Liouville extension of $K$ has the same cardinality as $K$, maximal $T$-Liouville towers on $K$ exist by Zorn's lemma.

\begin{lemma}
\label{lem:toptclosure}
Let $L$ be the top of a maximal $T$-Liouville tower on $K$. Then $L$ is Liouville closed and, therefore, $L$ is a $T$-Liouville closure of $K$.
\end{lemma}
\begin{proof}
By (a) and (b), $L$ does not have a gap, and by (c), $L$ is not grounded, so $L$ has asymptotic integration by Fact~\ref{fact:trich}. Let $s \in L$. We will show that $s$ has an integral in $L$ and an exponential integral in $L^\times$. We have $v(s- \der L) \not\subseteq(\Gamma^<_L)'$ by (e), so there is $y \in L$ with $v(s-y')> (\Gamma^<_L)'$. Since $L$ has asymptotic integration, we have $v(s-y') \in (\Gamma^>_L)'$ so by (d), there is $f \in \smallo_L$ with $f' = s-y'$. Then $s = (f+y)'$. Likewise, by (g) there is $b \in L^\times$ with $v(s- b^\dagger) > \Psi_L^\downarrow$. Asymptotic integration gives $v(s- b^\dagger) \in (\Gamma^>_L)'$, and we may take $g \in L^\times$ with $g \sim 1$ and $g^\dagger = s-b^\dagger$ by (f). Then $s= (bg)^\dagger$.
\end{proof}

\noindent
Lemma~\ref{lem:toptclosure} gives the existence of $T$-Liouville closures under our standing assumption that $T$ is power bounded. The rest of the section is focused on uniqueness.

\begin{lemma}
\label{lem:relativemaxismax}
Let $L$ be a Liouville closed $H_T$-field extension of $K$ and let $(K_\mu)_{\mu\leq \nu}$ be a $T$-Liouville tower on $K$. Suppose that $(K_\mu)_{\mu\leq \nu}$ is a tower in $L$, that is, each $K_\mu$ is an $H_T$-subfield of $L$. Suppose also that $(K_\mu)_{\mu\leq \nu}$ cannot be extended to a $T$-Liouville tower $(K_\mu)_{\mu\leq \nu+1}$ in $L$. Then $(K_\mu)_{\mu\leq \nu}$ is a maximal $T$-Liouville tower on $K$. 
\end{lemma}
\begin{proof}
By Fact~\ref{fact:Lexts} and Lemma~\ref{lem:closedimpliesmaximal}, it suffices to show that $K_\nu$ is Liouville closed. If $vs$ is a gap in $K_\nu$ for some $s \in K_\nu$, then $L$ contains an element $a$ with $a' = s$. By subtracting a constant from $a$, we may assume that $a \not\asymp1$. By Lemma~\ref{lem:gapgoesup} (if $a\prec 1$) or Lemma~\ref{lem:gapgoesdown} (if $a\succ 1$), we see that $K_\nu\langle a \rangle \subseteq L$ is a $T$-Liouville extension of $K_\nu$, contradicting the maximality of $(K_\mu)_{\mu\leq \nu}$ in $L$. Thus, $K_\nu$ has no gap and likewise, Lemma~\ref{lem:maxgoesdown} shows that $K_\nu$ is ungrounded, so $K_\nu$ has asymptotic integration by Fact~\ref{fact:trich}. 

Fix $s \in K_\nu$. If $v(s-\der K_\nu) \subseteq (\Gamma_{K_\nu}^<)'$, then $K\langle f\rangle$ is a $T$-Liouville extension of $K_\nu$ contained in $L$ for any $f \in L$ with $f' = s$ by Corollary~\ref{cor:bigint}, contradicting the maximality of $(K_\mu)_{\mu\leq \nu}$ in $L$. Therefore, we may take $y \in K_\nu$ with $v(s-y')>(\Gamma_{K_\nu}^<)'$. As $K_\nu$ has asymptotic integration, we have $v(s-y')\in(\Gamma_{K_\nu}^>)'$. If $s-y' \not\in \der\smallo_{K_\nu}$, then $K\langle g\rangle$ is a $T$-Liouville extension of $K_\nu$ contained in $L$ for any $g \in \smallo_L$ with $g' = s-y'$ by Corollary~\ref{cor:smallint}, again contradicting the maximality of $(K_\mu)_{\mu\leq \nu}$ in $L$. Thus, $s-y' \in \der\smallo_{K_\nu}$, so $s\in \der K_\nu$. A similar argument, using Lemmas~\ref{lem:smallexpint} and~\ref{lem:bigexpint} shows that $s$ has an exponential integral in $K_\nu^\times$.
\end{proof}

\noindent
Lemmas~\ref{lem:toptclosure} and~\ref{lem:relativemaxismax} can be used to remove the assumption ``$C_L = C$'' from Corollary~\ref{cor:charofTLclosure1} under our current assumption of power boundedness.

\begin{corollary}
\label{cor:charofTLclosure}
Let $L$ be a Liouville closed $H_T$-field extension of $K$. Then $L$ is a $T$-Liouville closure of $K$ if and only if $L$ has no proper Liouville closed $H_T$-subfields which contain $K$.
\end{corollary}
\begin{proof}
One implication holds by Corollary~\ref{cor:charofTLclosure1}. For the other, suppose that $L$ is not a $T$-Liouville closure of $K$ and let $(K_\mu)_{\mu\leq \nu}$ be a maximal $T$-Liouville tower on $K$ in $L$. Then $K_\nu$ is a $T$-Liouville closure of $K$ by Lemmas~\ref{lem:toptclosure} and~\ref{lem:relativemaxismax}. In particular, $K_\nu$ is a proper Liouville closed $H_T$-subfield of $L$ containing $K$. 
\end{proof}

\subsection{$\upl$-freeness and the uniqueness of $T$-Liouville closures}
Whether $K$ has a unique $T$-Liouville closure up to $\LdO(K)$-isomorphism is closely tied to the existence of gaps, which is in turn related to $\upl$-freeness.

\begin{lemma}
\label{lem:nogap}
Let $(K_\mu)_{\mu\leq \nu}$ be a $T$-Liouville tower on $K$ and suppose $K_\mu$ does not have a gap for all $\mu< \nu$. Then $K_\nu$ embeds over $K$ into any Liouville closed $H_T$-field extension of $K$.
\end{lemma}
\begin{proof}
Let $M$ be a Liouville closed $H_T$-field extension of $K$. We will construct an increasing chain of $\LdO(K)$-embeddings $(\imath_\mu\colon K_\mu\to M)_{\mu\leq\nu}$. Let $\imath_0\colon K_0\to M$ be the identity on $K$, and take increasing unions at limits. For successors, fix $\mu<\nu$ and let $\imath_\mu\colon K_\mu\to M$ be an $\LdO(K)$-embedding. Since $K_\mu$ has no gap, $K_{\mu+1}$ is an extension of type (c), (d), (e), (f), or (g). The embedding properties in Lemmas~\ref{lem:maxgoesdown},~\ref{lem:smallexpint} and~\ref{lem:bigexpint} and Corollaries~\ref{cor:smallint} and~\ref{cor:bigint} give an $\LdO(K)$-embedding $\imath_{\mu+1}\colon K_{\mu+1}\to M$ extending $\imath_\mu$.
\end{proof}

\begin{proposition}\label{prop:lambdafreeclosure}
Suppose $K$ is ungrounded and $\upl$-free. Then $K$ has a $T$-Liouville closure $L$ which embeds over $K$ into any Liouville closed $H_T$-field extension of $K$. Any $T$-Liouville closure of $K$ is $\LdO(K)$-isomorphic to $L$.
\end{proposition}
\begin{proof}
Let $(K_{\mu})_{\mu\leq \nu}$ be a maximal $T$-Liouville tower on $K$. We will prove by induction on $\mu\leq \nu$ that each $K_\mu$ is ungrounded and $\upl$-free. This holds when $\mu = 0$ by assumption and if $\mu \leq \nu$ is an infinite limit ordinal, then this follows from Lemma~\ref{lem:unionupl}. Let $\mu < \nu$ and suppose that $K_\mu$ is ungrounded and $\upl$-free. Then $K_\mu$ has no gap, so $K_{\mu+1}$ must be an extension of type (d), (e), (f), or (g). Then $K_{\mu+1}$ is ungrounded and $\upl$-free by Lemmas~\ref{lem:smallexpint} and~\ref{lem:bigexpint} and Corollaries~\ref{cor:smallint} and~\ref{cor:bigint}. Set $L\coloneqq K_\nu$, so $L$ is a $T$-Liouville closure of $K$, and let $M$ be a Liouville closed $H_T$-field extension of $K$. By Lemma~\ref{lem:nogap}, there is an $\LdO(K)$-embedding $\imath\colon L \to M$. Moreover, if $M$ is a $T$-Liouville closure of $K$, then $\imath(L)$ is a Liouville closed $H_T$-subfield of $M$ containing $K$, so $\imath(L)=M$ by Corollary~\ref{cor:charofTLclosure1}. Thus, $L$ is unique up to $\LdO(K)$-isomorphism.
\end{proof}

\begin{proposition}\label{prop:groundedclosure}
Suppose $K$ is grounded. Then $K$ has a $T$-Liouville closure $L$ which embeds over $K$ into any Liouville closed $H_T$-field extension of $K$. Any $T$-Liouville closure of $K$ is $\LdO(K)$-isomorphic to $L$.
\end{proposition}
\begin{proof}
Let $K_{\upo}$ be as in Corollary~\ref{cor:omegaconstruction}. Then $K_{\upo}$ is the union of an increasing chain of $T$-Liouville extensions of $K$, so $K_{\upo}$ is an $T$-Liouville extension of $K$ by Fact~\ref{fact:Lexts}. Moreover, $K_{\upo}$ is ungrounded and $\upl$-free, so $K_{\upo}$ has a $T$-Liouville closure $L$ which embeds over $K_\upo$ into any Liouville closed $H_T$-field extension of $K_\upo$ by Proposition~\ref{prop:lambdafreeclosure}. Then $L$ is a $T$-Liouville closure of $K$ as well, since $K_{\upo}$ is a $T$-Liouville extension of $K$. Let $M$ be a Liouville closed $H_T$-field extension of $K$. Then $M$ is closed under taking $\d$-logarithms, so Corollary~\ref{cor:omegaconstruction} gives an $\LdO(K)$-embedding $K_\upo\to M$ which further extends to an $\LdO(K)$-embedding $L\to M$. As in Proposition~\ref{prop:lambdafreeclosure}, uniqueness follows from this embedding property and Corollary~\ref{cor:charofTLclosure1}.
\end{proof}

\subsection{Gaps and the nonuniqueness of $T$-Liouville closures}
If $K$ has a gap $vs$, then we have a choice to make. Either we can adjoin an integral $a$ of $s$ with $a\prec 1$, as is done in Lemma~\ref{lem:gapgoesup}, or we can adjoin an integral $b$ of $s$ with $b \succ 1$, as in Lemma~\ref{lem:gapgoesdown}. This ``fork in the road'' prevents $K$ from having a unique $T$-Liouville closure, but as we will see below, this is really the only obstruction to uniqueness.

\begin{proposition}
\label{prop:gapclosures}
Let $\beta \in \Gamma$ be a gap in $K$. Then $K$ has $T$-Liouville closures $L_1$ and $L_2$ with $\beta \in (\Gamma_{L_1}^>)'$ and $\beta \in (\Gamma_{L_2}^<)'$. Let $M$ be a Liouville closed $H_T$-field extension of $K$. If $\beta \in (\Gamma_M^>)'$, then there is an $\LdO(K)$-embedding $L_1\to M$. Likewise, if $\beta \in (\Gamma_M^<)'$, then there is an $\LdO(K)$-embedding $L_2\to M$. Any $T$-Liouville closure of $K$ is $\LdO(K)$-isomorphic to either $L_1$ or $L_2$. 
\end{proposition}
\begin{proof}
Let $s \in K$ with $vs = \beta$. Let $K_1\coloneqq K\langle a \rangle$ be the $H_T$-field extension of $K$ given in Lemma~\ref{lem:gapgoesup}, so $a \prec 1$ and $a' = s$, and let $K_2\coloneqq K\langle b \rangle$ be the $H_T$-field extension of $K$ given in Lemma~\ref{lem:gapgoesdown}, so $b \succ 1$ and $b' = s$. Then $K_1$ is grounded, so it has a $T$-Liouville closure $L_1$ which embeds over $K_1$ into any Liouville closed $H_T$-field extension of $K_1$ by Proposition~\ref{prop:groundedclosure}. Likewise, $K_2$ has a $T$-Liouville closure $L_2$ which embeds over $K_2$ into any Liouville closed $H_T$-field extension of $K_2$. Now let $M$ be a Liouville closed $H_T$-field extension of $K$. If $\beta \in (\Gamma_M^>)'$, then the embedding property in Lemma~\ref{lem:gapgoesup} gives an $\LdO(K)$-embedding $K_1\to M$, which in turn extends to an $\LdO(K)$-embedding $L_1\to M$. If $\beta \in (\Gamma_M^<)'$, then using the embedding property in Lemma~\ref{lem:gapgoesdown} instead, we get an $\LdO(K)$-embedding $L_2\to M$. If $M$ is a $T$-Liouville closure of $K$, then $M$ is $\LdO(K)$-isomorphic to either $L_1$ or $L_2$ by Corollary~\ref{cor:charofTLclosure1} since $M$ contains the $\LdO(K)$-isomorphic image of either $L_1$ or $L_2$ as a Liouville closed $H_T$-subfield.
\end{proof}

\noindent
We can use Lemma~\ref{lem:uplgapcreator} to show that $H_T$-fields with asymptotic integration which are not $\upl$-free also have two distinct $T$-Liouville closures.

\begin{proposition}
\label{prop:uplclosures}
Suppose that $K$ has asymptotic integration and is not $\upl$-free. Then $K$ has $T$-Liouville closures $L_1$ and $L_2$ which are not $\LdO(K)$-isomorphic. If $M$ is a Liouville closed $H_T$-field extension of $K$, then there is an $\LdO(K)$-embedding of either $L_1$ or $L_2$ into $M$. Any $T$-Liouville closure of $K$ is $\LdO(K)$-isomorphic to either $L_1$ or $L_2$. 
\end{proposition}
\begin{proof}
Let $(\upl_\rho)$ be a $\upl$-sequence in $K$ with pseudolimit $\upl \in K$, so $v\big(\upl+(K^\times)^\dagger\big)$ is a cofinal subset of $\Psi^\downarrow$ by Fact~\ref{fact:uplminusdaggers}. Lemma~\ref{lem:bigexpint} gives an $H_T$-field extension $K\langle a \rangle$ of $K$ with $a>0$ and $a^\dagger = -\upl$. By Lemma~\ref{lem:uplgapcreator}, $va$ is a gap in $K\langle a \rangle$. By Proposition~\ref{prop:gapclosures}, $K\langle a \rangle$ has $T$-Liouville closures $L_1$ and $L_2$ with $va \in (\Gamma_{L_1}^>)'$ and $va \in (\Gamma_{L_2}^<)'$, one of which embeds over $K\langle a \rangle$ into any Liouville closed $H_T$-field extension of $K\langle a \rangle$. We claim that there is no $\LdO(K)$-embedding $L_1\to L_2$; in particular, $L_1$ and $L_2$ are nonisomorphic over $K$. To see this, take $b_1 \in L_1$ and $b_2\in L_2$ with $b_1 \prec 1$, $b_2 \succ 1$, and $b_1'=b_2' = a$. Then $(b_1')^\dagger = (b_2')^\dagger = a^\dagger = -\upl$. Suppose toward contradiction that $\imath\colon L_1\to L_2$ is an $\LdO(K)$-embedding. Then $\imath(b_1')^\dagger = (b_2')^\dagger$ so $\imath(b_1) = c_1b_2+c_2$ for some $c_1,c_2 \in C_{L_2}$ with $c_1 \neq 0$. Since $b_2 \succ 1$, this gives $\imath(b_1)\succ 1$, contradicting that $b_1\prec 1$.

Let $M$ be a Liouville closed $H_T$-field extension of $K$. Lemma~\ref{lem:bigexpint} gives an $\LdO(K)$-embedding $K\langle a \rangle\to M$ and Proposition~\ref{prop:gapclosures} allows us to extend this embedding to an $\LdO(K)$-embedding of either $L_1$ or $L_2$ into $M$. As in Proposition~\ref{prop:gapclosures}, we may use Corollary~\ref{cor:charofTLclosure1} to see that any $T$-Liouville closure of $K$ is $\LdO(K)$-isomorphic to either $L_1$ or $L_2$.
\end{proof}

\noindent
Putting together the above propositions, we can now precisely state our main theorem on the existence and uniqueness of $T$-Liouville closures.

\begin{theorem}
\label{thm:oneortwotlclosures}
If $K$ is grounded or if $K$ is ungrounded and $\upl$-free, then $K$ has exactly one $T$-Liouville closure up to $\LdO(K)$-isomorphism. If $K$ is ungrounded and not $\upl$-free, then $K$ has exactly two $T$-Liouville closures up to $\LdO(K)$-isomorphism. For any Liouville closed $H_T$-field extension $M$ of $K$, there is an $\LdO(K)$-embedding of some $T$-Liouville closure of $K$ into $M$.
\end{theorem}

\section{Logarithmic $H_T$-fields and logarithmic pre-$H_T$-fields}\label{sec:logHT}
\noindent
In this section, we assume that $T$ is power bounded with field of exponents $\Lambda$. We further assume that $T$ defines a restricted exponential function $\e$ and that $\Lambda$ is cofinal in the prime model of $T$. We let $\ln$ denote the compositional inverse of $\e$.

\medskip\noindent
Recall from Subsection~\ref{subsec:expsandlogs} the language $\cL_{\log} = \cL \cup \{\log\}$ and the $\cL_{\log}$-theory $T^{\e}$, which extends $T$ by axioms stating that $\log$ is a surjective logarithm. In this section and the next, we study pre-$H_{T^{\e}}$-fields, $H_{T^{\e}}$-fields, and their extensions. From a valuation-theoretic perspective, it is inconvenient to work with pre-$H_{T^{\e}}$-fields directly, so we instead work with a broader class of $\LdO_{\log}$-structures, called \emph{logarithmic pre-$H_T$-fields}, where the logarithm is not assumed to be surjective.

\begin{definition}
Let $K$ be a pre-$H_T$-field and let $\log$ be a logarithm on $K$. We say that $K$ is a \textbf{logarithmic pre-$H_T$-field} if
\begin{enumerate}
\item[(LH1)] $\log(a)' = a^\dagger$ for $a \in K^>$, and
\item[(LH2)] $\cO\subseteq \log(K^>)$. 
\end{enumerate}
We say that $K$ is a \textbf{logarithmic $H_T$-field} if $K$ is a logarithmic pre-$H_T$-field which is also an $H_T$-field. 
\end{definition}

\noindent
Clearly, every pre-$H_{T^{\e}}$-field is a logarithmic pre-$H_T$-field and, likewise, every $H_{T^{\e}}$-field is a logarithmic $H_T$-field. After proving some basic results, we will show in Proposition~\ref{prop:surjectiveisexp} below that a logarithmic (pre)-$H_T$-field is a (pre)-$H_{T^{\e}}$-field if and only if the logarithm is surjective. 

\begin{assumption}
For the remainder of this section, let $K = (K,\log,\cO,\der)$ be a logarithmic pre-$H_T$-field.
\end{assumption}

\noindent
For $a \in K^>$, axiom (LH1) tells us that $\log a$ is a $\d$-logarithm of $a$. Thus, we have the following consequence of Corollary~\ref{cor:expungrounded}:

\begin{corollary}\label{cor:ungrounded}
Any logarithmic pre-$H_T$-field is ungrounded.
\end{corollary}

\noindent
We now investigate the induced logarithms on the residue field and the constant field of $K$.

\begin{lemma}\label{lem:inducedlogres}
The logarithm on $K$ induces a well-defined logarithm on $\res K$ (also denoted by $\log$). With this induced logarithm, $\res K$ is a model of $T^{\e}$.
\end{lemma}
\begin{proof}
Let $a \in \cO$ with $\bar{a} >0$. Then $a \asymp 1$, so Lemma~\ref{lem:logvals} gives $\log a \in \cO$. Since $\cO$ is $T$-convex and $\ln$ is $\cL(\emptyset)$-definable and continuous at $1$, we have $\ln(1+\smallo) \subseteq \smallo$ by~\cite[Lemma 1.13]{DL95}. Since $a+\smallo = a(1+\smallo)$, axioms (L1) and (L2) give
\[
\log(a+\smallo)\ =\ \log a + \log(1+\smallo)\ =\ \log a+\ln(1+\smallo)\ \subseteq \ \log a + \smallo.
\]
Thus, $\log$ induces a well-defined map on $\res K$. It is routine to verify that this induced map is a logarithm on $\res K$, and it remains to show that this induced logarithm is surjective. To see this, let $a \in \cO$ and, using (LH2), take $b \in K^>$ with $\log b = a$. Since $\log b \in \cO$ is a $\d$-logarithm of $b$, Lemma~\ref{lem:logvals} tells us that $b \asymp 1$. Then $\bar{b}>0$ and $\log \bar{b} = \bar{a}$. 
\end{proof}

\begin{lemma}\label{lem:inducedlogC}
The restriction of the logarithm on $K$ to $C$ is a logarithm on $C$. With this restricted logarithm, $C$ is a model of $T^{\e}$. 
\end{lemma}
\begin{proof}
For $c \in C^>$, we have $\log(c)' = c^\dagger = 0$, so $\log c \in C$ as well. Thus, $\log(C^>) \subseteq C$, and it follows immediately that $\log|_C$ is a logarithm on $C$. To see that $\log(C^>) = C$, let $a \in C$ and, using that $C \subseteq \cO \subseteq \log(K^>)$, take $f \in K^>$ with $\log f =a$. Then $f^\dagger = a' = 0$, so $f \in C^>$. 
\end{proof}

\noindent
The natural $\cL$-embedding $C \to \res K$ is even an $\cL_{\log}$-embedding, where $\res K$ and $C$ are equipped with the logarithms from Lemmas~\ref{lem:inducedlogres} and~\ref{lem:inducedlogC}, respectively. Now, we can say more about the relationship between logarithmic pre-$H_T$-fields and pre-$H_{T^{\e}}$-fields.

\begin{proposition}\label{prop:surjectiveisexp}
The logarithmic pre-$H_T$-field $K$ is a pre-$H_{T^{\e}}$-field if and only if $\log(K^>) = K$. Moreover, if $K$ is a logarithmic $H_T$-field, then $K$ is an $H_{T^{\e}}$-field if and only if $\log(K^>) = K$.
 \end{proposition}
\begin{proof}
One direction is immediate. For the other, suppose that $\log(K^>) = K$. Then the $\cL_{\log}$-reduct of $K$ is a model of $T^{\e}$. By Lemma~\ref{lem:inducedlogres}, the residue field $\res K$ with the induced logarithm models $T^{\e}$ as well. It follows that $\cO$ is $T^{\e}$-convex, and it remains to show that $\der$ is a $T^{\e}$-derivation. By Corollary~\ref{cor:expterms}, any $\cL_{\log}(\emptyset)$-definable function on $K$ is given piecewise by a composition of $\cL(\emptyset)$-definable functions, $\log$, and $\exp$. By~\cite[Lemma 2.6]{FK19}, $\der$ is compatible with a composition of functions so long as it is compatible with each constituent function. Since $\der$ is compatible with $\log$ and with all $\cL(\emptyset)$-definable $\cC^1$-functions by assumption, it remains to show that $\der$ is compatible with $\exp$. For $u \in K$, we have $u = \log (\exp u)$, so taking derivatives gives
\[
u'\ =\ \log (\exp u)'\ =\ (\exp u)\inv \exp(u)'.
\]
Thus, $ \exp (u)' = \exp(u) u'$, as desired. 
\end{proof}

\noindent
Next, we provide a short test for checking whether an $\LdO$-embedding of logarithmic pre-$H_T$-fields is also an $\LdO_{\log}$-embedding, along with a longer proposition on extending the logarithm on $K$ to certain pre-$H_T$-field extensions of $K$.

\begin{lemma}\label{lem:logclosedtest}
Let $M$ be a logarithmic pre-$H_T$-field and let $\imath\colon K \to M$ be an $\LdO$-embedding. Let $f,g\in K^>$ with $f \sim g$. If $\imath(\log g) = \log_M\imath (g)$, then $\imath (\log f)= \log_M\imath(f)$. 
\end{lemma}
\begin{proof}
Take $\epsilon \in \smallo$ such that $f = g(1+\epsilon)$. Then 
\[
\log f \ =\ \log g + \log(1+\epsilon) \ =\ \log g + \ln(1+\epsilon).
\]
As $\imath$ is an $\cL$-embedding, we have $\imath \big(\ln(1+\epsilon)\big)=\ln\imath(1+\epsilon)$. We have $\imath(\log g) = \log_M\imath(g)$ by assumption, so $\imath(\log f) = \log_M\imath(f)$.
\end{proof}

\begin{proposition}\label{prop:logclosedtest2}
Let $L$ be a pre-$H_T$-field extension of $K$ with $\res L = \res K$. Let $(a_i)_{i \in I}$ be a family of elements in $L^>$ with $a_i \not\asymp 1$ for each $i$ such that 
\[
\Gamma_L\ =\ \Gamma \oplus \bigoplus_{i \in I}\Lambda v a_i,
\]
and let $(b_i)_{i \in I}$ be a family of elements in $L$ such that $b_i' = a_i^\dagger$ for each $i \in I$. Then there is a unique logarithm on $L$ extending the logarithm on $K$ such that $\log_L a_i = b_i$ for each $i \in I$. With this logarithm, $L$ is a logarithmic pre-$H_T$-field extension of $K$. If $M$ is also a logarithmic pre-$H_T$-field extension of $K$ and if $\imath\colon L \to M$ is an $\LdO(K)$-embedding, then $\imath$ is an $\LdO_{\log}(K)$-embedding if and only if $\imath(b_i) = \log_M \imath(a_i)$ for each $i \in I$.
\end{proposition}
\begin{proof}
Let $f \in L^>$. Our assumption on $\Gamma_L$ and $\res L$ gives
\[
f\ =\ g(1+\epsilon)\prod_{i \in I}a_i^{\lambda_i}
\]
for some $g \in K^>$, some $\epsilon \in \smallo_L$, and some family $(\lambda_i)_{i \in I}$ of exponents in $\Lambda$ where only finitely many $\lambda_i$ are nonzero. Set
\[
\log_L f\ \coloneqq \ \log g + \ln(1+\epsilon) + \sum_{i \in I}\lambda_ib_i.
\]
It is routine to show that this does not depend on the choice of $g$. Before we show that $\log_L$ is a logarithm on $L$, we first note that
\[
(\log_L f)'\ =\ (\log g)' + \ln(1+\epsilon)' + \sum_{i \in I}\lambda_ib_i'\ =\ g^\dagger + (1+\epsilon)^\dagger + \sum_{i \in I}\lambda_ia_i^\dagger \ =\ f^\dagger,
\]
so (LH1) holds. We now turn to verifying (L1)--(L4). A straightforward computation gives $\log_L(f_1f_2) = \log_L f_1 +\log_L f_2$ for all $f_1,f_2 \in L^>$, so $\log_L$ is a group homomorphism. To see that $\log_L$ is even an ordered group embedding, we assume that $f> 1$, and we need to show that $\log_L f>0$. If $f\succ 1$, then $(\log_Lf)' = f^\dagger >0$ by (PH1). Since $\log_L f$ is a $\d$-logarithm of $f$, Lemma~\ref{lem:logvals} gives $v(\log_L f) = \chi(vf) < 0$. Thus, $(\log_L f)^\dagger > 0$ as well, so $\log_L f = (\log_L f)' /(\log_L f)^\dagger >0$. Now, assume that $f \asymp 1$, so each $\lambda_i = 0$ and $f = g(1+\epsilon)$. If $f \not\sim 1$, then $\bar{g}=\bar{f} >1$, so $\log \bar{g} > 0$ by Lemma~\ref{lem:inducedlogres}. Thus, $\log g> \smallo_L$, and since $\ln(1+\epsilon) \in \smallo_L$, we have
\[
\log_L f \ =\ \log g+ \ln(1+\epsilon)\ \sim\ \log g\ >\ 0.
\]
If $f \sim 1$, then we may assume that $g = 1$, so $\epsilon > 0$ and $\log_L f = \ln (1+\epsilon) > 0$. 

For (L2), we assume that $\e(-1)\leq f \leq \e(1)$, and we need to show that $\log_L f = \ln f$. Our assumption on $f$ gives that each $\lambda_i = 0$ and $\e(-1)\leq g \leq \e(1)$. Then
\[
\log_L f \ =\ \log g + \ln(1+ \epsilon) \ =\ \ln g + \ln(1+\epsilon) \ =\ \ln f.
\]

For (L3), let $\lambda \in \Lambda$ with $\lambda > 1$, and assume that $f \geq \lambda^2$. We need to show that $f\geq \lambda \log_L f$. If $f\succ 1$, then $v(\log_L f) = \chi(vf) > vf$ by~\cite[Lemma 9.2.18]{ADH17}, so $f\succ \log_L f$. In particular, $f\geq \lambda\log_L f$. Thus, we may assume that $f \asymp 1$, so each $\lambda_i = 0$ and $f = g(1+\epsilon)$, where $g \asymp 1$. Lemma~\ref{lem:leqyminus1} gives $\epsilon = \e(\ln(1+\epsilon)) - 1\geq \ln(1+\epsilon)$, so 
\begin{equation}\label{eq:logepsilon}
\lambda \log_L f\ =\ \lambda \log g+ \lambda\ln(1+\epsilon) \ \leq \ \lambda \log g + \lambda\epsilon.
\end{equation}
If $f \sim \lambda^2$, then we may arrange that $g = \lambda^2$ and $\epsilon \geq 0$. Thus, $\lambda \log g = \lambda \log \lambda^2 \leq \lambda^2$ and $\lambda\epsilon \leq \lambda^2\epsilon$. Combined with (\ref{eq:logepsilon}), this gives $\lambda \log_L f \leq \lambda^2+\lambda^2\epsilon = \lambda^2(1+\epsilon) = f$, as desired. If $f \not\sim \lambda^2$, then $\bar{g}=\bar{f}> \lambda^2$, so Lemmas~\ref{lem:powerstrictlog} and~\ref{lem:inducedlogres} give $\bar{g}> \lambda\log \bar{g}$. Thus, $g+ (g-\lambda)\epsilon > \lambda\log g$ since $(g-\lambda)\epsilon \prec 1$. Again, (\ref{eq:logepsilon}) gives
\[
\lambda \log_L f\ <\ g+ (g-\lambda)\epsilon+ \lambda\epsilon\ =\ g(1 +\epsilon) \ =\ f.
\]

For (L4) let $\rho \in \Lambda$. Then $f^\rho = g^\rho(1+\epsilon)^\rho\prod_{i \in I}a_i^{\rho\lambda_i}$, so
\[
\log_L f^\rho\ =\ \log g^\rho + \ln\big((1+\epsilon)^\rho\big) + \sum_{i \in I}\rho\lambda_ib_i \ =\ \rho\log g + \rho\ln(1+\epsilon) + \rho\sum_{i \in I}\lambda_ib_i\ =\ \rho \log_L f,
\]
where the equality $\ln\big((1+\epsilon)^\rho\big) = \rho\ln(1+\epsilon)$ holds by~\cite[Lemma 6.4.1]{Fo10}.

Finally, for (LH2), let $a \in \cO_L$ and, using that $\res L = \res K$, take $b \in \cO$ with $a- b\prec 1$. Then $b \in \log(K^>) \subseteq \log(L^>)$, so $a \in \log(L^>)$ by Lemma~\ref{lem:closeislog}.

Now let $M$ be a logarithmic pre-$H_T$-field extension of $K$ and let $\imath\colon L \to M$ be an $\LdO(K)$-embedding. Clearly, if $\imath$ is an $\LdO_{\log}(K)$-embedding, then $\imath(b_i) = \log_M \imath(a_i)$ for each $i\in I$. For the other implication, we assume that $\imath(b_i) = \log_M \imath(a_i)$ for each $i \in I$, and we need to show that $\imath(\log_L f) = \log_M \imath(f)$, where $f$ is as above. Using Lemma~\ref{lem:logclosedtest} and the fact that $f\sim g\prod_{i \in I}a_i^{\lambda_i}$, we may assume that $f = g\prod_{i \in I}a_i^{\lambda_i}$. Since 
\[
\log_L \Big(g\prod_{i \in I}a_i^{\lambda_i}\Big) \ =\ \log_L g + \sum_{i \in I}\lambda_i\log_L (a_i)
\]
and since $g \in K$, this further reduces to showing that $\imath(\log_L a_i) = \log_M \imath(a_i)$ for each $i\in I$. This holds by our assumption, since $\log_L a_i = b_i$ for each $i$. Uniqueness of $\log_L$ follows from this embedding property by taking $\imath\colon L \to L$ to be the identity map.
\end{proof}

\noindent
For the remainder of this article, we will just write $\log$ instead of $\log_L$ when the logarithmic pre-$H_T$-field $L$ is clear from context. The conditions on $\res L$ and $\Gamma_L$ in the above lemma are always satisfied when $L$ is an immediate pre-$H_T$-field extension of $K$.

\begin{corollary}\label{cor:immlogH}
Let $L$ be an immediate pre-$H_T$-field extension of $K$. Then there is a unique logarithm on $L$ extending the logarithm on $K$, and with this logarithm, $L$ is a logarithmic pre-$H_T$-field extension of $K$. If $M$ is also a logarithmic pre-$H_T$-field extension of $K$, then any $\LdO(K)$-embedding $L\to M$ is necessarily an $\LdO_{\log}(K)$-embedding.
\end{corollary}

\subsection{The exponential closure}
In this subsection, we prove that every logarithmic pre-$H_T$-field has a minimal logarithmic pre-$H_T$-field extension with a surjective logarithm. Below is the key Lemma.

\begin{lemma}\label{lem:addexp}
Let $f \in K\setminus \log(K^>)$. Then $K$ has a logarithmic pre-$H_T$-field extension $K\langle a \rangle$ with $a>0$ and $\log a = f$ such that for any logarithmic pre-$H_T$-field extension $M$ of $K$ with $f \in \log(M^>)$, there is a unique $\LdO_{\log}(K)$-embedding $K\langle a \rangle\to M$. Moreover, the extension $K\langle a \rangle$ has the following properties:
\begin{enumerate}
\item $va \not\in \Gamma$ and $\Gamma_{K\langle a \rangle}=\Gamma \oplus \Lambda va$;
\item $\res K\langle a \rangle = \res K$;
\item $\Psi$ is cofinal in $\Psi_{K\langle a \rangle}$;
\item a gap in $K$ remains a gap in $K\langle a \rangle$;
\item if $K$ is $\upl$-free, then so is $K\langle a \rangle$.
\end{enumerate}
\end{lemma}
\begin{proof}
We claim that $v\big(f'- (K^\times)^\dagger\big) \subseteq\Psi^{\downarrow}$. Suppose not and take $g \in K^\times$ with $v(f'- g^\dagger)>\Psi$. By replacing $g$ with $-g$ if necessary, we may assume that $g>0$, so $v(f'- g^\dagger) = v(f-\log g)' > \Psi$. Since $\Psi$ is cofinal in $(\Gamma^<)'$, we have $f- \log g \in \cO\subseteq \log(K^>)$. Take $h \in K^>$ with $f - \log g = \log h$. Then $f = \log(gh)$, a contradiction.

With this claim out of the way, we apply Lemma~\ref{lem:bigexpint} with $f'$ in place of $s$ to get a pre-$H_T$-field $K\langle a \rangle$ extending $K$ with $a^\dagger = f'$ which has properties (1)--(5) above. By property (1) and Proposition~\ref{prop:logclosedtest2}, there is a unique logarithm on $K\langle a \rangle$ with $\log a = f$ making $K\langle a \rangle$ a logarithmic pre-$H_T$-field extension of $K$. Now let $M$ be a logarithmic pre-$H_T$-field extension of $K$ with $f \in \log(M^>)$ and set $b\coloneqq \exp f \in M$. Then $b^\dagger = f'$, so Lemma~\ref{lem:bigexpint} gives a unique $\LdO(K)$-embedding $K\langle a \rangle \to M$ that sends $a$ to $b$. By the uniqueness part of Proposition~\ref{prop:logclosedtest2}, this is even an $\LdO_{\log}(K)$-embedding. Since any $\LdO_{\log}(K)$-embedding $K\langle a \rangle \to M$ must send $a$ to $b = \exp f$, this embedding is unique, even without the requirement that $a$ be sent to $b$.
\end{proof}

\noindent
Theorem~\ref{thm:expclosure} below follows by iterating Lemma~\ref{lem:addexp} (we also use that an increasing union of $\upl$-free logarithmic $H_T$-fields is $\upl$-free; see Lemma~\ref{lem:unionupl}).

\begin{theorem}\label{thm:expclosure}
$K$ has a logarithmic pre-$H_T$-field extension $K^{\e}$ with a surjective logarithm such that for any logarithmic pre-$H_T$-field extension $M$ of $K$ with a surjective logarithm, there is a unique $\LdO_{\log}(K)$-embedding $K^{\e} \to M$. The extension $K^{\e}$ has the following properties:
\begin{enumerate}
\item $\res K^{\e} = \res K$;
\item $\Psi$ is cofinal in $\Psi_{K^{\e}}$;
\item a gap in $K$ remains a gap in $K^{\e}$;
\item if $K$ is $\upl$-free, then so is $K^{\e}$.
\end{enumerate}
\end{theorem}

\noindent
We refer to $K^{\e}$ as the \textbf{exponential closure of $K$}. By Proposition~\ref{prop:surjectiveisexp}, the extension $K^{\e}$ is a pre-$H_{T^{\e}}$-field. If $K$ is a logarithmic $H_T$-field, then $K^{\e}$ is an $H_{T^{\e}}$-field by (1) and Corollary~\ref{cor:HTres}. The universal property in Theorem~\ref{thm:expclosure} gives that $K^{\e}$ is unique up to unique $\LdO_{\log}(K)$-isomorphism. This also gives the aforementioned minimality: if $M$ is a logarithmic pre-$H_T$-subfield of $K^{\e}$ containing $K$ with $M = \log(M^>)$, then $M = K^{\e}$. 
\subsection{Adjoining integrals and the logarithmic $H_T$-field hull}
In this subsection, we prove variants of the results in Subsection~\ref{subsec:ints}. We begin with the following immediate consequences of Corollaries~\ref{cor:smallint},~\ref{cor:bigint}, and~\ref{cor:immlogH}:

\begin{corollary}\label{cor:logsmallint}
Let $s \in K$ with $vs \in (\Gamma^>)'$ and $s \not\in \der\smallo$. Then $K$ has an immediate logarithmic pre-$H_T$-field extension $K\langle a \rangle$ with $a \prec 1$ and $a' = s$ such that for any logarithmic pre-$H_T$-field extension $M$ of $K$ with $s \in \der\smallo_M$, there is a unique $\LdO_{\log}(K)$-embedding $K\langle a \rangle \to M$. If $K$ is $\upl$-free, then so is $K\langle a \rangle$.
\end{corollary}

\begin{corollary}\label{cor:logbigint}
Let $s \in K$ with $v(s-\der K)\subseteq (\Gamma^<)'$. Then $K$ has an immediate logarithmic pre-$H_T$-field extension $K\langle a \rangle$ with $a' = s$ such that for any logarithmic pre-$H_T$-field extension $M$ of $K$ and $b \in M$ with $b' = s$, there is a unique $\LdO_{\log}(K)$-embedding $K\langle a \rangle \to M$ sending $a$ to $b$. If $K$ is $\upl$-free, then so is $K\langle a \rangle$.
\end{corollary}

\noindent
Now we show how to find an integral for $s \in K$ when $vs$ is a gap. We can't use Lemma~\ref{lem:gapgoesup} directly, as the pre-$H_T$-field extension constructed in that lemma is grounded, so it does not admit a logarithm. We rectify this issue by invoking the extension in Corollary~\ref{cor:omegaconstruction}.

\begin{lemma}\label{lem:loggapgoesup}
Let $s \in K$ and suppose $vs$ is a gap in $K$. Then $K$ has a logarithmic pre-$H_T$-field extension $K\langle a \rangle_\upo$ with $a \prec 1$ and $a' = s$ such that for any logarithmic pre-$H_T$-field extension $M$ of $K$ with $s \in \der\smallo_M$, there is a unique $\LdO_{\log}(K)$-embedding $K\langle a \rangle_\upo\to M$. The logarithmic pre-$H_T$-field $K\langle a \rangle_\upo$ is $\upo$-free with $\res K\langle a \rangle_\upo = \res K$.
\end{lemma}
\begin{proof}
Lemma~\ref{lem:gapgoesup} provides a grounded pre-$H_T$-field extension $K\langle a \rangle$ of $K$ with $a \prec 1$, $a' = s$, and $\res K\langle a \rangle = \res K$. Applying Corollary~\ref{cor:omegaconstruction} with $|a|\inv$ in place of $s$, we further extend $K\langle a \rangle$ to an $\upo$-free pre-$H_T$-field $K\langle a \rangle_\upo$ with $\res K\langle a \rangle_\upo = \res K$. Lemma~\ref{lem:gapgoesup} and the proof of Corollary~\ref{cor:omegaconstruction} tell us that 
\[
\Gamma_{K_{\upo}}\ =\ \Gamma \oplus \bigoplus_n\Lambda va_n,\qquad \Psi_{K_{\upo}}\ =\ \Psi \cup \{va_0^\dagger,va_1^\dagger,\ldots\},\qquad \Psi\ <\ va_0^\dagger\ <\ va_1^\dagger \ <\ \cdots,
\]
where $a_0 = |a|\inv$ and $a_{n+1}'=a_n^\dagger$ for each $n$. Using the pre-$H_T$-field axioms, one can easily check that each $a_n$ is positive and infinite, so by Proposition~\ref{prop:logclosedtest2}, there is a unique logarithm on $K\langle a \rangle_\upo$ which extends the logarithm on $K$ such that $\log a_n = a_{n+1}$ for each $n$. Let $M$ be a logarithmic pre-$H_T$-field extension of $K$ with $s \in \der\smallo_M$. By Lemma~\ref{lem:gapgoesup}, there is a unique $\LdO(K)$-embedding $\imath\colon K\langle a \rangle\to M$. Let $b_0\coloneqq \imath(a_0) \in M$ and for each $n$, let $b_{n+1}\coloneqq \log b_n$. By the proof of Corollary~\ref{cor:omegaconstruction}, the embedding $\imath$ extends to an $\LdO(K)$-embedding $K\langle a \rangle_\upo\to M$ which sends $a_n$ to $b_n$ for each $n$. By Proposition~\ref{prop:logclosedtest2}, this is even an $\LdO_{\log}$-embedding and, as an $\LdO_{\log}$-embedding, it is unique.
\end{proof}

\noindent
If $K$ is a logarithmic $H_T$-field with a gap, then we can instead use Lemma~\ref{lem:gapgoesdown} and apply Corollary~\ref{cor:omegaconstruction} with $|a|$ in place of $s$ to show the following:

\begin{lemma}\label{lem:loggapgoesdown}
Let $K$ be a logarithmic $H_T$-field, let $s \in K$, and suppose $vs$ is a gap in $K$. Then $K$ has a logarithmic $H_T$-field extension $K\langle a \rangle_\upo$ with $a\succ 1$ and $a' = s$ such that for any logarithmic pre-$H_T$-field extension $M$ of $K$ and $b \in M$ with $b \succ 1$ and $b' = s$, there is a unique $\LdO_{\log}(K)$-embedding $K\langle a \rangle \to M$ sending $a$ to $b$. The logarithmic $H_T$-field $K\langle a \rangle$ is $\upo$-free with $\res K\langle a \rangle_{\upo} = \res K$.
\end{lemma}

\noindent
The following proposition has the same proof as Theorem~\ref{thm:HThull}, except we use Lemma~\ref{lem:loggapgoesup} in place of Lemma~\ref{lem:gapgoesup}, and we use Corollary~\ref{cor:logsmallint} in place of Corollary~\ref{cor:smallint}.

\begin{proposition}\label{prop:logHThull}
$K$ has a logarithmic $H_T$-field extension $H_T^{\log}(K)$ with $\res H_T^{\log}(K) = \res K$ such that for any logarithmic $H_T$-field extension $M$ of $K$, there is a unique $\LdO_{\log}(K)$-embedding $H_T^{\log}(K)\to M$.
\end{proposition}

\subsection{Constant field extensions}
We end this section with a proposition on extending logarithmic $H_T$-fields by constants. 

\begin{proposition}
\label{prop:logconstext}
Let $K$ be a logarithmic $H_T$-field and let $E$ be a $T^{\e}$-extension of $C$. Then there is a logarithmic $H_T$-field extension $L$ of $K$ where $C_L$ is $\cL_{\log}(C)$-isomorphic to $E$ such that for any logarithmic $H_T$-field extension $M$ of $K$ and any $\cL_{\log}(C)$-embedding $\imath\colon C_L\to C_M$, there is a unique $\LdO_{\log}(K)$-embedding $L\to M$ extending $\imath$.
\end{proposition}
\begin{proof}
Let $L$ be the $H_T$-field extension of $K$ given by Proposition~\ref{prop:constext}, so $C_L$ is $\cL(C)$-isomorphic to $E$. By the proof of Proposition~\ref{prop:constext}, we also have $\Gamma_L = \Gamma$. Using the $\cL(C)$-isomorphism $C_L\to E$, we pull the logarithm on $E$ back to a logarithm on $C_L$, and we view $C_L$ as a model of $T^{\e}$ with this logarithm. We need to extend this logarithm on $C_L$ to a logarithm on $L$. Let $f \in L^>$ and take $g \in K^>$ with $f \asymp g$. Take $c \in C_L^>$ with $f/g \sim c$, and take $\epsilon \in \smallo_L$ with $f = cg(1+\epsilon)$. Set
\[
\log f\ \coloneqq \ \log g + \log c + \ln(1+\epsilon),
\]
where $\log g$ is evaluated in $K$ and $\log c$ is evaluated in $C_L$. It is routine to check that this assignment doesn't depend on the choice of $g$. Checking that this map is a logarithm is also quite straightforward; for convenience, we may assume that $g = 1$ in the case that $f \asymp 1$. For (LH1), we have
\[
\log(f)'\ =\ \log(g)' + \log(c)' + \ln(1+\epsilon)' \ =\ g^\dagger + c^\dagger + (1+\epsilon)^\dagger \ =\ f^\dagger,
\]
where the second equality uses that $\log(c)' = c^\dagger = 0$ for $c \in C_L^>$. Since for each $a \in \cO_L$, there is $d \in C_L$ with $|a-d|\leq 1$ and since the logarithm on $C_L$ is surjective, (LH2) follows from Lemma~\ref{lem:closeislog}. 

Now let $M$ be a logarithmic $H_T$-field extension of $K$ and let $\imath\colon C_L\to C_M$ be an $\cL_{\log}(C)$-embedding. Using the embedding property of $L$, we uniquely extend $\imath$ to an $\LdO(K)$-embedding $\jmath\colon L \to M$. We will show that $\jmath$ is even an $\LdO_{\log}(K)$-embedding. For $c \in C_L^>$ and $g \in K^>$, we have 
\[
\log \jmath(cg)\ =\ \log \imath(c)+\log g\ =\ \imath(\log c)+\log g\ =\ \jmath\big(\log (cg)\big).
\]
Let $f \in L^>$. Since we can find $c \in C_L^>$ and $g \in K^>$ with $f \sim cg$, Lemma~\ref{lem:logclosedtest} and the above computation allow us to conclude that $\log \jmath(f) = \jmath(\log f)$.
\end{proof}

\section{Liouville closed logarithmic $H_T$-fields}\label{sec:logLiouville}
\noindent
In this section, we keep the same assumptions as in the previous section ($T$ is power bounded, $T$ defines a restricted exponential function, $\Lambda$ is cofinal in the prime model of $T$). Let $K$ be a logarithmic $H_T$-field. As in Section~\ref{sec:Liouville}, we investigate Liouville closed extensions of $K$. Now that we have a logarithm present, we can relate exponential integrals to actual exponentials:

\begin{lemma}\label{lem:Liouvilleexponentials}
The following are equivalent:
\begin{enumerate}
\item[(i)] $\log(K^>) = K$ and every element in $K$ has an integral in $K$;
\item[(ii)] $K$ is Liouville closed.
\end{enumerate}
\end{lemma}
\begin{proof}
Suppose (i) holds and let $f \in K$. We need to find an exponential integral for $f$ in $K^{\neq}$. Take $g \in K$ with $g' = f$. Then $(\exp g)^\dagger = g' = f$. Now, suppose (ii) holds and let $a \in K$. We need to show that $a \in \log(K^>)$. Take $b \in K^>$ with $b^\dagger = a'$. Then $\log(b)' = a'$ so $a-\log b\in C$. Since $C = \log(C^>)$ by Lemma~\ref{lem:inducedlogC}, we may take $c \in C^>$ with $a-\log b= \log c$. Then $a = \log(bc)$. 
\end{proof}

\begin{definition}
A \textbf{logarithmic $T$-Liouville extension} of $K$ is a logarithmic $H_T$-field extension $L$ of $K$ which is also a $T$-Liouville extension of $K$, as defined in Section~\ref{sec:Liouville}. A \textbf{logarithmic $T$-Liouville closure} of $K$ is a logarithmic $T$-Liouville extension of $K$ which is Liouville closed.
\end{definition}

\noindent
Any logarithmic $T$-Liouville extension of $K$ has the same cardinality as $K$, and an increasing union of logarithmic $T$-Liouville extensions is itself a logarithmic $T$-Liouville extension. The logarithmic $H_T$-field extension $K^{\e}$ in Theorem~\ref{thm:expclosure} is a logarithmic $T$-Liouville extension, as are the extensions considered in Corollaries~\ref{cor:logsmallint} and~\ref{cor:logbigint} and in Lemmas~\ref{lem:loggapgoesup} and~\ref{lem:loggapgoesdown}.

\begin{definition}
A \textbf{logarithmic $T$-Liouville tower on $K$} is a strictly increasing chain $(K_\mu)_{\mu \leq \nu}$ of logarithmic $H_T$-fields such that:
\begin{enumerate}
\item $K_0 = K$;
\item if $\mu\leq \nu$ is an infinite limit ordinal, then $K_\mu = \bigcup_{\eta<\mu}K_\eta$;
\item if $\mu < \nu$ and $K_\mu\neq \log(K_\mu^>)$, then $K_{\mu+1} = K_\mu^{\e}$;
\item if $\mu<\nu$ and $K_\mu= \log(K_\mu^>)$, then $K_{\mu+1}$ is one of the following extensions of $K_\mu$:
\begin{enumerate}
\item $K_{\mu+1} = K_\mu\langle a_\mu\rangle_\upo$, where $a_\mu' = s_\mu \in K_\mu$ with $a_\mu \prec 1$ and $vs_\mu$ is a gap in $K_\mu$;
\item $K_{\mu+1} = K_\mu\langle a_\mu\rangle_\upo$, where $a_\mu' = s_\mu \in K_\mu$ with $a_\mu \succ 1$ and $vs_\mu$ is a gap in $K_\mu$;
\item $K_{\mu+1} = K_\mu\langle a_\mu\rangle$, where $a_\mu' = s_\mu \in K_\mu$ with $a_\mu \prec 1$, $vs_\mu \in (\Gamma^>_{K_\mu})'$, and $s_\mu \not\in \der \smallo_{K_\mu}$;
\item $K_{\mu+1} = K_\mu\langle a_\mu\rangle$, where $a_\mu' = s_\mu \in K_\mu$ with $v(s_\mu-\der K_\mu)\subseteq (\Gamma^<_{K_\mu})'$.
\end{enumerate}
\end{enumerate}
The logarithmic $H_T$-field $K_\nu$ is called the \textbf{top of the tower $(K_\mu)_{\mu \leq \nu}$}.
\end{definition}

\noindent
The extensions in (a) and (b) correspond to Lemmas~\ref{lem:loggapgoesup} and~\ref{lem:loggapgoesdown}, respectively, and the logarithm in these extensions is the logarithm defined in those Lemmas. The extensions in (c) and (d) correspond to Corollaries~\ref{cor:logsmallint} and~\ref{cor:logbigint}. If $(K_\mu)_{\mu \leq \nu}$ is a logarithmic $T$-Liouville tower on $K$, then each $K_\mu$ is a logarithmic $T$-Liouville extension of $K$. Maximal logarithmic $T$-Liouville towers on $K$ exist by Zorn's lemma, and we briefly verify below that the analogs of Lemmas~\ref{lem:toptclosure} and~\ref{lem:relativemaxismax} hold in this setting.

\begin{lemma}
\label{lem:logtoptclosure}
Let $L$ be the top of a maximal logarithmic $T$-Liouville tower on $K$. Then $L$ is Liouville closed and, therefore, $L$ is a logarithmic $T$-Liouville closure of $K$.
\end{lemma}
\begin{proof}
As in the proof of Lemma~\ref{lem:toptclosure}, maximality tells us each element in $L$ has an integral in $L$. Maximality also tells us that $L = \log(L^>)$, so $L$ is Liouville closed by Lemma~\ref{lem:Liouvilleexponentials}.
\end{proof}

\begin{lemma}
\label{lem:logrelativemaxismax}
Let $L$ be a Liouville closed logarithmic $H_T$-field extension of $K$ and let $(K_\mu)_{\mu\leq \nu}$ be a logarithmic $T$-Liouville tower on $K$. Suppose that $(K_\mu)_{\mu\leq \nu}$ is a tower in $L$ which can not be extended to a logarithmic $T$-Liouville tower $(K_\mu)_{\mu\leq \nu+1}$ in $L$. Then $(K_\mu)_{\mu\leq \nu}$ is a maximal logarithmic $T$-Liouville tower on $K$. 
\end{lemma}
\begin{proof}
By Fact~\ref{fact:Lexts} and Lemma~\ref{lem:closedimpliesmaximal}, it suffices to show that $K_\nu$ is Liouville closed. The proof that each element in $K_\nu$ has an integral in $K_\nu$ is essentially the same as the proof of Lemma~\ref{lem:relativemaxismax}. Since $L$ is Liouville closed, we have $L = \log(L^>)$ by Lemma~\ref{lem:Liouvilleexponentials}, so Theorem~\ref{thm:expclosure} allows us to identify $K_\nu^{\e}$ with a logarithmic $H_T$-subfield of $L$. Since $K_\nu^{\e}$ is a logarithmic $T$-Liouville extension of $K_\nu$, we have $K_\nu = K_\nu^{\e}$. Thus, $K_{\nu}$ is Liouville closed, again by Lemma~\ref{lem:Liouvilleexponentials}. 
\end{proof}

\noindent
We have the following analog of Corollary~\ref{cor:charofTLclosure}, where Lemmas~\ref{lem:logtoptclosure} and~\ref{lem:logrelativemaxismax} are used in place of Lemmas~\ref{lem:toptclosure} and~\ref{lem:relativemaxismax}.

\begin{corollary}
\label{cor:logcharofTLclosure}
Let $L$ be a Liouville closed logarithmic $H_T$-field extension of $K$. Then $L$ is a logarithmic $T$-Liouville closure of $K$ if and only if $L$ has no proper Liouville closed logarithmic $H_T$-subfields which contain $K$.
\end{corollary}

\noindent
An analog of Lemma~\ref{lem:nogap} also goes through with the obvious changes to the proof.

\begin{lemma}
\label{lem:lognogap}
Let $(K_\mu)_{\mu\leq \nu}$ be a logarithmic $T$-Liouville tower on $K$ and suppose $K_\mu$ does not have a gap for all $\mu< \nu$. Then $K_\nu$ admits an $\LdO_{\log}(K)$-embedding into any Liouville closed logarithmic $H_T$-field extension of $K$.
\end{lemma}

\subsection{Uniqueness and nonuniqueness of logarithmic $T$-Liouville closures}
\begin{proposition}\label{prop:loglambdafreeclosure}
Suppose $K$ is $\upl$-free. Then $K$ has a logarithmic $T$-Liouville closure $L$ which embeds over $K$ into any Liouville closed logarithmic $H_T$-field extension of $K$. Any logarithmic $T$-Liouville closure of $K$ is $\LdO_{\log}(K)$-isomorphic to $L$.
\end{proposition}
\begin{proof}
Let $(K_{\mu})_{\mu\leq \nu}$ be a maximal logarithmic $T$-Liouville tower on $K$. We will prove by induction on $\mu\leq \nu$ that each $K_\mu$ is $\upl$-free. This holds when $\mu = 0$ by assumption and if $\mu \leq \nu$ is an infinite limit ordinal, then this follows from Lemma~\ref{lem:unionupl}. Let $\mu < \nu$ and suppose that $K_\mu$ is $\upl$-free. If $K_\mu \neq \log(K_\mu^>)$, then $K_{\mu+1} = K_\mu^{\e}$ is $\upl$-free by Theorem~\ref{thm:expclosure}. Suppose that $K_\mu = \log(K_\mu^>)$. Since $K_\mu$ is $\upl$-free, it has no gap, so $K_{\mu+1}$ must be an extension of type (c) or (d). Then $K_{\mu+1}$ is $\upl$-free by Corollaries~\ref{cor:logsmallint} and~\ref{cor:logbigint}. 

Set $L\coloneqq K_\nu$, so $L$ is a logarithmic $T$-Liouville closure of $K$. Let $M$ be a Liouville closed logarithmic $H_T$-field extension of $K$. By Lemma~\ref{lem:lognogap}, there is an $\LdO_{\log}(K)$-embedding $\imath\colon L \to M$. Moreover, if $M$ is a logarithmic $T$-Liouville closure of $K$, then in particular, $M$ is a $T$-Liouville closure of $K$ and $\imath(L)$ is a Liouville closed $H_T$-subfield of $M$ containing $K$, so $\imath(L)=M$ by Corollary~\ref{cor:charofTLclosure1}. Thus, $L$ is unique up to $\LdO_{\log}(K)$-isomorphism.
\end{proof}

\noindent
If $K$ is not $\upl$-free, then $K$ has two distinct logarithmic $T$-Liouville closures. As in the case of $H_T$-fields, it is helpful to handle the case that $K$ has a gap and the case that $K$ has asymptotic integration separately.

\begin{proposition}
\label{prop:loggapclosures}
Let $\beta \in \Gamma$ be a gap in $K$. Then $K$ has logarithmic $T$-Liouville closures $L_1$ and $L_2$ with $\beta \in (\Gamma_{L_1}^>)'$ and $\beta \in (\Gamma_{L_2}^<)'$. Let $M$ be a Liouville closed logarithmic $H_T$-field extension of $K$. If $\beta \in (\Gamma_M^>)'$, then there is an $\LdO_{\log}(K)$-embedding $L_1\to M$. Likewise, if $\beta \in (\Gamma_M^<)'$, then there is an $\LdO_{\log}(K)$-embedding $L_2\to M$. Any logarithmic $T$-Liouville closure of $K$ is $\LdO_{\log}(K)$-isomorphic to either $L_1$ or $L_2$. 
\end{proposition}
\begin{proof}
Let $s \in K$ with $vs = \beta$. Let $K_1\coloneqq K\langle a \rangle_{\upo}$ be the logarithmic $H_T$-field extension of $K$ given in Lemma~\ref{lem:loggapgoesup}, so $a \prec 1$ and $a' = s$, and let $K_2\coloneqq K\langle b \rangle_{\upo}$ be the logarithmic $H_T$-field extension of $K$ given in Lemma~\ref{lem:loggapgoesdown}, so $b \succ 1$ and $b' = s$. Then $K_1$ is $\upl$-free, so it has a logarithmic $T$-Liouville closure $L_1$ which embeds over $K_1$ into any Liouville closed logarithmic $H_T$-field extension of $K_1$ by Proposition~\ref{prop:loglambdafreeclosure}. Likewise, $K_2$ has a logarithmic $T$-Liouville closure $L_2$ which embeds over $K_2$ into any Liouville closed logarithmic $H_T$-field extension of $K_2$. Now let $M$ be a Liouville closed logarithmic $H_T$-field extension of $K$. If $\beta \in (\Gamma_M^>)'$, then the embedding property in Lemma~\ref{lem:loggapgoesup} gives an $\LdO_{\log}(K)$-embedding $K_1\to M$, which in turn extends to an $\LdO_{\log}(K)$-embedding $L_1\to M$. If $\beta \in (\Gamma_M^<)'$, then using the embedding property in Lemma~\ref{lem:loggapgoesdown} instead, we get an $\LdO_{\log}(K)$-embedding $L_2\to M$. If $M$ is a logarithmic $T$-Liouville closure of $K$, then $M$ is $\LdO_{\log}(K)$-isomorphic to either $L_1$ or $L_2$ by Corollary~\ref{cor:charofTLclosure1}, since $M$ contains the $\LdO_{\log}(K)$-isomorphic image of either $L_1$ or $L_2$ as a Liouville closed logarithmic $H_T$-subfield.
\end{proof}

\begin{proposition}
\label{prop:loguplclosures}
Suppose that $K$ has asymptotic integration and is not $\upl$-free. Then $K$ has logarithmic $T$-Liouville closures $L_1$ and $L_2$ which are not $\LdO_{\log}(K)$-isomorphic. If $M$ is a Liouville closed logarithmic $H_T$-field extension of $K$, then there is an $\LdO_{\log}(K)$-embedding of either $L_1$ or $L_2$ into $M$. Any logarithmic $T$-Liouville closure of $K$ is $\LdO_{\log}(K)$-isomorphic to either $L_1$ or $L_2$. 
\end{proposition}
\begin{proof}
Let $(\upl_\rho)$ be a $\upl$-sequence in $K$ with pseudolimit $\upl \in K$. First, we consider the case that $-\upl$ has an integral $f \in K$. If there were $a \in K^>$ with $\log a = f$, then we would have $a^\dagger= f' = -\upl$, contradicting that $v\big(\upl+(K^\times)^\dagger\big)$ is a cofinal subset of $\Psi^\downarrow$ by Fact~\ref{fact:uplminusdaggers}. Thus, $f$ has no exponential in $K^>$, so we use Lemma~\ref{lem:addexp} to extend $K$ to a logarithmic $H_T$-field $K\langle a \rangle$ with $\log a = f$. Then $a^\dagger = -\upl$, so $va$ is a gap in $K\langle a \rangle$ by Lemma~\ref{lem:uplgapcreator}. By Proposition~\ref{prop:loggapclosures}, $K\langle a \rangle$ has logarithmic $T$-Liouville closures $L_1$ and $L_2$ with $va \in (\Gamma_{L_1}^>)'$ and $va \in (\Gamma_{L_2}^<)'$, one of which embeds over $K\langle a \rangle$ into any Liouville closed logarithmic $H_T$-field extension of $K\langle a \rangle$. Using Lemma~\ref{lem:Liouvilleexponentials} and the universal property of Lemma~\ref{lem:addexp}, we see that either $L_1$ or $L_2$ embeds over $K$ into any Liouville closed logarithmic $H_T$-field extension of $K$.

Now, consider the case that $-\upl$ has no integral in $K$. Since $K$ has asymptotic integration, we may use either Corollary~\ref{cor:logsmallint} or Corollary~\ref{cor:logbigint} to extend $K$ to an immediate logarithmic $H_T$-field $K\langle f\rangle$ where $f' = s$. Since $K\langle f\rangle$ is an immediate extension of $K$, the $\upl$-sequence $(\upl_\rho)$ remains a $\upl$-sequence in $K\langle f\rangle$ and $K\langle f \rangle$ has asymptotic integration, so by the previous case, $K\langle f \rangle$ has logarithmic $T$-Liouville closures $L_1$ and $L_2$ with $v(\exp f) \in (\Gamma_{L_1}^>)'$ and $v(\exp f) \in (\Gamma_{L_2}^<)'$, one of which embeds over $K\langle f \rangle$ into any Liouville closed logarithmic $H_T$-field extension of $K\langle f \rangle$. The embedding properties in Corollaries~\ref{cor:logsmallint} and~\ref{cor:logbigint} ensure either $L_1$ or $L_2$ embeds over $K$ into any Liouville closed logarithmic $H_T$-field extension of $K$.

The embedding properties of $L_1$ and $L_2$, together with Corollary~\ref{cor:charofTLclosure1}, ensure that any logarithmic $T$-Liouville closure of $K$ is $\LdO_{\log}(K)$-isomorphic to either $L_1$ or $L_2$. The proof that $L_1$ and $L_2$ are not isomorphic to each other is exactly the same as in the proof of Proposition~\ref{prop:uplclosures}.
\end{proof}

\noindent
Let us combine the above propositions into one theorem.

\begin{theorem}
\label{thm:logoneortwotlclosures}
If $K$ is $\upl$-free, then $K$ has exactly one logarithmic $T$-Liouville closure up to $\LdO_{\log}(K)$-isomorphism. Otherwise, $K$ has exactly two logarithmic $T$-Liouville closures up to $\LdO_{\log}(K)$-isomorphism. For any Liouville closed logarithmic $H_T$-field extension $M$ of $K$, there is an $\LdO_{\log}(K)$-embedding of some logarithmic $T$-Liouville closure of $K$ into $M$.
\end{theorem}

\subsection{An application to $\R_{\anexp}$-Hardy fields}
In this subsection, let $\cR$, $\cL_{\cR}$, and $T_{\cR}$ be as in the introduction, that is, $\cR$ is an o-minimal expansion of the real field, $T_{\cR}$ is the $\cL_{\cR}$-theory of $\cR$, $\cL_{\cR}$ is assumed to include a constant symbol for each $r \in \R$, and $T_{\cR}$ is assumed to have quantifier elimination and a universal axiomatization. Let $\cH$ be an $\cR$-Hardy field and let $[f]$ be a germ of a real-valued unary function at $+\infty$. Then $[f]$ is said to be \textbf{comparable to $\cH$} if for each $[g]\in \cH$, either $g(x)<f(x)$ eventually, or $g(x) > f(x)$ eventually, or $g(x) = f(x)$ eventually (where \emph{eventually} means \emph{for all sufficiently large $x$}). If $[f]$ is comparable to $\cH$, then set
\[
\cH\langle [f]\rangle\ \coloneqq \ \big\{t([f]): t\text{ is a unary $\cL_{\cR}(\cH)$-term}\big\}. 
\]
If, in addition to being comparable with $\cH$, the function $f$ is eventually $\cC^1$ and $[f'] \in \cH\langle [f]\rangle$, then $\cH\langle [f]\rangle$ is an $\cR$-Hardy field; see~\cite[Lemma 5.12]{DMM94} and the remarks at the end of~\cite{DMM94}. If $f' = g$ or $f = \exp(g)$ for some $[g] \in \cH$, then $[f]$ is comparable to $\cH$ by Boshernitzan~\cite[Theorem 5.3]{Bo81}; see also~\cite{Ro83}. In both cases, $f$ is eventually $\cC^1$ and $[f']\in \cH\langle [f]\rangle$, so it follows that
\begin{itemize}
\item $\cH\langle [\exp g]\rangle$ is an $\cR$-Hardy field for $[g]\in \cH$, and
\item $\cH\langle [f]\rangle$ is an $\cR$-Hardy field if $[f']\in \cH$. 
\end{itemize}

\medskip\noindent
Since any increasing union of $\cR$-Hardy fields is an $\cR$-Hardy field and since Hardy fields are bounded in size, Zorn's lemma and the remarks above give us a \emph{Liouville closed} $\cR$-Hardy field extension of $\cH$ where every germ has an integral and an exponential (thus, every germ also has a nonzero exponential integral). We denote by $\Li_{\cR}(\cH)$ the intersection of all Liouville closed $\cR$-Hardy field extensions of $\cH$. Then $\Li_{\cR}(\cH)$ is a $T_{\cR}$-Liouville closure of $\cH$ by Corollary~\ref{cor:charofTLclosure1}, since the constant field of any $\cR$-Hardy field is $\R$.

\medskip\noindent
Here is an application when $\cR = \R_{\anexp}$. The appropriate language here is $\cL_{\anexp}$, which includes a function symbol for each restricted analytic function, as well as function symbols for $\exp$ and $\log$. By~\cite[Corollary 4.6]{DMM94}, $\R_{\anexp}$ has quantifier elimination and a universal axiomatization in this language. Since each constant function is analytic, our assumptions at the beginning of the subsection hold for this expansion. For an $\R_{\anexp}$-Hardy field $\cH$, let us write $\Li_{\anexp}(\cH)$ instead of $\Li_{\R_{\anexp}}(\cH)$. Recall that $\T_{\anexp}$ is the canonical expansion of the field $\T$ of logarithmic-exponential transseries to an $H_{\anexp}$-field. The following theorem is an analog of a theorem on Hardy fields from~\cite{AD02}.

\begin{theorem}\label{thm:Ranexpembedding}
Let $\cH$ be an $\R_{\anexp}$-Hardy field and let $\imath\colon \cH \to \T_{\anexp}$ be an $\LdO_{\anexp}$-embedding. Then $\imath$ extends to an $\LdO_{\anexp}$-embedding $\Li_{\anexp}(\cH)\to \T_{\anexp}$.
\end{theorem}
\begin{proof}
If $\cH = \R_{\anexp}$, then we extend $\imath$ to an $\LdO_{\anexp}$-embedding of
\[
\cH(\R_{\anexp})\ \coloneqq \ \big\{[t]:t \text{ is a unary }\cL_{\anexp}(\emptyset)\text{-term}\big\}
\]
by sending $[t] \in \cH(\R_{\anexp})$ to $t(x) \in \T_{\anexp}$ where $x \in \T_{\anexp}$ is the distinguished positive infinite element with derivative $x' = 1$. One can easily verify that this is an $\LdO_{\anexp}$-embedding. Thus, by replacing $\cH$ by $\cH(\R_{\anexp})$ if need be, we may assume that $\cH$ is a proper extension of $\R_{\anexp}$. Let $K\coloneqq \imath(\cH) \subseteq \T_{\anexp}$ and let $(K_\mu)_{\mu\leq \nu}$ be a maximal logarithmic $T_{\an}$-Liouville tower on $K$ in $\T_{\anexp}$. Lemma~\ref{lem:logrelativemaxismax} tells us that $K_{\nu}$ is a logarithmic $T_{\an}$-Liouville closure of $K$. Moreover, none of the $H_T$-fields $K_\mu$ have a gap by~\cite[Lemma 6.6]{AD02}, so Lemma~\ref{lem:lognogap} gives an $\LdO_{\anexp}$-embedding $\jmath\colon K_\nu\to \Li_{\anexp}(\cH)$ which extends $\imath\inv$. Since $\Li_{\anexp}(\cH)$ is a logarithmic $T_{\an}$-Liouville closure of $\cH$, we have $\jmath(K_\nu) = \Li_{\anexp}(\cH)$ by Corollary~\ref{cor:charofTLclosure1}. Thus, we may take $\jmath\inv\colon \Li_{\anexp}(\cH)\to \T_{\anexp}$ to be our $\LdO_{\anexp}$-embedding.
\end{proof}

\begin{remark}
Using $T_{\rexp}$, the theory of the reals with the restricted exponential function, in place of $T_{\an}$ in the proof above, one can also show that any $\LdO_{\exp}$-embedding of an $\R_{\exp}$-Hardy field $\cH\supseteq\R$ into $\T_{\exp}$ can be extended to $\Li_{\exp}(\cH)$. Here, $\T_{\exp}$ is the expansion of $\T$ by just the exponential function.

Using Lemmas~\ref{lem:relativemaxismax} and~\ref{lem:nogap} in place of Lemmas~\ref{lem:logrelativemaxismax} and~\ref{lem:lognogap} in the proof above, one can show that any $\LdO_{\an}$-embedding of an $\R_{\an}$-Hardy field $\cH$ into $\T_{\an}$ can be extended to $\Li_{\an}(\cH)$. The same holds for any polynomially bounded reduct $\cR$ of $\R_{\anexp}$, so long as $\cR$ satisfies the assumptions at the beginning of this subsection.
\end{remark}

\section{The order 1 intermediate value property}\label{sec:IVP}
\noindent
In this section, let $K$ be a pre-$H_T$-field. We do not assume that $T$ is power bounded. Our goal is to prove the following extension result: 

\begin{theorem}\label{thm:IVP}
$K$ has a pre-$H_T$-field extension $M$ with the following property: for every $\cL(M)$-definable continuous function $F\colon M\to M$ and every $b_1,b_2 \in M$ with 
\[
b_1'\ <\ F(b_1),\qquad b_2'\ >\ F(b_2),
\]
there is $a \in M$ between $b_1$ and $b_2$ with $a' = F(a)$. 
\end{theorem}

\noindent
Before proving this theorem, we need a lemma about extensions of pre-$H_T$-fields:

\begin{lemma}\label{lem:phIVPext}
Suppose $K$ is ungrounded, let $M$ be a $\TdO$-extension of $K$, and suppose that
\[
\cO_M\ = \ \big\{y\in M:|y|<d\text{ for all }d \in K\text{ with } d >\cO\big\},
\]
and that $g'>0$ for all $g \in M$ with $g>\cO_M$. Then $M$ is a pre-$H_T$-field extension of $K$. 
\end{lemma}
\begin{proof}
Let $g \in M$ with $g \succ 1$. Then $|g|' >0$ by assumption, so $g^\dagger = |g|^\dagger > 0$, proving (PH1). For (PH2), let $f \in \cO_M$. We need to show that $f'\prec g^\dagger$. We will do this by showing that $v(f')>\Psi$ and that $vg^\dagger \in \Psi^\downarrow$. To see that $vf' > \Psi,$ let $\gamma \in \Psi$ and, using that $K$ is ungrounded, take $d \in K$ with $d>\cO$ and $v d'> \gamma$. Then $d+f,d-f>\cO_M$, so $d'+f',d'-f'>0$. This gives $-d'<f'<d'$, so $vf'\geq vd'> \gamma$. To see that $vg^\dagger \in \Psi^\downarrow$, take $d \in K$ with $|g|>d>\cO$. Then $|g|d^{-1/2}>d^{1/2}>\cO$ so $\big(|g|d^{-1/2}\big)^\dagger = g^\dagger - d^\dagger/2> 0$, which gives $g^\dagger > d^\dagger / 2$. As $d^\dagger > 0$, we have $vg^\dagger \leq vd^\dagger \in \Psi$, so $vg^\dagger \in \Psi^\downarrow$.
\end{proof}

\noindent
Theorem~\ref{thm:IVP} follows by iterating the following proposition; an analog of~\cite[Theorem 4.3]{AD05}. See also~\cite{vdD00}, where van den Dries proves this for $\cR$-Hardy fields. 

\begin{proposition}\label{prop:IVPstep}
Let $F\colon K\to K$ be an $\cL(K)$-definable continuous function and let $b_1,b_2 \in K$ with 
\[
b_1'\ <\ F(b_1),\qquad b_2'\ >\ F(b_2).
\]
Then there is a pre-$H_T$-field extension $M$ of $K$ and $a \in M$ between $b_1$ and $b_2$ with $a' = F(a)$. 
\end{proposition}
\begin{proof}
If $\cO = K$, then let $M \coloneqq K\langle a \rangle$ be a simple $T$-extension of $K$ where $a$ lies between $b_1$ and $b_2$. Using Fact~\ref{fact:transext}, extend the $T$-derivation on $K$ uniquely to a $T$-derivation on $M$ by so that $a' = F(a)$. Then $M$, with this derivation and the $T$-convex valuation ring $\cO_M\coloneqq M$, is the desired pre-$H_T$-field extension of $K$. Having handled this case, we assume for the remainder of the proof that $\cO \neq K$ (so $\der$ is continuous by Lemma~\ref{lem:ctsderiv}).

Next, we arrange that $K$ is ungrounded. If $T$ is not power bounded, then $T$ defines an exponential function~\cite{Mi96}, so $K$ is necessarily ungrounded by Corollary~\ref{cor:expungrounded}. If $T$ is power bounded and $K$ is grounded, then we use Corollary~\ref{cor:omegaconstruction} to replace $K$ with an ungrounded pre-$H_T$-field extension. 

Now, let us handle the case that $b_1<b_2$. Let $I$ be the interval $(b_1,b_2)$ and set
\[
A\ \coloneqq \ \big\{y \in I: y' < F(y)\big\}.
\]
Since $b_1'< F(b_1)$ and since $F$ and $\der$ are continuous, we have $y'<F(y)$ for all $y \in I$ sufficiently close to $b_1$. Thus, $A$ is nonempty. Likewise, $y'> F(y)$ for all $y \in I$ sufficiently close to $b_2$, so $A$ is not cofinal in $I$. If $A$ has a supremum $b \in I$, then $b' = F(b)$ by continuity, and we may take $M=K$. Thus, we may assume that $A$ has no supremum in $I$. Let $M\coloneqq K\langle a \rangle$ be a simple $T$-extension of $K$ where $a$ realizes the cut $A^\downarrow$. Using Fact~\ref{fact:transext}, we extend the $T$-derivation on $K$ uniquely to a $T$-derivation on $M$ by with $a' = F(a)$. We also equip $M$ with the $T$-convex valuation ring
\[
\cO_M\ \coloneqq \ \big\{y\in M:|y|<d\text{ for all }d \in K\text{ with } d >\cO\big\}.
\]
We claim that $M$ is a pre-$H_T$-field extension of $K$. By Lemma~\ref{lem:phIVPext}, it is enough to show that $g' >0$ for all $g \in M$ with $g>\cO_M$. Let $G\colon K\to K$ be an $\cL(K)$-definable function with $G(a)>\cO_M$. We may assume $G(a) \not\in K$. Suppose toward contradiction that 
\[
G(a)' \ = \ G^{[\der]}(a)+G'(a)F(a)\ \leq \ 0.
\]
Take $d \in K$ with $G(a)>d$ and take a subinterval $J \subseteq I$ with $a \in J^M$ such that $G$ is $\cC^1$ on $J$ and such that for all $y \in J$, we have
\[
G(y)\ > \ d,\qquad G^{[\der]}(y)+G'(y)F(y)\ \leq \ 0.
\]
Let $y \in J$. Since $G(y)>d>\cO$ and $K$ is a pre-$H_T$-field, we have $G(y)' = G^{[\der]}(y)+G'(y)y'>0$, so 
\[
\big(G^{[\der]}(y)+G'(y)y'\big)-\big(G^{[\der]}(y)+G'(y)F(y)\big) \ =\ G'(y)\big(y'-F(y)\big)\ > \ 0.
\]
By shrinking $J$, we may assume that $G'(y)$ has constant sign on $J$, so $y'-F(y)$ has constant sign on $J$ as well. This is a contradiction, as $J$ contains elements of $A$ as well as elements of $I\setminus A$. The case that $b_1>b_2$ is virtually identical; we instead let $I$ be the interval $(b_2,b_1)$ and let
\[
A\ \coloneqq \ \big\{y \in I: y' >F(y)\big\}.\qedhere
\]
\end{proof}

\begin{remark}
If $T$ is power bounded, then in Proposition~\ref{prop:IVPstep}, we can take $M$ to be an $H_T$-field (simply use Theorem~\ref{thm:HThull} to take the $H_T$-field hull of $M$). This can be used to strengthen Theorem~\ref{thm:IVP} accordingly. We can also take $M$ to be an $H_T$-field if $T$ is an exponential theory of the form considered in Subsection~\ref{subsec:expsandlogs} (for example, if $T = T_{\anexp}$). For this, we use Proposition~\ref{prop:logHThull} to pass to the logarithmic $H_T$-field hull, followed by Theorem~\ref{thm:expclosure} to pass to the exponential closure.
\end{remark}


\end{document}